\documentclass[11pt,reqno]{amsproc}

\title{Suppression of the Benjamin--Feir instability by constant vorticity}

\author{Bowen Chen}
\address{Department of Mathematics, University of Maryland, College Park, MD 20742}
\email{bchen0@umd.edu} 

\author{Huy Q. Nguyen}
\address{Department of Mathematics, University of Maryland, College Park, MD 20742}
\email{hnguye90@umd.edu}



\usepackage{amsmath}
\usepackage{amsthm}
\usepackage{amssymb}
\usepackage{booktabs}
\usepackage{mathrsfs}
\usepackage[margin=1in]{geometry}
\usepackage{times}
\usepackage{setspace}
\usepackage{graphicx}
\usepackage{listings}
\usepackage{float}
\usepackage{tabularx}
\usepackage{stmaryrd}
\usepackage{subfigure}
\usepackage{listings}
\usepackage{verbatim} 
\usepackage{multirow}
\usepackage[table]{xcolor}
\definecolor{g}{rgb}{0,0.5,0}%

\usepackage{tikz}
\usetikzlibrary{decorations.pathmorphing}
\usepackage{mathtools}
\usepackage{bigstrut}
\usepackage{esint}
\usepackage{extarrows}
\usepackage{newtxtt}
\usepackage{times}
\usepackage{color}
\usepackage{tikz-cd}
\usepackage{hyperref}

\theoremstyle{plain}
\newtheorem{theo}{Theorem}[section]
\newtheorem{prop}[theo]{Proposition}
\newtheorem{lemm}[theo]{Lemma}

\newtheorem{defi}[theo]{Definition}
\theoremstyle{definition}
\newtheorem{rema}[theo]{Remark} 

\numberwithin{equation}{section}
\usepackage{tcolorbox}
\tcbset{
  colframe=black!75,
  colback=gray!10,
  coltitle=black,
  fonttitle=\bfseries,
  sharp corners,
  boxrule=0.5pt,
}
\newtcolorbox{mmain}[1][]{%
  colback=white,
  colframe=red!20,
  fonttitle=\bfseries,
  title={Input:},
  sharp corners,
  boxrule=0.5pt,
  #1
}

\newtcolorbox{mmaout}[1][]{%
  colback=gray!5,
  colframe=cyan!20,
  fonttitle=\bfseries,
  title={Output:},
  sharp corners,
  boxrule=0.5pt,
  #1
}

\newtcolorbox{mmacode}[1][]{%
  colback=white,
  colframe=green!20,
  fonttitle=\bfseries,
  title={Code:},
  sharp corners,
  boxrule=0.5pt,
  #1
}

\DeclareMathOperator{\Span}{span}
\newcommand{\bq}{\begin{equation}}
\newcommand{\eq}{\end{equation}}

\newcommand{\R}{\mathbb{R}}
\newcommand{\F}{\mathcal{F}}
\def\L{\mathcal{L}}
\def\Ls{\mathscr{L}}
\def\sign{\mathrm{sgn}}
\def\Z{\mathbb{Z}}
\def\C{\mathbb{C}}
\def\pinv{\p_x^{-1}}
\def\rev{\mathfrak{r}}
\def\Ush{U^\sharp}
\def\B{\mathcal{B}}
\def\Lmat{\mathtt{L}}
\def\fD{\mathfrak{D}}
\def\signp{\sign^+}
\def\errorTol{O(\eps^3,\nu^2\eps^2,\nu^4\eps,\nu^6)}

\def\res{\operatorname{Res}_{\lambda=0}}
\def\bracketNu#1{^{[#1]}}
\def\anaR#1{\R\{#1\}}
\def\anaC#1{\C\{#1\}}
\def\sC{\mathscr{C}}
\def\sE{\mathscr{E}}

\def\critmu{\underline\mu}
\def\condMuEps#1#2{\eps\in(-\eps_#2,\eps_#2),\mu\in(0,\mu_#1)}
\newcommand{\vet}[2]{\begin{bmatrix}#1 \\ #2 \end{bmatrix}}
\def\velo{\mathbf{u}}
\def\newRev{\mathcal{R}}
\def\rL{\mathring{L}}

\DeclareMathOperator{\RE}{Re}

\DeclareMathOperator{\IM}{Im}

\DeclareMathOperator{\I}{I}

\def\eps{\varepsilon}
\def\na{\nabla}

\def\wh{\widehat}

\def\mez{\frac{1}{2}}

\def\Rr{\mathbb{R}}
  
\def\Nn{\mathbb{N}}
\def\Zz{\mathbb{Z}}
\def\Cc{\mathbb{C}}

\def\cK{\mathcal{K}}

\def\cT{\mathcal{T}}
\def\cL{\mathcal{L}}

\def\cH{\mathcal{H}}
\def\cV{\mathcal{V}}
\def\ld{\lambda}

\def\p{\partial}

\def\na{\nabla}

\def\ol{\overline}
\def\T{\mathbb{T}}

\def\wt{\widetilde}

\def\Om{\Omega}
\def\ld{\lambda}

\newcommand{\rem}[1]{O(\eps^{#1})}

\def\L{\mathcal{L}}

\renewcommand{\th}{\mathtt{t}_h}

\begin{document}
\maketitle

\begin{abstract}
We study the effect of constant vorticity on the Benjamin--Feir instability of small-amplitude Stokes waves in infinite depth. Working with the full water-wave equations, we rigorously establish a transition between modulational instability and stability as the vorticity crosses a critical value. We consider right-propagating waves with constant vorticity $\gamma$, associated with the background shear current $-\gamma y\mathbf e_x$, and normalize the gravitational acceleration and carrier wavenumber to one. For each fixed $\gamma<2/\sqrt{3}$, we prove that sufficiently small-amplitude Stokes waves are Benjamin--Feir unstable. For each fixed $\gamma>2/\sqrt{3}$, these waves are spectrally stable with respect to longitudinal, long-wavelength perturbations. These results rigorously confirm, away from the critical value, the stability transition predicted by the formal nonlinear Schrödinger approximation of Thomas, Kharif, and Manna.
\end{abstract}

\tableofcontents

% We define the transform in Wahlen's paper
\def\WahlenMat{
\begin{bmatrix}
1 & 0\\
\frac{\gamma}{2}\pinv & 1
\end{bmatrix}
}
\def\WahlenMatPi{
\begin{bmatrix}
1 & 0\\
\frac{\gamma}{2}\pinv\Pi_0^\perp & 1
\end{bmatrix}
}

\section{Introduction}
This work concerns the stability of periodic traveling wave solutions to the two-dimensional water wave problem. For irrotational flows, such waves were first constructed by Stokes \cite{Stokes1847} in 1847 and later  rigorously justified by Nekrasov \cite{Nekrasov}, Levi-Civita \cite{Levi_Civita}, and Struik \cite{Struik1926}. It was long believed since then that small-amplitude Stokes waves are stable until the discovery of Benjamin and Feir \cite{BenjaminFeir1967} in 1967, who, through experiments and formal analysis found that Stokes waves are unstable with the respect to {\it longitudinal, long-wave perturbations}. The first rigorous proof of this {\it Benjamin--Feir (modulational) instability} was found for the finite-depth case by Bridges and Mielke \cite{finite_depth_BF_instability_1995} in 1995. The infinite depth case was recently resolved by Nguyen and Strauss \cite{Huy21}. After that, in the series of works \cite{Berti_Full, Berti_finite_depth, critical_depth_unstable}, Berti, Maspero, and Ventura rigorously justified the complete Benjamin--Feir spectrum which forms a {\it figure-8 curve} near the origin of the complex plane, in both the finite and infinite depth cases. We also refer to \cite{Yang} and \cite{Hur} for alternative proofs of the Benjamin--Feir instability. On the other hand, there is a vast literature on numerical and formal analysis for the stability of Stokes waves with respect to various types of perturbations, e.g. \cite{Bryant, McLean1, McLean2, DecOli, OliverasDeconinck2015, Creedon_high_freq}. The methods and techniques developed in the aforementioned rigorous works have led to rigorous proof of {\it high-frequency instabilities} \cite{Hur, BertiMasperoVentura2025, Berti_isola}  and {\it transverse instabilities} \cite{Huy_transverse, Huy_transverse_finite_depth, Transverse-RSY, Berti-McLean} of irrotational gravity Stokes waves. Regarding irrotational capillary-gravity Stokes waves, we refer to \cite{HurYang-capillary, HsiaoMaspero2025, HsiaoWang2026} for Benjamin--Feir and high-frequency instabilities, and \cite{Truong_transverse} for transverse instabilities. 

These advances provide a detailed understanding of the instability of Stokes waves in the irrotational setting. In many physical situations, however, surface waves propagate on currents whose horizontal velocity varies with depth, resulting in nonzero vorticity. Such shear currents occur in coastal and tidal flows and in wind-driven surface layers \cite{Lentz1992, StaceyMonismithBurau1999}. Unlike a uniform current, which can be removed by a change of reference frame, vertical shear modifies the intrinsic dispersion and nonlinear interactions of surface waves. Understanding its effect on the instability of Stokes waves is therefore a natural extension of the irrotational theory. A fundamental case is that of constant vorticity, corresponding to a background current with a linear velocity profile, ${\bf u_s}=-\gamma y\mathbf e_x$, where $\gamma\in \Rr$ is the induced constant vorticity. This provides a tractable model of wave-current interaction while retaining the effects of vertical shear \cite{ConstantinKalimerisScherzer2015}. Moreover, constant vorticity is preserved by the two-dimensional incompressible Euler evolution, making the assumption consistent with the fully nonlinear dynamics \cite{ConstantinIvanovProdanov2008, Wahlen_Hamiltonian, BertiFranzoiMaspero2021}. 

In this work, we investigate the stabilizing/destabilizing effect of constant vorticity on the Benjamin--Feir (modulational) instability of periodic traveling waves. In 2012,  Thomas, Kharif and Manna \cite{ThomasKharifManna2012} approached this problem using a formal multiple-scale expansion. They formally derived a nonlinear Schr\"odinger (NLS) equation governing the slowly varying wave envelope, accounting for the coupling between the background vorticity and the mean flow induced by modulation. Their NLS has the form 
\bq\label{NLS}
ia_\tau+La_{\xi\xi}=M|a|^2a,
\eq
where the constants $L\in \Rr$ and $M\in \Rr$  depend on the constant vorticity $\gamma\in \Rr$ and the fluid's depth $h\in (0, \infty]$. \eqref{NLS} admits plane-wave solutions $a_*= a_0\exp(-iMa_0^2\tau)$, $a_0\in \Rr$. The authors then studied the stability of $a_*$ under long-wave perturbations and found that $a_*$ is stable if \eqref{NLS} is defocusing (i.e., $LM>0$) and unstable  if \eqref{NLS} is focusing (i.e., $LM<0$). In particular, in infinite depth, the plane-wave solution $a_*$ undergoes a transition from modulational instability to stability as $\gamma$ increases through the {\it critical value} 
\bq\label{gammacri}
\gamma_{cri}=2\sqrt\frac{gk}{3},
\eq
 where $g>0$ is the gravitational acceleration and $k$ is the carrier wavenumber. More precisely, $a_*$ is modulationally stable when  $\gamma>\gamma_{cri}$ and modulationally unstable when  $\gamma<\gamma_{cri}$. The same critical value $\gamma_{cri}$ was also predicted in \cite{GuyenneKairzhanSulem2022}  using a Hamiltonian Dysthe equation, a higher-order nonlinear Schr\"odinger envelope model.

In this paper, we rigorously establish the vorticity-induced transition \eqref{gammacri} predicted by Thomas, Kharif and Manna for the full water-wave system in infinite depth. For each fixed constant vorticity $\gamma<\gamma_{cri}$, we prove that sufficiently small-amplitude Stokes waves exhibit Benjamin--Feir instability and fully describe its Benjamin--Feir spectrum. For $\gamma>\gamma_{cri}$, we prove spectral stability with respect to longitudinal, long-wavelength perturbations. Therefore, positive constant vorticity of magnitude larger than $\gamma_{cri}$  suppresses the Benjamin--Feir instability.

In order to state the result, we fix a constant vorticity $\gamma$ and a corresponding  Stokes wave of small amplitude $\eps$ and wave number $k$. We  perturb the Stokes wave by  longitudinal, long-wave perturbations  of Bloch form \(e^{i\mu x}f(x)\), where $f$ is $2\pi$-periodic and $\mu$ is small.  After linearizing the full water wave system around the Stokes wave and performing a number of transformations, we arrive at a linear Hamiltonian operator $\L_{\sqrt\mu,\eps}: H^1(\T)^2\to L^2(\T)^2$, $\T:=\Rr/2\pi \Zz$.  We refer to  \eqref{def:analytic_r} for the notation $r(\mu^{m_1}\eps^{n_1},\dots,\mu^{m_p}\eps^{n_p})$ for analytic remainder functions. We normalize $k=1$ and $g=1$, so that $\gamma_{cri}=\frac{2}{\sqrt{3}}$. Our main result is the following complete description of the spectrum of $\L_{\sqrt\mu,\eps}$ near the origin in the complex plane.
\begin{theo}
    \label{thm:main}
    Let  $\gamma\in \Rr\setminus\{0\}$. There exist $\eps_0,\mu_0>0$, an open set $U\subset \C$ containing 0, and real-valued analytic remainder  functions $r_\pm(\eps)$, $r_j(\mu^2,\mu\eps,\eps^2)$, $j\in\{0, 1\}$,  $r_2(\mu^3,\mu^2\eps,\mu\eps^2,\eps^3)$, $r(\eps^3)$, all depending on $\gamma$, such that  the following holds.  For all $\eps\in(-\eps_0,\eps_0),\mu\in(0,\mu_0)$, the spectrum of the operator $\L_{\sqrt\mu,\eps}$ in $U$  consists of  three eigenvalues $\lambda_*,\lambda_1,\lambda_2$, where 
   \bq
  \ld_*= \frac{i c_0^3}{c_0^2-1}\mu-\frac{ic_0^3}{(c_0^2-1)^4}\mu^2+i\mu r_0(\mu^2,\mu\eps,\eps^2)\in i\Rr,
   \eq
   $c_0>0$ being the unperturbed wave speed given by \eqref{eq:relation_c0_gamma}. Denote 
 \[
   \Delta=-\frac{4 c_0^6}{\left(c_0^2+1\right)^6}\mu^2+\frac{-c_0^6+3 c_0^4-c_0^2+3}{\left(c_0^2+1\right)^3}\eps^2+r_2(\mu^3,\mu^2\eps,\mu\eps^2,\eps^3). 
   \]
(i) (Modulational stability) If  $\gamma>\frac{2}{\sqrt{3}}$, then $\Delta<0$ and 
    \begin{equation}
        \lambda_{1,2}=\mez \left(\frac{2 i c_0^3}{c_0^2+1}\mu-i\mu r_1(\mu^2,\mu\eps,\eps^2)\pm i\mu\sqrt{-\Delta} \right)\in i\Rr. 
    \end{equation}   
(ii) (Modulational instability) If  $\gamma<\frac{2}{\sqrt{3}}$, then there exists a piecewise analytic function $\critmu:(-\eps_0,\eps_0)\to [0,\mu_0)$ depending on $\gamma$, of the form
    $$
    \critmu(\eps)=
    \begin{cases}
        \sqrt{\frac{(c_0^2+1)^3(-c_0^6+3 c_0^4-c_0^2+3)}{4c_0^6}}\eps+r_+(\eps^2) &\quad\text{if~} \eps\ge 0,\\
        -\sqrt{\frac{(c_0^2+1)^3(-c_0^6+3 c_0^4-c_0^2+3)}{4c_0^6}}\eps+r_-(\eps^2) & \quad\text{if~} \eps< 0,
    \end{cases}
    $$ 
    such that for any $\eps\in(-\eps_0,\eps_0)$, we have 
    \bq
    \Delta>0~\text{if}~0<\mu<\critmu(\eps),\quad\Delta<0~\text{if}~\critmu(\eps)<\mu<\mu_0,
    \eq
    and
    \begin{equation}
        \lambda_{1,2}=
        \begin{cases}
\mez \left(\frac{2 i c_0^3}{c_0^2+1}\mu-i\mu r_1(\mu^2,\mu\eps,\eps^2)\pm \mu\sqrt{\Delta} \right),& 0<\mu<\critmu(\eps),\\
\mez\left(\frac{2 i c_0^3}{c_0^2+1}\mu-ir(\eps^3) \right),&\mu=\critmu(\eps),\\
        \mez \left(\frac{2 i c_0^3}{c_0^2+1}\mu-i\mu r_1(\mu^2,\mu\eps,\eps^2)\pm i\mu\sqrt{-\Delta} \right),&\critmu(\eps)<\mu<\mu_0.
    \end{cases}
    \end{equation}
\end{theo}
Theorem \ref{thm:main} provides the first rigorous full-Euler demonstration, in infinite depth, that constant-vorticity shear can suppress the Benjamin--Feir instability of small-amplitude Stokes waves. Although the constant-vorticity problem contains the irrotational case \(\gamma=0\), Theorem \ref{thm:main} is not simply a parameter-dependent extension of the irrotational instability result. The vorticity parameter changes the qualitative spectral picture.

 In the unstable regime, namely $\gamma<\frac{2}{\sqrt{3}}$, small-amplitude Stokes waves are subject to Benjamin--Feir instability. In the irrotational case, this instability  was proven in \cite{Huy21} and the full Benjamin--Feir spectrum, including  the figure-$8$ curve, was justified in \cite{Berti_Full}.  It follows from part (ii) of Theorem \ref{thm:main} and \eqref{form:figure8} that for any $\gamma<\frac{2}{\sqrt{3}}$, the Benjamin--Feir spectrum forms a figure-$8$ curve approximated by the formula \eqref{form:figure8}.  Below are plots of approximate  figure-$8$ curves corresponding to various values of $\gamma<\frac{2}{\sqrt{3}}\approx 1.155$, in the rescaled coordinates $(\frac{\RE \ld}{\eps^2}, \frac{\IM \ld}{\eps})$:
\begin{figure}[H]
    \centering
    \includegraphics[width=0.75\linewidth, height=8cm]{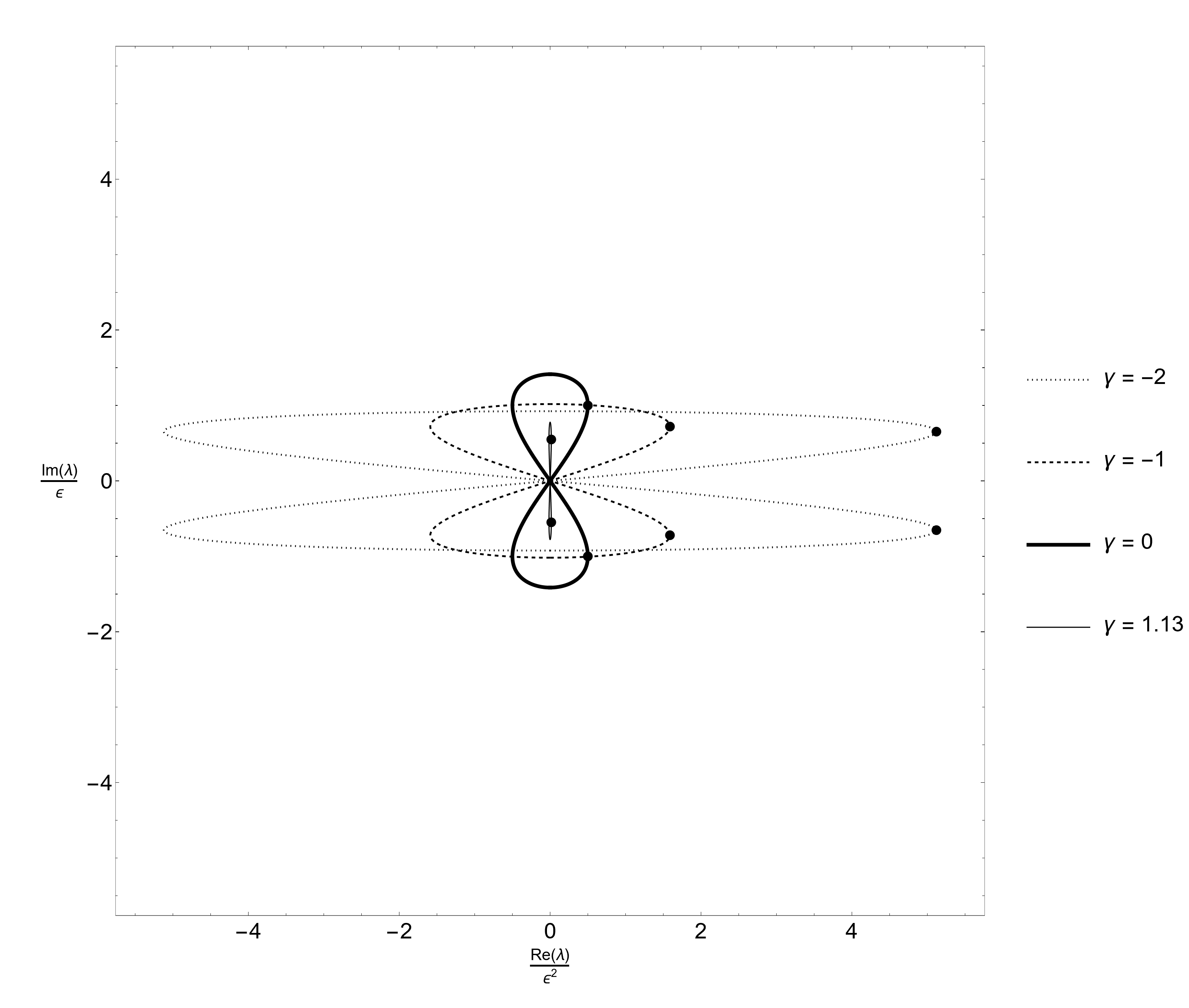}
    \caption{Approximate figure  ``8'' for $\gamma=-2, -1, 0, 1.13$}
\end{figure}
Several remarks about Theorem \ref{thm:main} are in order:
\begin{rema}
Irrotational Stokes waves are also susceptible to high-frequency instabilities arising from collisions of nonzero imaginary eigenvalues, in contrast to the Benjamin--Feir instability, which develops near the origin in $\Cc$.  See \cite{DecOli, Creedon_high_freq, Hur, BertiMasperoVentura2025, Berti_isola}. Theorem \ref{thm:main} describes the spectrum only in a neighborhood of the origin in $\Cc$ and therefore does not rule out high-frequency instabilities, even when $\gamma>\frac{2}{\sqrt{3}}$.
\end{rema}
\begin{rema}
For the proof of Theorem \ref{thm:main}, we expand the matrix representing the linearized operator in a suitable finite-dimensional subspace up to the orders $\eps^2$ and $\nu^4$, $\nu=\sqrt{\mu}$. See Proposition \ref{prop:finalmatrix}. We expect that for critical vorticity $\gamma=\frac{2}{\sqrt{3}}$, higher-order expansions would reveal the stability/instability. This is reminiscent of the Benjamin--Feir instability at critical depth  \cite{critical_depth_unstable}.
\end{rema}
\begin{rema}
We show in \eqref{maxReld} that for each fixed $\gamma\in (-\infty, \frac{2}{\sqrt{3}})$, upon taking $\eps_0=\eps_0(\gamma)$ and $\mu_0=\mu_0(\gamma)$ smaller  if necessary, the maximum growth rate is 
\bq
   \max_{\mu\in (0, \mu_0)}\RE \ld=\frac{(-c_0^6+3 c_0^4-c_0^2+3)\eps^2}{8c_0^3}+O(\eps^3).
\eq
\end{rema}
\begin{rema}
The approach in the present paper will be adapted in a follow-up work to investigate the effect of constant vorticity on the Benjamin--Feir instability of Stokes waves in finite depth.  
\end{rema}
\subsection{Ideas of proof}
Our starting point is the generalization \eqref{eq:main_with_t} of the Zakharov-Craig-Sulem (ZCS) formulation for water waves with constant vorticity. This formulation is a coupled system of the free boundary $\eta(x, t)$ and the trace $\psi(x, t)$ of the potential part of the fluid velocity. Vorticity induces the transport terms $\gamma \eta\eta_x$, $\gamma \eta\psi_x$ and the nonlocal term $\gamma \p_x^{-1}G(\eta)\psi$, where $G(\eta)$ denotes the Dirichlet-Neumann operator for the fluid domain. The existence of  small-amplitude periodic traveling wave solutions (Stokes waves) has been proven in \cite{2025}. Our analysis requires expansion of these Stokes waves up to second order of their small amplitude $\eps$, which is carried out in Appendix \ref{appendix}. Following \cite{Huy21}, we consider $L=2\pi n$ periodic perturbations, linearize the ZCS system around a fixed Stoke wave $(\eta^*, \psi^*)$, and use Alinhac's good unknown to obtain the linear operator $\Ls_\eps$ \eqref{eq:v_matrix}. Moreover, we prove that for sufficiently small $\eps$, $0$ is the only eigenvalue of $\Ls_\eps$ in a neighborhood of $0\in \Cc$ and it has algebraic multiplicity $4$ with explicit formulas for the associated (generalized) eigenvectors. However, unlike the irrotational setting \cite{Huy21}, $\Ls_\eps$ is not in the canonical Hamiltonian form due to additional vorticity-induced terms, including the nonlocal term $\gamma \p_x^{-1}G(\eta^*)$. In \cite{Wahlen_Hamiltonian}, Wahl\'en observed that there is an explicit linear change of variables that transforms the ZCS system with constant vorticity to a canonical Hamiltonian system. Therefore, we return to the  ZCS system, use  Wahl\'en's change of variables followed by the good unknown transformation and the conformal mapping transformation to flatten $\{y=\eta^*\}$ and arrive at  the linear Hamiltonian system \eqref{sys:W} for $(w_1, w_2)$. For long-wave perturbations, we consider the ansatz $w_j=e^{\ld t}e^{i\mu x}q_j(x)$, where the $q_j$'s are $2\pi$-periodic and $\mu=\frac{m_0}{n_0}$ is a small rational number, so that the $w_j$'s have period $L=2\pi n_0$. After all the above transformations, the nonlocal term $\gamma \p_x^{-1}G(\eta)\psi$ produces the singular term $(i\mu +\p_x)^{-1}\sign(D+\mu)$, which is well-defined for all real $\mu\in (-\mez, \mez)\setminus\{0\}$ but is unbounded as $\mu\to 0$. This term is a manifestation of the singular nature of modulational instability in infinite depth with constant vorticity. Indeed, one does not confront such a singular term when either $\gamma=0$ or the depth is finite. We desingularize this  term by the singular transformation $A$ given  by \eqref{def:A}. The resulting linear operator is then bounded as $\mu\to 0$ but discontinuous at $\mu=0$ due to the term $\sign(D+\mu)$.  This obstructs any continuation arguments from $\mu=0$. Noticing however that our goal is to obtain the spectrum for $\mu>0$, we modify $\sign$ to $\sign^+$ (see \eqref{def:signplus}) to obtain a right-continuous function at $0$. The resulting $2\times 2$ operator is analytic in $\eps$ and continuous in $\mu\in [0, \mez)$ but fails to be analytic in $\mu\in [0, \mez)$ due to the presence of $|D+\mu|^\mez$. In order to employ analytic perturbation methods, we change $\nu=\sqrt{\mu}$ and arrive at the final operator $\L^+_{\nu, \eps}$ \eqref{def:Lplus} which is analytic in $\eps$ and $\nu\in [0, \frac{1}{\sqrt{2}})$, and thus can be further extended analytically to $\nu\in (-\frac{1}{\sqrt{2}}, \frac{1}{\sqrt{2}}
)$ as in \eqref{def:Lplus_new}. The following diagram provides a roadmap for all the transformations that we perform:

% https://q.uiver.app/#q=WzAsNyxbMywwLCJ2X2oiXSxbMywyLCJsX2oiXSxbMCwwLCIoXFxvbFxcZXRhLFxcb2xcXHBzaSkiXSxbMCwyLCIoXFxvbFxcZXRhLFxcb2xcXHhpKSJdLFs0LDIsIldfaiJdLFs1LDIsInFfaiJdLFs2LDIsInVfaiJdLFswLDEsIlxcV2FobGVuTWF0IiwyXSxbMiwwLCJcXHRleHR7Z29vZCB1bmtub3dufSJdLFsyLDMsIlxcV2FobGVuTWF0UGkiLDJdLFszLDEsIlxcdGV4dHtnb29kIHVua25vd259IiwyXSxbMSw0LCIoXFx6ZXRhXyosXFx6ZXRhX1xcc2hhcnApIiwyXSxbNCw1LCJlXntcXGxhbWJkYSB0fSBlXntpXFxtdSB4fSIsMl0sWzUsNiwiQV57LTF9IiwyXV0=
\[\begin{tikzcd}
	{(\ol\eta,\ol\psi)} &&& {v_j} &&& \\
	\\
	{(\ol\eta,\ol\xi)} &&& {l_j} & {W_j} & {q_j} & {u_j}
	\arrow["{\text{good unknown}}", from=1-1, to=1-4]
	\arrow["\WahlenMatPi"', from=1-1, to=3-1]
	\arrow["\WahlenMat"', from=1-4, to=3-4]
	\arrow["{\text{good unknown}}"', from=3-1, to=3-4]
	\arrow["{(\zeta_*,\zeta_\sharp)}"', from=3-4, to=3-5]
	\arrow["{e^{\lambda t} e^{i\mu x}}"', from=3-5, to=3-6]
	\arrow["{A^{-1}}"', from=3-6, to=3-7]
\end{tikzcd}\]
Having the operator $\L_{\nu, \eps}^+$ \eqref{def:Lplus_new} analytic in $(\eps, \nu)$ near $(0, 0)$, we analyze its spectrum near $0\in \Cc$. Due to nonzero vorticity and the modification of $\sign$ to $\sign^+$, $0$ is an eigenvalue of $\L_{0, 0}$ but only with algebraic multiplicity $3$. Then, using the relations between our transformations and the formulas for the four (generalized) eigenvectors of $\Ls_\eps$, we show that for small $\eps$, $0$ is still an eigenvalue of $\L^+_{0, \eps}$, with algebraic multiplicity $3$ and geometric multiplicity $2$. Next, we perturb $\eps$ and $\nu$ from $0$ by means of Kato's similarity transformation, a method introduced to the Stokes wave instability problem by Berti, Maspero, and Ventura \cite{Berti_Full}. This transformation provides a formula for the three-dimensional basis $\mathcal{V}_{\nu, \eps}$ of the perturbed eigenspace in terms of the unperturbed eigenspace and the spectral projection of $\L^+_{\nu, \eps}$. Consequently, eigenvalues near $0\in \Cc$  of $\L_{\nu, \eps}^+$ are precisely eigenvalues of the $3\times 3$ matrix $\Lmat_{\nu,\eps}$ representing  $\L_{\nu, \eps}^+$ on the $\mathcal{V}_{\nu, \eps}$ basis. Combining the Hamiltonian structure and the reversibility of $\L_{\nu, \eps}^+$ yields symmetries of entries of $\Lmat_{\nu,\eps}$. 

With the help of Mathematica's symbolic computation, we expand $\Lmat_{\nu,\eps}$ in powers of $\nu$  and $\eps$ up to $O(\eps^3, \nu^2\eps^2, \nu^4\eps, \nu^6)$ error terms. We note in particular that the $\nu\eps$ and $\nu^3\eps$ terms do not vanish, a manifestation of the fact that $\Lmat_{\nu,\eps}$ is not smooth in $\mu$ near $0$. The above finite-order expansions are insufficient for the proof of Theorem \ref{thm:main}. We first prove that the $(1,1)$, $(1,2)$, $(2,1)$, $(2,2)$, and $(3,3)$ entries of $\Lmat_{\nu,\eps}$ are analytic in $\mu=\nu^2$. Using this, the fact that $0$ is an eigenvalue of $\L^+_{0, \eps}$ with algebraic multiplicity $3$ and geometric multiplicity $2$, and leveraging parity of variable coefficients of $\L_{\nu, \eps}^+$ and eigenvectors, we obtain a number of {\it exact} relations and vanishing of certain entries of the $\nu^j$ coefficients of $\Lmat_{\nu,\eps}$, for $j\in \{0, 1, 2\}$. By ``exact'' we mean that these relations and vanishing are valid at all orders $\eps^k$, $k\ge 0$. This detailed analysis leads us to the final form \eqref{eq:Lmat_final} of $\Lmat_{\nu,\eps}$. Finally, we devise a direct method to analyze the spectrum of the $3\times 3$ matrix \eqref{eq:Lmat_final}. The aforementioned analyticity in $\mu$ of certain entries implies that the coefficients of the characteristic polynomial are analytic in $(\mu, \eps)$.  For $\mu=\eps=0$, the characteristic polynomial has a purely imaginary simple root and a purely imaginary double root. We use the analytic implicit function theorem to continue the simple root to a purely imaginary simple root as $(\mu, \eps)$ is perturbed, and derive a second-order approximation for this root. The other two roots are then roots of a quadratic polynomial and their real parts are determined by the sign of the associated discriminant. When $\gamma>\frac{2}{\sqrt{3}}$, the discriminant is negative for all small $\mu$ and $\eps$, leading to modulational stability. On the other hand, when $\gamma<\frac{2}{\sqrt{3}}$, there is a threshold $\critmu(\eps)>0$ such that the discriminant is positive when $\mu< \critmu(\eps)$, leading to Benjamin--Feir instability and the figure-$8$ curve. 
\subsection{Notation}\label{sec:notation}
  \leavevmode
\begin{itemize}
    \item $\T_L:=\R/(L\Z)$, $\T\equiv \T_{2\pi}$.
    \item The Fourier transform for an $L$-periodic function is defined by
    \begin{equation}
        \wh{f}^L(k):=\int_{\T_L} f(x)e^{-i\frac{2\pi}{L}kx}dx,\quad \forall k\in\Z,\quad \wh{f}(k) \equiv \wh{f}^{2\pi}(k),    \end{equation}
    so that
    \begin{equation}
        f(x)=\frac{1}{L}\sum_{k\in\Z}\wh f^L(k)e^{i\frac{2\pi}{L}kx},\quad f\in L^2(\T_L).
    \end{equation}
    Fourier multipliers are then defined by
\begin{equation}\label{def:FM}
    m\left(D_L\right) f(x)=\frac{1}{L} \sum_{k \in \mathbb{Z}} e^{i \frac{2 \pi}{L} k x} m\left(k \frac{2 \pi}{L}\right) \hat{f}^L(k),\quad m(D)\equiv m(D_{2\pi}).
\end{equation}
    \item For $\sigma,s\ge 0$ and $f: \T\to \Cc$,  we define the norm
        \begin{align}
        \|f\|_{H^{\sigma,s}}^2&:=\sum_{k\in\Z}(1+|k|^2)^{s}e^{2\sigma|k|}|\wh f(k)|^2
    \end{align}
   and the corresponding function space
        \begin{align}
            H^{\sigma,s}(\T):=\{f\in L^2(\T):~\|f\|_{H^{\sigma,s}}<\infty\}
        \end{align}
      together with the subspaces 
        \begin{equation}
            H^{\sigma,s}_{\mathrm{ev}}(\T):=\{f\in H^{\sigma,s}(\T):f~\text{is~even}\},\quad H^{\sigma,s}_{\mathrm{odd}}(\T):=\{f\in H^{\sigma,s}(\T): f~\text{is~odd}\}.
        \end{equation}
    We set 
    \[
    H^s\equiv H^{0, s}.
    \]
    \item For $f,g \in \left(L^2(\T)\right)^2$, 
    $$(f, g):=\int_{\T} f(x) \cdot ~\overline{g}(x)dx,$$ 
    where the `bar' denotes complex conjugation  and the `dot' denotes the real dot product. 
   % \item $D:=\p_x/i$, and $f(D)e^{ikx}=f(k)e^{ikx}$. 
    \item For an $L$-periodic function $f$, we define the projections onto the zero mode and nonzero modes by
\[
\Pi_0f=\fint f=\frac{1}{L}\int_0^Lfdx,\quad \Pi_0^\perp f=f-\Pi_0f.
\]
If $\Pi_0f=0$, we denote the unique mean-free primitive of $f$ by $\pinv f$, which is $L$-periodic.  We note the identities  
\begin{equation}\label{dinverse}
\p_x\pinv=\mathrm{Id},\quad \pinv\p_x=\Pi_0^\perp.
\end{equation}
\item For independent variables $(X_1,\dots,X_n)\in\R^n$ and a Banach space $B$, we denote the ring of $B$-valued analytic functions near $0\in \R^n$ by $B\{X_1,\dots,X_n\}$. 
\begin{itemize}
    \item For a function $F(X_1,\dots,X_n)\in B\{X_1,\dots,X_n\}$, we denote by $[X_j^m]F$ the coefficient of $X_j^m$ in the expansion of $F$, that is, 
    \begin{equation}\label{nota:coefficient}
    [X_j^m]F:=m!\p_{X_j}^mF|_{X_j=0}.
   \end{equation}
\item If $F(\mu,\eps)$ is an analytic function in a neighborhood of $(0, 0)$, by writing $F=O(\mu^{m_1}\eps^{n_1},\dots,\mu^{m_p}\eps^{n_p})$ we mean $F(\mu,\eps)=\sum_{j=1}^p \mu^{m_j}\eps^{n_j}G_j(\mu,\eps)$ for some analytic functions $G_j(\mu,\eps)$ near $(0,0)$.
Note that $O(\mu^{m_1}\eps^{n_1},\dots,\mu^{m_p}\eps^{n_p})$ defined here is not  the same as the standard notation $F=O(|\mu^{m_1}\eps^{n_1}|+\dots+|\mu^{m_p}\eps^{n_p}|)$, where the latter means 
\[
|F|\le C\sum_{k=1}^p |\mu^{m_k}\eps^{n_k}|.
\]
  For example,  $\mu \eps =O(\mu^2+\eps^2)$  but $\mu\eps \ne O(\mu^2,\eps^2)$.
\item By $r(\mu^{m_1}\eps^{n_1},\dots,\mu^{m_p}\eps^{n_p})\in\anaR{\mu,\eps}$ we mean any real analytic function satisfying 
   \begin{equation}
   \label{def:analytic_r}
       r(\mu^{m_1}\eps^{n_1},\dots,\mu^{m_p}\eps^{n_p})=O(\mu^{m_1}\eps^{n_1},\dots,\mu^{m_p}\eps^{n_p}).
   \end{equation}
\end{itemize}
   \item We denote by $even_0$ a $2\pi$-periodic even function with mean zero on $\T$, and $odd$ a $2\pi$-periodic odd function. They may denote different functions in different contexts.
\end{itemize}
\subsection{Water waves with constant vorticity}
We consider two-dimensional inviscid, incompressible fluids in the moving domain  $\Omega(t)=\{(x, y): x \in \mathbb{R}, y<\eta(x, t)\} $, where $\{ y=\eta(x,t)\}$ is the free surface. Denoting the velocity field by $\velo:=\vet{u}{v}$, the water wave system reads
\begin{align}
\label{eq:Euler_orig}
    \velo _t+(\velo\cdot\na)\velo&=-\na p-\vet{0}{g},\quad \na \cdot \velo =0&& \text{in }\Om,\\
    p&=p_{atm}&&\text{on }\{y=\eta(x)\},\\
    \label{eq:kinematic}
    \eta_t&=-u\eta_x+v&&\text{on }\{y=\eta(x)\},\\
    v&\to 0&&\text{when }\{y\to -\infty\},
\end{align}
where $g>0$ is the acceleration due to gravity, and $p_{atm}$ is the constant atmospheric pressure above the free surface. We assume that the vorticity is constant,
\begin{equation}
v_x-u_y=\gamma,
\end{equation} 
a property preserved by the flow. Then we can express the velocity field as the sum of a shear flow and a potential flow, that is, 
\begin{equation}
\velo=
        -\gamma y \mathbf e_x+\na\Phi,
    \end{equation}
     where $\mathbf e_x$ is the unit vector in the $x$-direction and $\Phi$ is a harmonic in $\Omega$. We denote
     \begin{equation}
     \psi(x)=\Phi(x, \eta(x)).
     \end{equation}
In the moving frame $(x, y,  t)\to (x- ct, y, t)$, the water wave system with constant vorticity $\gamma$ can be recast in terms of $\eta$ and $\psi$:
\begin{equation}
\label{eq:main_with_t}
\begin{cases}
\eta_t=F_1(\eta,\psi,c):= c\partial_x\eta +G(\eta) \psi+\gamma \eta \eta_x, \\
\psi_t=F_2(\eta,\psi,c,P):= c\partial_x\psi -g \eta-\frac{{\psi_x^2}}{2}+\frac{\big(\eta_x \psi_x+G(\eta) \psi\big)^2}{2\left(1+\eta_x^2\right)}+\gamma \eta \psi_x+\gamma \partial_x^{-1} G(\eta) \psi+P,
\end{cases}
\end{equation}
where $P\in \Rr$ denotes the Bernoulli constant.  Here $G(\eta)$ denotes the nonlocal Dirichlet-Neumann operator defined by
\begin{align}
   [ G(\eta)\psi](x)&:=\sqrt{1+\eta_x^2}\p_n\Theta|_{y=\eta(x)}
   =\Theta_y(x,\eta(x))-\eta_x(x)\Theta_x(x,\eta(x)),
\end{align}
where $n$ denotes the exterior unit normal to $\{y=\eta(x)\}$ and $\Theta$ solves the elliptic problem
\[
\begin{cases}
\Delta \Theta=0\quad\text{in~}\Omega,\\
\Theta\vert_{y=\eta(x)}=\psi(x),\quad \na \Theta(x, y)\in L^2(\Omega).
\end{cases}
\]
\eqref{eq:main_with_t} is a generalization of the well-known Zakharov-Craig-Sulem formulation \cite{Zakharov, Craig_Sulem} for irrotational flows. We refer to \cite{ConstantinIvanovProdanov2008, Wahlen_Hamiltonian} for the derivation of \eqref{eq:main_with_t}. 

By a Stokes wave we mean any steady periodic solution  $(\eta(x), \psi(x)): \T\to \Rr^2$ of \eqref{eq:main_with_t}, that is, $(\eta(x), \psi(x))$ satisfies 
\bq
\label{eq:main}
\begin{cases}
0=c \eta_x+G(\eta) \psi+\gamma \eta \eta_x, \\
0=c \psi_x-g \eta-\frac{{\psi_x^2}}{2}+\frac{\big(\eta_x \psi_x+G(\eta) \psi\big)^2}{2\left(1+\eta_x^2\right)}+\gamma \eta \psi_x+\gamma \partial_x^{-1} G(\eta) \psi+P.
\end{cases}
\eq
By the change of variables
\begin{equation}
\label{eq:nondim}
    (\eta,\psi,c,P,t,\gamma)\to(\eta,\sqrt{g}\psi,\sqrt{g}c,gP,\frac{t}{\sqrt{g}},\sqrt{g}\gamma),
\end{equation}
we assume (without loss of generality) hereafter that  $g=1$.
\subsection{Small Stokes waves with constant vorticity}
The existence of small irrotational Stokes waves is a classical result \cite{Nekrasov, Levi_Civita,Struik1926,Analyticity_of_DN_op}. In the presence of constant vorticity, such a result can be found in \cite[Theorem 1.1]{2025}. For the sake of completeness and to include the second-order expansion, we  will prove Theorem \ref{thm:Stokes_without_g}  in Appendix \ref{appendix:local_bif}. This theorem covers both the finite and infinite depth cases, where  $h\in(0,\infty]$ denotes the depth. 
\begin{theo}
\label{thm:Stokes_without_g}
Let  $\sigma>0$ and $s\ge\frac72$ such that $s+\mez \in\Z$. For any constant vorticity $\gamma$ and any depth $h\in (0, \infty]$, there exist $\eps_{St}=\eps_{St}(\gamma,h)>0$ and a unique family of solutions $\big(\eta(\cdot; a,h), \psi(\cdot; a,h), c(a,h)\big)\in H^s(\T)\times H^s(\T)\times \Rr$ to the traveling wave system \eqref{eq:main}, parametrized by $a\in (-\eps_{St}, \eps_{St})$ such that 

(i) the mapping $a\mapsto \big(\eta(\cdot; a,h), \psi(\cdot; a,h), c(a,h)\big)\in H^{\sigma, s}(\T)\times H^{\sigma, s}(\T)\times \Rr$ is analytic, and

(ii) $\eta(\cdot; a,h)$ is even and $\psi(\cdot; a,h)$ is odd. 

Moreover, for $P=0$ the following expansions hold: 
\begin{equation}
    \label{expand:Stokes_without_g_all_depth}
        \begin{aligned}
c=c(a,h)&=c_0+a^2c_2+O(a^3),\\ 
    \eta(x)=\eta(x;a,h)&= 
        a \Big[\th \cos(x)\Big]+\frac{ a ^2}4\Big[(c_0^2-2\th+c_0^2\th ^2)\\
        &+\Big(\frac{1}{\th}+\frac{c_0^2+\th(1-2c_0^4)+c_0^2\th ^2}{c_0^4}\Big)\cos(2x)\Big]+O(a^3),
    \\
    \psi(x)=\psi(x;a,h)&= a \Big[ c_0\sin(x)\Big]+ a ^2\Big[\frac{c_0^4+c_0^2\th+(1-2c_0^4)\th^2+c_0^2\th^3+(1+c_0^4)\th^4}{8c_0^3\th^2}\sin(2x)\Big]+O(a^3),
\end{aligned}
\end{equation}
and 
\begin{equation}
    \begin{aligned}
        c_2=c_2(h)&=\frac{1}{8 c_0^5 \th^2 \left(c_0^2+\th\right)}\Bigg(c_0^8+2 c_0^6 \th+\left(11 c_0^8-7 c_0^4+1\right) \th^4+\left(3 c_0^4-5 c_0^8\right) \th^2\\
        &+\left(-2 c_0^{10}+4 c_0^6+2 c_0^2\right) \th^5+\left(-2 c_0^{10}-2 c_0^6+2 c_0^2\right) \th^3\Bigg),
    \end{aligned}
\end{equation}
 where $\th:=\tanh(h)$ and $c_0$ is the positive solution of 
\begin{equation}
    \label{eq:relation_c0_gamma_general}
\frac{c_0}{\th}-\frac{1}{c_0}=\gamma.
\end{equation}
For $h=\infty$ and arbitrary $P$, the following expansions hold:
\begin{equation}
    \label{expand:Stokes_without_g}
        \begin{aligned}
c=c(a,\infty)&=\Big(c_0-(c_0-\frac{1}{c_0})P\Big)+a^2\Big(\frac{-4 c_0^8 +11 c_0^6  -7 c_0^4 +3 c_0^2  +1}{8 c_0^5  }\Big)+O(a^3),\\ 
    \eta(x)=\eta(x;a,\infty)&=P+a \cos x+a^2\Big(\frac{c_0^2-1}{2}+\frac{1+2c_0^2-c_0^4}{4c_0^4}\cos(2 x)\Big)+O(a^3),\\
    \psi(x)=\psi(x;a,\infty)&=ac_0\sin x+a^2\Big(\frac{1+c_0^2}{4c_0^3}\sin(2 x)\Big)+O(a^3), 
\end{aligned}
\end{equation}
where $c_0$ is the positive solution of
\begin{equation}
    \label{eq:relation_c0_gamma}
c_0-\frac{1}{c_0}=\gamma.
\end{equation}
\end{theo}
In what follows, we  only consider the infinite depth case and assume without loss of generality that $P=0$. Indeed, when $h=\infty$, the general $P$ case can be reduced to $P=0$ by setting 
\begin{equation}
    \label{eq:symm_P}
    \wt{\eta}=\eta-P,\quad \wt{c}=c+\gamma P
\end{equation}
and noticing that $G\left(\widetilde{\eta}+P \right) \psi=G(\widetilde{\eta})\psi$. 
\section{Linearization and transformations}
We will study the stability of the Stokes waves constructed in Theorem \ref{thm:Stokes_without_g} with respect periodic  perturbations of period $L=n_02\pi$ for some $n_0\in\Nn$.
\subsection{Linearization}
 To linearize \eqref{eq:main_with_t} about Stokes waves, we follow \cite{Huy21}  and use the shape-derivative formula:
\begin{prop}[\cite{Lannes,Alazard}]\label{shape}
    For $L$-periodic functions, the derivative of the Dirichlet-Neumann operator $\eta\mapsto G(\eta)\psi$ is given by
\begin{equation}
\label{shapederi}
    \frac{\delta G(\eta) \psi}{\delta \eta}(\ol{\eta})=-G(\eta)(B \ol{\eta})-\partial_x(V \ol{\eta}),
\end{equation}
where
\begin{equation}
\label{def:B_and_V}
    B=B(\eta, \psi)=\frac{G(\eta) \psi+\partial_x \psi \partial_x \eta}{1+\left|\partial_x \eta\right|^2}, \quad V=V(\eta, \psi)=\partial_x \psi-B \partial_x \eta.
\end{equation}
Moreover, if $\eta$ is even and $\psi$ is odd, then $B$ is odd and $V$ is even.
\end{prop}
Next, we use Proposition \ref{shape} to calculate the differential of the nonlinearities $F_1$ and $F_2$ in  \eqref{eq:main_with_t}.
\begin{lemm}
    We have 
    \begin{equation}\label{eq:dF1}
    \begin{aligned}
        \frac{\delta F_1(\eta, \psi, c)}{\delta(\eta, \psi)}(\ol{\eta}, \ol{\psi})
        &=\partial_x\Big((c-V) \ol{\eta}\Big)+G(\eta)\Big(\ol{\psi}-B \ol{\eta}\Big)+\gamma\Big(\eta_x\ol{\eta}+\eta\ol{\eta}_x\Big)\\
        &=\partial_x\Big((c-V+\gamma\eta) \ol{\eta}\Big)+G(\eta)\Big(\ol{\psi}-B \ol{\eta}\Big),
    \end{aligned}
\end{equation}
\begin{multline}\label{eq:dF2}
       \frac{\delta F_2(\eta, \psi, c,P)}{\delta(\eta, \psi)}(\ol{\eta}, \ol{\psi})=(c-V)\ol\psi_x-\ol\eta-BV_x\ol\eta+BG(\eta)(\ol\psi-B\ol\eta)\\
       +\gamma\Big((\psi_x-V)\ol\eta+\eta\ol\psi_x+\p_x^{-1}G(\eta)(\ol\psi-B\ol\eta)+\Pi_0(V\ol\eta)\Big),
\end{multline}
where $B$ and $V$ are given by \eqref{def:B_and_V}. Moreover, we have 
\begin{equation}\label{identity}
    \begin{aligned}
     \frac{\delta F_2(\eta, \psi, c,P)}{\delta(\eta, \psi)}(\ol{\eta}, \ol{\psi})-B \frac{\delta F_1(\eta, \psi, c)}{\delta(\eta, \psi)}(\ol{\eta}, \ol{\psi})
    &=-\bigg(1+(V-c-\gamma\eta)B_x\bigg)\ol\eta\\
    &\quad+(c-V+\gamma\eta)\p_x(\ol\psi-B\ol\eta)\\
    &\quad +\gamma\p_x^{-1}G(\eta)(\ol\psi-B\ol\eta)+\gamma \Pi_0(V\ol\eta).
\end{aligned}
\end{equation}
\end{lemm}
\begin{proof}
    We first note that \eqref{eq:dF1} follows from the shape-derivative formula \eqref{shapederi} and the formula 
    \begin{align*}
    \frac{\delta}{\delta(\eta, \psi)}\big(\gamma\eta\eta_x\big)(\ol{\eta},\ol{\psi})=\gamma(\eta_x\ol{\eta}+\eta\ol{\eta}_x).
\end{align*}
Likewise, we have
\begin{align*}
    \frac{\delta}{\delta(\eta, \psi)}\big(\gamma\eta\psi_x\big)(\ol{\eta},\ol{\psi})=\gamma(\psi_x\ol{\eta}+\eta\ol{\psi}_x). 
\end{align*}
Invoking \eqref{shapederi} again and using the fact that $\p_x^{-1}\p_x=\Pi_0^\perp$ (see \eqref{dinverse}), we find
\begin{align*}
    \frac{\delta}{\delta(\eta, \psi)}\bigg(\gamma\p_x^{-1}G(\eta)\psi\bigg)(\ol\eta,\ol\psi)=-\gamma\p_x^{-1}G(\eta)(B\ol{\eta})-\gamma\Pi_0^\perp(V\ol{\eta})+\gamma\p_x^{-1}G(\eta)\ol{\psi}.
\end{align*}
In addition, we have
\begin{align*}
    &\quad \frac{\delta}{\delta(\eta, \psi)}\frac{\big(\eta_x \psi_x+G(\eta) \psi\big)^2}{2\left(1+\eta_x^2\right)}(\ol{\eta},\ol{\psi})\\
    &=\frac{\big(\eta_x \psi_x+G(\eta) \psi\big)}{\left(1+\eta_x^2\right)}\frac{\delta}{\delta(\eta, \psi)}(\eta_x\psi_x+G(\eta)\psi)(\ol{\eta},\ol\psi)-\frac{\big(\eta_x \psi_x+G(\eta) \psi\big)^2}{2\left(1+\eta_x^2\right)^2}\frac{\delta}{\delta(\eta, \psi)}(1+\eta_x^2)(\ol\eta,\ol\psi)\\
    &=B\bigg(\eta_x\ol\psi_x+\ol\eta_x\psi_x-G(\eta)(B\ol\eta)-\p_x(V\ol\eta)+G(\eta)\ol\psi\bigg)-B^2\eta_x\ol\eta_x.
\end{align*}
Consequently, recalling that $V=\psi_x-B\eta_x$, we obtain
\begin{align*}
    &\quad \frac{\delta F_2(\eta, \psi, c,P)}{\delta(\eta, \psi)}(\ol{\eta}, \ol{\psi})\\
    &=c\ol\psi_x-\ol\eta-\psi_x\ol\psi_x+\Bigg(B\Big(\eta_x\ol\psi_x+\psi_x\ol\eta_x-G(\eta)(B\ol\eta)-\p_x(V\ol\eta)+G(\eta)\ol\psi\Big)-B^2\eta_x\ol\eta_x\Bigg)\\
    &\quad+\gamma\Big(\psi_x\ol{\eta}+\eta\ol{\psi}_x-\p_x^{-1}G(\eta)(B\ol{\eta})-\big(V\ol{\eta}-\Pi_0(V\ol\eta)\big)+\p_x^{-1}G(\eta)\ol{\psi}\Big)\\
    &=(c-\psi_x+B\eta_x)\ol\psi_x-\ol\eta+(B\psi_x-B^2\eta_x)\ol\eta_x-BV_x\ol\eta-BV\ol\eta_x+BG(\eta)(\ol\psi-B\ol\eta)\\
    &\quad +\gamma\Bigg((\psi_x-V)\ol\eta+\eta\ol\psi_x+\p_x^{-1}G(\eta)\ol\psi-\p_x^{-1}G(\eta)(B\ol\eta)+\Pi_0(V\ol\eta)\Bigg)\\
    &=(c-V)\ol\psi_x-\ol\eta-BV_x\ol\eta+BG(\eta)(\ol\psi-B\ol\eta)+\gamma\Big((\psi_x-V)\ol\eta+\eta\ol\psi_x+\p_x^{-1}G(\eta)(\ol\psi-B\ol\eta)+\Pi_0(V\ol\eta)\Big),
\end{align*}
proving \eqref{eq:dF2}. 

Finally,  \eqref{identity} follows from the following calculation: 
\begin{align*}
     &\quad\frac{\delta F_2(\eta, \psi, c,P)}{\delta(\eta, \psi)}(\ol{\eta}, \ol{\psi})-B \frac{\delta F_1(\eta, \psi, c)}{\delta(\eta, \psi)}(\ol{\eta}, \ol{\psi})\\
     &=(c-V)\ol\psi_x-\ol\eta-BV_x\ol\eta+BG(\eta)(\ol\psi-B\ol\eta)+\gamma\Big((\psi_x-V)\ol\eta+\eta\ol\psi_x+\p_x^{-1}G(\eta)(\ol\psi-B\ol\eta)+\Pi_0(V\ol\eta)\Big)\\
     &\quad-B\bigg(\partial_x\big((c-V) \ol{\eta}\big)+G(\eta)\big(\ol{\psi}-B \ol{\eta}\big)+\gamma\big(\eta_x\ol{\eta}+\eta\ol{\eta}_x\big)\bigg)\\
     &=(c-V)\ol\psi_x-\ol\eta-BV_x\ol\eta-B(c-V)\ol\eta_x+BV_x\ol\eta\\
     &\quad+\gamma\bigg((\psi_x-V)\ol\eta+\eta\ol\psi_x+\p_x^{-1}G(\eta)(\ol\psi-B\ol\eta)-B\eta_x\ol{\eta}-B\eta\ol{\eta}_x+\Pi_0(V\ol\eta)\bigg)\\
         &=(c-V)(\ol\psi_x-B\ol\eta_x)-\ol\eta+\gamma\bigg((\psi_x-V-B\eta_x)\ol\eta+\eta(\ol\psi_x-B\ol\eta_x)+\p_x^{-1}G(\eta)(\ol\psi-B\ol\eta)+\Pi_0(V\ol\eta)\bigg)\\
    &=(c-V+\gamma\eta)(\ol\psi_x-B\ol\eta_x)-\ol\eta+\gamma\p_x^{-1}G(\eta)(\ol\psi-B\ol\eta)+\gamma\Pi_0(V\ol\eta)\\
    &=(c-V+\gamma\eta)\Big(\p_x(\ol\psi-B\ol\eta)+B_x\ol\eta\Big)-\ol\eta+\gamma\p_x^{-1}G(\eta)(\ol\psi-B\ol\eta)+\gamma\Pi_0(V\ol\eta)\\
    &=-\big(1+(V-c-\gamma\eta)B_x\big)\ol\eta+(c-V+\gamma\eta)\p_x(\ol\psi-B\ol\eta)+\gamma\p_x^{-1}G(\eta)(\ol\psi-B\ol\eta)+\gamma\Pi_0(V\ol\eta).
\end{align*}

\end{proof}
From now on, we fix a Stokes wave $\left(\eta^*, \psi^*, c^*, P^*=0\right)$ as given in \eqref{expand:Stokes_without_g}, with amplitude $a=\eps$. Let  $(V^*, B^*)$ be the trace of the velocity of this Stokes wave, as given by \eqref{def:B_and_V}, 
\begin{equation}
    B^*=\frac{G(\eta^*) \psi^*+\partial_x \psi^* \partial_x \eta^*}{1+\left|\partial_x \eta^*\right|^2}, \quad V^*=\partial_x \psi^*-B^* \partial_x \eta^*.
\end{equation}
We fix $c=c^*$ and $P=P^*=0$ in \eqref{eq:main_with_t} and linearize   around $(\eta^*, \psi^*)$. The formulas \eqref{eq:dF1} and \eqref{eq:dF2} yield
\begin{align}\label{eq:etabar:00}
 \partial_t \ol{\eta}=\frac{\delta F_1(\eta^*, \psi^*, c^*)}{\delta(\eta, \psi)}(\ol{\eta}, \ol{\psi})&= \partial_x\Big((c^*-V^*+\gamma \eta^*) \ol{\eta}\Big)+G(\eta^*)\Big(\ol{\psi}-B^* \ol{\eta}\Big), \\
\nonumber\partial_t \ol{\psi}=\frac{\delta F_2(\eta^*, \psi^*, c^*, P^*)}{\delta(\eta, \psi)}(\ol{\eta}, \ol{\psi})&=-\big(1+(V-c-\gamma\eta)B_x\big)\ol\eta+(c-V+\gamma\eta)\p_x(\ol\psi-B\ol\eta)\\ 
&\quad+\gamma\p_x^{-1}G(\eta)(\ol\psi-B\ol\eta)+\gamma\Pi_0(V\ol\eta). \label{eq:psibar:00}
\end{align}
Guided by the irrotational case \cite{Huy21}, we use Alinhac's \textit{good unknown}
\begin{equation}\label{def:v1v2}
    v_1=\ol{\eta}, \quad v_2=\ol{\psi}-B^* \ol{\eta}.
\end{equation}
By invoking \eqref{identity}, the system \eqref{eq:etabar:00}-\eqref{eq:psibar:00} is equivalent to 
\begin{align}
    \label{eq:v1}\p_tv_1&=\p_x\Big((c^*-V^*+\gamma\eta^*)v_1\Big)+G(\eta^*)v_2,\\
    \label{eq:v2}\p_tv_2&=-\Big(1+(V^*-c^*-\gamma\eta^*)\p_xB^*\Big)v_1+\gamma \Pi_0(V^* v_1)+(c^*-V^*+\gamma\eta^*)\p_xv_2+\gamma\p_x^{-1}G(\eta^*)v_2.
\end{align}
In matrix form, this reads
\begin{equation}
    \label{eq:v_matrix}
    \p_t \vet{v_1}{v_2}=\Ls_\eps\vet{v_1}{v_2},\quad \Ls_\eps:=
    \begin{bmatrix}
        \p_x\circ (c^*-V^*+\gamma\eta^*)& G(\eta^*)\\
        -\Big(1+(V^*-c^*-\gamma\eta^*)\p_xB^*\Big)+\gamma\Pi_0\circ V^* & (c^*-V^*+\gamma\eta^*)\p_x+\gamma\p_x^{-1}G(\eta^*)
    \end{bmatrix},
\end{equation}
where $\Ls_\eps: H^1(\T)^2\to L^2(\T)^2$ continuously. 

 In particular, we have
\[
\Ls_0=\begin{bmatrix}
c_0\p_x & |D|\\
-1 & c_0\p_x+\gamma\p_x^{-1}|D|
\end{bmatrix}.
\]
The spectrum of $\Ls_0$ is comprised  of the purely imaginary eigenvalues 
\begin{equation}
\ld^\pm(k)=i\left[c_0k-\mez \gamma\sign(k)\pm \sqrt{\frac14 \gamma^2\sign^2(k)+|k|}\right],\quad k\in \Zz,
\end{equation}
which are the roots of the characteristic equation
\begin{equation}\label{chareq:0}
\ld^2-i\ld(2kc_0-\gamma \sign(k))+|k|-kc_0(kc_0-\gamma\sign(k))=0.
\end{equation}
\begin{lemm}\label{lemm:0ev}
$\ld=0$ is an isolated eigenvalue of $\Ls_0$ and its generalized eigenspace has dimension $4$. 
\end{lemm}
\begin{proof}
Clearly $\ld^\pm(0)=0$. Furthermore, using the relation $\gamma=c_0-c_0^{-1}$, we find that $\ld^+(-1)=\ld^-(1)=0$. It remains to show that the set $\{\ld^+(k): k\ne 0, -1\}\cup \{\ld^-(k): k\ne 0, 1\}$ has positive distance to $0$. To this end, we first note that 
\begin{equation}\label{ldpm1}
\ld^+(1)=i(c_0+c_0^{-1}),\quad \ld^-(-1)= -i(c_0+c_0^{-1}). 
\end{equation}
We assume for the sake of contradiction that there exists $|k|\ge 2$ such that 
$$|\lambda|<\min\left\{1,\frac{c_0^2}{1+3c_0+\frac{1}{c_0}}\right\}.$$
Then, on one hand, we have
\begin{align*}
   \left| \ld^2-i\ld(2kc_0-\gamma \sign(k))\right| &\le  |\lambda|^2+|\lambda|\bigg(2c_0|k|+c_0+\frac{1}{c_0}\bigg)\\
    &\le |\lambda||k|+|\lambda|\bigg(2c_0|k|+(c_0+\frac{1}{c_0})|k|\bigg)\\
    &\le |\lambda||k|\bigg(1+3c_0+\frac{1}{c_0}\bigg)\\
    &< c_0^2|k|.
\end{align*}
On the other hand, we have 
\[
-|k|+kc_0(kc_0-\gamma\sign(k))=c_0^2(k^2-|k|)\ge c_0^2|k|,
\]
implying that $\ld$ cannot be a root of \eqref{chareq:0}. 
\end{proof}
When $\gamma=0$, it was proven in \cite{Huy21} that for all sufficiently small $\eps$,  $\Ls_\eps$ admits $\ld=0$ as an eigenvalue with algebraic multiplicity $4$.  This comes from the fact that   \eqref{eq:main} is invariant under translation in $x$ and the operation $\psi\mapsto \psi+C$, and  that the Stokes waves family depends on two parameters $a$ and $P$. We now extend this result to all $\gamma \in \Rr$:
\begin{theo}
    \label{thm:eigenvectors_R} 
For all $\gamma \in \Rr$, there exists $\eps_0=\eps_0(\gamma)>0$ such that for all $0<|\eps|<\eps_0$,  $0$ is an eigenvalue of $\Ls_\eps$ with algebraic multiplicity $4$ and geometric multiplicity $2$ or $3$. Moreover, 
    \begin{align}\label{form:R1}
        R_1&=[0,1]^T~\text{and}\\ \label{form:R2}
        R_2&=\Big[\p_x\eta^*, \p_x\psi^*-B^*\p_x\eta^*\Big]^T
        \end{align}
        are eigenvectors, 
\begin{align}
             R_3&=\Big[\p_a\eta\vert_{(a, P)=(\eps, 0)}, \big(\p_a\psi-B^*\p_a\eta\big)\vert_{(a, P)=(\eps, 0)}\Big]^T
        \end{align}
        can be an eigenvector or a generalized eigenvector, depending on the parameters $\eps$ and $\gamma$, and
        \begin{align}
        R_4&=\Big[\p_P\eta\vert_{(a, P)=(\eps, 0)}, \big(\p_P\psi-B^*\p_P\eta\big)\vert_{(a, P)=(\eps, 0)}\Big]^T
    \end{align}
       is a generalized eigenvector. 
 \end{theo}
\begin{proof}
According to Lemma \ref{lemm:0ev}, $0$ is an isolated  eigenvalue of $\Ls_0$ and its generalized eigenspace has dimension $4$. Since $\Ls_\eps$ is continuous in $\eps$,  the semicontinuity of  separated parts of the spectrum (see IV-\textsection 3.4 in \cite{Kato}) implies that when $\eps$ is small enough the spectrum of $\Ls_\eps$ is likewise  separated in two parts, where the part near $0$ has a total dimension of generalized eigenspaces equals  $4$.  

Since $G(\eta^*)1=\p_x1=0$, it follows from \eqref{eq:v_matrix} that
   \begin{equation}
       \label{eq:thmR1} 
       \Ls_\eps R_1=\Ls_\eps \vet{0}{1}=0.
   \end{equation}
We differentiate \eqref{eq:main} with respect to $x$ and evaluate the resulting equations at $(a,P)=(\eps,0)$ to get
\[
\frac{\delta F_1(\eta^*, \psi^*, c^*)}{\delta(\eta, \psi)}(\p_x\eta^*, \p_x\psi^*)=0,  
\quad \frac{\delta F_2(\eta^*, \psi^*, c^*, P^*)}{\delta(\eta, \psi)}(\p_x\eta^*, \p_x\psi^*)=0.
\] 
Invoking  the identities \eqref{eq:dF1} and \eqref{identity} with $\ol\eta=\p_x\eta^*,\ol\psi=\p_x\psi^*$, we obtain 
\begin{align*}
    &\partial_x\Big((c^*-V^*+\gamma\eta^*) \p_x\eta^*\Big)+G(\eta^*)\Big(\p_x\psi^*-B^* \p_x\eta^*\Big)=0,
\end{align*}
and
\begin{align*}
    &
    -\bigg(1+(V^*-c^*-\gamma\eta^*)\p_xB^*\bigg)\p_x\eta^*+(c^*-V^*+\gamma\eta^*)\p_x(\p_x\psi^*-B^*\p_x\eta^*)\\
    &\quad+\gamma\p^{-1}_xG(\eta^*)(\p_x\psi^*-B^*\p\eta^*)+\gamma \Pi_0(V\p_x\eta^*)=0.
\end{align*}
These two identities imply
    \begin{equation}
        \label{eq:thmR2}
        \Ls_\eps  R_2=0.
    \end{equation}
    Next, we  differentiate \eqref{eq:main} with respect to $a$ and evaluate it at $(a,P)=(\eps,0)$ to get
    \begin{align*}
&\frac{\delta F_1(\eta^*, \psi^*, c^*)}{\delta(\eta, \psi)}\bigg(\p_a\eta\vert_{(a, P)=(\eps, 0)}, \p_a\psi\vert_{(a, P)=(\eps, 0)}\bigg)=-\p_a c\vert_{(a, P)=(\eps, 0)}\p_x\eta^*,\\
& \frac{\delta F_2(\eta^*, \psi^*, c^*, P^*)}{\delta(\eta, \psi)}\bigg(\p_a\eta\vert_{(a, P)=(\eps, 0)}, \p_a\psi\vert_{(a, P)=(\eps, 0)}\bigg)=-\p_a c\vert_{(a, P)=(\eps, 0)}\p_x\psi^*.
\end{align*}   
Using \eqref{eq:dF1} and \eqref{identity} with $(\ol\eta,\ol\psi)=(\p_a\eta^*,\p_a\psi^*)\vert_{(a, P)=(\eps, 0)}$, we obtain
\begin{equation}
    \label{eq:thmR3}
    \begin{aligned}
        \Ls_\eps R_3
        &=\vet{-\p_a c\vert_{(a, P)=(\eps, 0)}\p_x\eta^*}{-\p_a c\vert_{(a, P)=(\eps, 0)}\p_x\psi^*+\p_a c\vert_{(a, P)=(\eps, 0)}B^*\p_x\eta^*}=-\p_a c\vert_{(a, P)=(\eps, 0)}R_2.
    \end{aligned}
\end{equation}
Hence, $R_3$ is either an eigenvector or a generalized eigenvector, depending on whether $\p_a c\vert_{(a, P)=(\eps, 0)}$ vanishes or not. 

 Finally, differentiating \eqref{eq:main} in $P$ yields
\begin{align*}
&\frac{\delta F_1(\eta^*, \psi^*, c^*)}{\delta(\eta, \psi)}\bigg(\p_P\eta\vert_{(a, P)=(\eps, 0)}, \p_P\psi\vert_{(a, P)=(\eps, 0)}\bigg)=-\p_P c\vert_{(a, P)=(\eps, 0)}\p_x\eta^*,\\
& \frac{\delta F_2(\eta^*, \psi^*, c^*, P^*)}{\delta(\eta, \psi)}\bigg(\p_P\eta\vert_{(a, P)=(\eps, 0)}, \p_P\psi\vert_{(a, P)=(\eps, 0)}\bigg)=-\p_P c\vert_{(a, P)=(\eps, 0)}\p_x\psi^*-1,
\end{align*}
and hence
\begin{equation}
    \label{eq:thmR4}
    \Ls_\eps  R_4=-\p_P c\vert_{(a, P)=(\eps, 0)}R_2-R_1=\gamma R_2-R_1\ne 0.
\end{equation}
We have used the fact that $\p_P c\vert_{(a, P)=(\eps, 0)}=-(c_0-\frac{1}{c_0})=-\gamma$ which in turn follows from the expansion for $c$ in \eqref{expand:Stokes_without_g}. Finally, the linear independence of the $R_j$'s can be readily seen from the expansions 
\bq
R_2=\vet{-\eps\sin x}{\eps c_0\cos x}+\rem 2,\quad R_3=\vet{\cos x}{c_0\sin x}+O(\eps),
    \eq
    provided  $\eps\ne 0$. 
\end{proof}
For long wave perturbations of the Stokes wave, we seek solutions to the linearized system \eqref{eq:v1}-\eqref{eq:v2} of the form 
\begin{equation}
\label{eq:v12_vstar12}
    v_j(x,t)=e^{\lambda t}e^{i\mu x}v_j^*(x),
\end{equation}
where  $v_j^*(x)$ is $2\pi$-periodic. We assume 
\bq\label{choice:mu} 
\mu=\frac{m_0}{n_0}\in (-\mez, \mez)\setminus\{0\}~\text{ is a small rational number and choose}~L=n_0 2\pi \gg 2\pi,
\eq
 so that the $v_j(\cdot, t)$ is $L$-periodic. The parameter $\mu$ is referred to as the \textit{Floquet exponent}. For conjugation with $e^{i\mu x}$, we recall from  \cite[Lemma 3.5]{Huy21}    that for any Fourier multiplier $a$ and $f\in L^2(\T)$, we have
\begin{equation}
    \label{eq:Bloch_general}
    e^{-i\mu x}a(D_L)\big(e^{i\mu x}f(x)\big)=a(D+\mu)f(x).
\end{equation}
\subsection{Hamiltonian formulation}
A linear system is called Hamiltonian of it has the form $\p_t f=JBf$, where 
    \[
J:=\begin{bmatrix}
    0& \mathrm{Id}\\
    -\mathrm{Id}&0
\end{bmatrix}
\]
is the canonical symplectic matrix and $B$ is self-adjoint. A transformation $T$ is called symplectic if  $T^*JT=J$. If $f$ satisfies a linear Hamiltonian system and $T$ is a symplectic transformation, then $Tf$ satisfies the Hamiltonian system $\p_t (Tf)=J(T^{-1})^*BT^{-1}(Tf)$. 

As discussed in \cite{ConstantinIvanovProdanov2008}, due to the presence of non-zero vorticity, the water wave system  \eqref{eq:main_with_t} of $(\eta, \psi)$ is not Hamiltonian with respect to the canonical symplectic matrix $J$. Nevertheless, \cite{Wahlen_Hamiltonian} discovered that  \eqref{eq:main_with_t}  can be transformed into a canonical Hamiltonian system via the change of variables 
\begin{equation}
(\eta, \psi)\mapsto (\eta, \xi),\quad \xi=\psi-\frac{\gamma}{2}\pinv\Pi_0^\perp \eta. 
\end{equation}
In this new variable pair $(\eta, \xi)$, the linearization of \eqref{eq:main_with_t} around the fixed Stokes wave is a canonical Hamiltonian system of the perturbation $(\ol{\eta}, \ol{\xi})$. 

Next, we use the symplectic  \textit{good unknown} transformation 
\begin{equation}
\label{eq:l12_eta_xi}
    \vet{l_1}{l_2}=
    \begin{bmatrix}
        1&0\\
        -B^* & 1
    \end{bmatrix}
    \vet{\ol\eta}{\ol\xi}, 
\end{equation}
so that  $(l_1,l_2)$ again satisfies a  Hamiltonian system. Recalling the definition of the original good unknown $(v_1, v_2)$ in \eqref{def:v1v2}, we obtain the relations
\begin{align}
    \label{eq:v1_l1_temp}l_1&=\ol{\eta}=v_1,\\   \label{eq:v2_l2_temp}
 l_2&=\ol\xi-B^*\ol\eta=\ol\psi-\frac{\gamma}{2}\pinv\Pi_0^\perp \ol\eta-B^*\ol\eta=v_2-\frac{\gamma}{2}\pinv\Pi_0^\perp v_1.
\end{align}
In view of the ansatz  \eqref{eq:v12_vstar12} for  $v_j$, we define 
 \bq
 S(\T_L)=\left\{f\in L^2(\T_L):~ e^{-i\mu x}f(x)~\text{is}~ 2\pi\text{-periodic}\right\},
 \eq
so that $v_j(\cdot ,t)\in S(\T_L)$ for each $t$. We also recall that $\mu=\frac{m_0}{n_0}\in (-\mez, \mez)\setminus\{0\}$ and $L=2\pi n_0$.  We have the following characterizations of $S(\T_L)$ in physical and frequency spaces. 
\begin{lemm}
    \label{lemma: STL_Fourier}
Let   $f\in L^2(\T_L)$. 

(1) $f\in S(\T_L)$ if and only if 
\bq\label{char:SL:0} 
f(x+2\pi)=e^{2\pi i \mu }f(x)\quad\forall x\in \Rr.
\eq

(2) $f\in S(\T_L)$ if and only if 
    \begin{equation}\label{char:SL}
        \forall k\notin n_0\Z+m_0,\quad \wh f^L(k)=0.
    \end{equation}
\end{lemm}
\begin{proof}
(1) The characterization \eqref{char:SL:0} equivalent to the fact that $e^{-i\mu x} f(x)$ is a $2\pi$-periodic function. 

(2) Since $L=2\pi n_0$, the $L$-periodic function $e^{-i\mu x}f$ is $2\pi$-periodic if and only if $ \wh{e^{-i\mu x}f}^L(k)= 0$ for all $k\notin n_0\Z$.  Recalling that $\mu=\frac{m_0}{n_0}$, we have $\wh{e^{-i\mu x}f}^L(k)=\wh{f}^L(k+m_0)$. Thus, $f\in S(\T_L)$ if and only if  \eqref{char:SL} holds.
\end{proof}
We recall the definition \eqref{def:FM} of Fourier multipliers $m(D_L)$. In particular, we will use the  Hilbert transform 
$\mathcal H_Lu$:
\begin{equation}
    \widehat{\mathcal{H}_L u}^L(k)=-i \sign(k) \widehat{u}^L(k),
\end{equation}
where
\bq\label{def:sign}
\sign(k)=\begin{cases}
1\quad\text{if}~k>0,\\
0\quad\text{if}~k=0,\\
-1 \quad\text{if}~k<0.
\end{cases} 
\eq
 We record some properties of $S(\T_L)$ in the next lemma.
\begin{lemm}
\label{lemma: STL} 
Let  $f\in S(\T_L)$. Then we have 

(1) $\Pi_0 f=0$. 

(2) If  $m(D_L)f\in L^2(\T_L)$ (see the definition \eqref{def:FM}), then $m(D_L)f\in S(\T_L)$. In particular, we have $\mathcal{H}f\in S(\T_L)$, and  if $\Pi_0 f=0$ then $\pinv f\in S(\T_L)$.
\end{lemm}
\begin{proof}
 (1)     Since $f$ is $L$-periodic, we have 
 \[
  \Pi_0 f=\fint _0^L f(x)dx=\fint _0^L f(x+2\pi)dx. 
        \]
   Invoking \eqref{char:SL:0} yields 
   \[
   \Pi_0 f=\fint _0^L e^{2\pi i \mu }f(x)dx=e^{2\pi i \mu }\Pi_0f.
   \]
 But  $e^{2\pi i \mu }\ne 1$ for $0<|\mu|<\mez$, so   $\Pi_0 f=0$.

    (2) The first assertion is a direct consequence of the characterization \eqref{char:SL}. The second assertion follows from the boundedness of $\mathcal{H}: L^2(\T_L)\to L^2(\T_L) $  and  $\p_x^{-1}=\frac{1}{i D_L}: \{f\in L^2(\T_L): \Pi_0 f=0\}\to H^1(\T_L)$.
\end{proof}
By virtue of Lemma \ref{lemma: STL}, since $v_1$ and $v_2$ are in $S(\T_L)$, we have $\Pi_0v_1=0 $,  $\Pi_0^\perp v_1=v_1$. Hence,  the relations \eqref{eq:v1_l1_temp}-\eqref{eq:v2_l2_temp} simplify to 
\bq
    \label{eq:vtol}l_1=v_1,\quad l_2=v_2-\frac{\gamma}{2}\pinv v_1,
\eq
and  thus  $l_1$, $l_2\in S(\T_L)$.
\subsection{Conformal mapping}
In order to straighten Stokes wave's boundary $\eta^*$ in the term  $G(\eta^*)v_1$ of $\Ls_\eps $ \eqref{eq:v_matrix}, we  follow \cite{Huy21} and use a conformal mapping. More precisely, by   \cite[Proposition 3.3]{Huy21}, there exists a holomorphic bijection $z_1(x, y)+i z_2(x, y)$ from $\mathbb{R}_{-}^2=\{(x, y) \in\mathbb{R}^2: y<0\}$ onto $\left\{(x, y) \in \mathbb{R}^2: y<\eta^*(x)\right\}$, where $z_1(x+2 \pi, y)=2 \pi+z_1(x, y)$ and $z_2(x+2 \pi, y)=z_2(x, y)$ for all $(x, y) \in \mathbb{R}_{-}^2$, $z_1$ is odd in $x$, and $z_2$ is even in $x$. Then, we define the Riemann stretch 
\begin{equation}
    \zeta(x)=z_1(x,0),
\end{equation}
and the operators 
\begin{equation}
    \zeta_\sharp f=f\circ \zeta ,  \qquad   \zeta_*f=\zeta' (f\circ \zeta),
\end{equation}
so that  
\bq\label{identity:zetaop}
\zeta_*\p_xf=\p_x(\zeta_\sharp f).
\eq
 We have  
\bq\label{props:zeta}
\zeta(x)-x~\text{ is~odd,~} 2\pi\text{-periodic}~\text{and}~\zeta(x)-x=O(\eps).
\eq
 By  \cite[Lemma 3.4]{Huy21},  if  $f\in H^1(\T_L)$ then 
\begin{equation}\label{formula:GH}
    G\left(\eta^*\right) f=\partial_x\left(\zeta_{\sharp}^{-1} \mathcal{H}_L\left(\zeta_{\sharp} f\right)\right)
\end{equation}
which in conjunction with \eqref{identity:zetaop} yields 
\bq\label{identity:zetaG}
\zeta_*   G\left(\eta^*\right) f= |D_L|(\zeta_\sharp f).
\eq
\begin{lemm} \label{lemma:zeta}
   (1) If  $f\in L^2(\T_L)$, then
    \begin{equation}\label{Pi0zeta:1}
        \Pi_0 \zeta_*f=\Pi_0 f.
    \end{equation}
  If $g\in L^2(\T_L)$ and  $\Pi_0 g=0$, then 
    \begin{equation}\label{Pi0zeta:2}
        \pinv\zeta_*g=\Pi_0^\perp \zeta_\sharp\pinv g.
    \end{equation}

    (2) If $f(x)\in S(\T_L)$, then $\zeta_*f,\zeta_\sharp f\in S(\T_L)$. Conversely,  if $g\in L^2(\T_L)$ satisfies $\zeta_*g\in S(\T_L)$ or $\zeta_\sharp g\in S(\T_L)$, then $g\in S(\T_L)$. 
\end{lemm}
\begin{proof}
   (1) Since $\zeta(x)-x$ is $2\pi$-periodic, it is also $L$-periodic, and hence 
    \begin{align*}
        \Pi_0 f= \fint_0^L f(x)dx=\fint_0^L f(\zeta(x'))\zeta'(x')dx' =\Pi_0(\zeta_*f),
    \end{align*}
    proving \eqref{Pi0zeta:1}. 
   
   Next,  we assume that $g\in L^2(\T_L)$ and $\Pi_0 g=0$. Then  $f:=\pinv g \in L^2(\T_L)$, $\p_x f=g$, and hence 
 \eqref{identity:zetaop} implies
  \bq\label{Pi0zeta:10}
  \zeta_*g=\p_x\zeta_\sharp\pinv g.
  \eq
   Since $\Pi_0 \zeta_*g=\Pi_0 g=0$ in view of \eqref{Pi0zeta:1}, taking $\pinv$ of \eqref{Pi0zeta:10} yields \eqref{Pi0zeta:2}. 
   
    (2) can be verified using the characterization \eqref{char:SL:0} of $S(\T_L)$, that $\zeta(x)-x$ is $2\pi$-periodic, and that $\zeta: \Rr\to \Rr$ is a bijection. 
\end{proof}
Using the above conformal mapping, we make change of variables 
\bq
    \label{eq:ltoW} w_1=\zeta_* l_1,\quad w_2=\zeta_\sharp l_2.
\eq
It can be verified directly that the preceding transformation is symplectic,  so that $(w_1, w_2)$ again satisfies a Hamiltonian system. Moreover, $w_1$, $w_2\in S(\T_L)$ by Lemma \ref{lemma:zeta} (2).

Combining \eqref{eq:ltoW} with  \eqref{eq:vtol}, we find 
\bq\label{eq:vtoW}
w_1=\zeta_* v_1,\quad w_2=\zeta_\sharp v_2-\frac{\gamma}{2}\pinv \zeta_* v_1.
\eq
To arrive at the above formula for $w_2$, we have applied Lemmas \ref{lemma: STL} and \ref{lemma:zeta} to have 
\begin{equation*}
    v_1\in S(\T_L)\implies \zeta_\sharp \pinv v_1\in S(\T_L)\implies \zeta_\sharp \pinv v_1=\Pi_0^\perp \zeta_\sharp \pinv v_1=\pinv \zeta_* v_1.
\end{equation*}
Next, we use the relations in \eqref{eq:vtoW} and the system \eqref{eq:v1}-\eqref{eq:v2} satisfied by $(v_1, v_2)$ to obtain the Hamiltonian system satisfied by $(w_1, w_2)$. To this end, we define the functions
\begin{equation}\label{def:pq}
    p=\frac{c^*-\zeta_\sharp V^*+\gamma\zeta_\sharp \eta^*}{\zeta'},\qquad q=-p\p_x(\zeta_\sharp B^*)
\end{equation}
which depend only on the fixed Stokes wave. Then, for any functions $f_1,f_2\in H^1(\T_L)$, invoking the identities \eqref{identity:zetaop}, \eqref{identity:zetaG}, $\zeta_\sharp\Pi_0=\Pi_0$,  \eqref{formula:GH} and \eqref{dinverse}, we deduce 
\begin{equation}
\label{identity:zeta}
    \begin{aligned}
        &\quad \zeta_*\Big(\p_x\big((c^*-V^*+\gamma\eta^*)f_1\big)+G(\eta^*)f_2\Big)=\p_x\big(p(x)\zeta_*f_1\big)+|D_L|(\zeta_\sharp f_2),\\
        &\quad \zeta_\sharp\Big(-\big(1+(V^*-c^*-\gamma\eta^*)\p_xB^*\big)f_1+\gamma \Pi_0(V^* f_1)+(c^*-V^*+\gamma\eta^*)\p_xf_2+\gamma\p_x^{-1}G(\eta^*)f_2\Big)\\
        &=-\frac{1+q(x)}{\zeta'(x)}\zeta_*f_1+\gamma \Pi_0(V^* f_1)+p(x)\p_x(\zeta_\sharp f_2)+\gamma\Big(\mathcal{H}_L(\zeta_\sharp f_2)-\Pi_0\zeta_\sharp^{-1}\mathcal{H}_L(\zeta_\sharp f_2)\Big).
    \end{aligned}
\end{equation}
We note in particular that 
\begin{align*}
\zeta_\sharp \p_x^{-1}G(\eta^*)f_2&=\zeta_\sharp \p_x^{-1}\partial_x\left(\zeta_{\sharp}^{-1} \mathcal{H}_L\left(\zeta_{\sharp} f_2\right)\right)=\zeta_\sharp\Pi_0^\perp  \left(\zeta_{\sharp}^{-1} \mathcal{H}_L(\zeta_{\sharp} f_2)\right)\\
&=\mathcal{H}_L(\zeta_{\sharp} f_2)-\Pi_0\zeta_{\sharp}^{-1} \mathcal{H}_L(\zeta_{\sharp} f_2).
\end{align*}
Now we apply $\zeta_*$ to \eqref{eq:v1} and $\zeta_\sharp$ to \eqref{eq:v2} and invoke the identities in \eqref{identity:zeta}. Additionally,   Lemmas \ref{lemma: STL} and \ref{lemma:zeta} imply that  $\Pi_0(V^* v_1)=0$ and $\Pi_0\zeta_\sharp^{-1}\mathcal{H}_L(\zeta_\sharp v_2)=0$ since $V^*v_1\in S(\T_L)$ and $\zeta_\sharp^{-1}\mathcal{H}_L(\zeta_\sharp v_2)\in S(\T_L)$.  It follows that 
\begin{equation*}
    \begin{aligned}
        \p_t \zeta_*v_1
        &=\p_x (p(x)\zeta_*v_1)+|D_L|(\zeta_\sharp v_2),\\
        \p_t\zeta_\sharp v_2 
        &=-\frac{1+q(x)}{\zeta'(x)}\zeta_*v_1+p(x)\p_x(\zeta_\sharp v_2)+\gamma\mathcal{H}_L(\zeta_\sharp v_2).
    \end{aligned}
\end{equation*}
Then invoking the relations in \eqref{eq:vtoW} yields
\begin{align*}
    \p_t w_1&=\p_x(pw_1)+|D_L|\Big(w_2+\frac{\gamma}{2}\pinv w_1\Big),\\
\end{align*}
and 
\begin{align*}
    \p_t \Big(w_2+\frac{\gamma}{2}\pinv w_1\Big)&=-\frac{1+q}{\zeta'}w_1+p\p_x\Big(w_2+\frac{\gamma}{2}\pinv w_1\Big)+\gamma\mathcal{H}_L\Big(w_2+\frac{\gamma}{2}\pinv w_1\Big).
\end{align*}
These equations together imply 
\begin{align*}
    \p_t w_2
    &=-\frac{1+q}{\zeta'}w_1+p\p_xw_2+\frac{\gamma}{2}p\p_x\p_x^{-1}w_1+\gamma \mathcal{H}_L\Big(w_2+\frac{\gamma}{2}\pinv w_1\Big)\\
    &\qquad-\frac{\gamma}{2}\p_x^{-1}\p_x\Bigg(pw_1+\mathcal{H}_L\Big(w_2+\frac{\gamma}{2}\pinv w_1\Big)\Bigg)\\
    &=-\frac{1+q}{\zeta'}w_1+p\p_xw_2+\frac{\gamma}{2}\mathcal{H}_L\Big(w_2+\frac{\gamma}{2}\pinv w_1\Big).
\end{align*}
Where we have used the fact that $pw_1\in S(\T_L)$ and $\Pi_0\cH_L=0$, so that 
\[
\pinv \p_x(pw_1)=\Pi_0^\perp(p w_1)=pw_1,\quad \p_x^{-1}\p_x \cH_L=\cH_L.
\]
Thus, $(w_1, w_2)$ satisfies the Hamiltonian system
\begin{equation}\label{sys:W}
    \begin{aligned}
        \p_t\begin{bmatrix}
            w_1\\
            w_2
        \end{bmatrix}=\begin{bmatrix}
            \p_x\circ p+\frac{\gamma}{2}\mathcal{H}_L & |D_L|\\
            -\frac{1+q(x)}{\zeta^{\prime}(x)}+\frac{\gamma^2}{4}\pinv \mathcal{H}_L& p\p_x+\frac{\gamma}{2}\mathcal{H}_L
        \end{bmatrix}\begin{bmatrix}
            w_1\\
            w_2
        \end{bmatrix}.
    \end{aligned}
\end{equation}
The system \eqref{sys:W} is  simpler than the $v$-system \eqref{eq:v1}-\eqref{eq:v2} in that the genuine pseudo-differential operator $G(\eta_*)$  is transformed  to the Fourier multiplier $|D_L|$. The next lemma provides properties of the variable coefficients $p$, $q$, and $\zeta'$ of \eqref{sys:W}. 
\begin{lemm}\label{lemm:pq}
$\zeta$ is an odd function, $p$ and $q$ are even functions,  and the following  expansions hold:
\bq\label{expand:zeta}
    \zeta(x)=x+\eps\sin(x)+\eps^2\frac{1+2c_0^2+c_0^4}{4c_0^4}\sin(2x)+\rem 3,
\eq
\bq\label{expand:p}
    p=c_0-\eps\frac{\left(c_0^2+1\right)   \cos (x)}{c_0}+\eps ^2\frac{\left(c_0^2+1\right)  \Big(3 c_0^4+2 c_0^2+2 \left(c_0^4-4 c_0^2-1\right) \cos (2 x)+1\Big)}{8 c_0^5}+\rem 3,
    \eq
    \bq\label{expand:q}
     q=-\eps c_0^2   \cos (x)+\eps ^2\frac{\left(c_0^2+1\right)  \Big(\left(c_0^2-2\right) \cos (2 x)+c_0^2\Big)}{2 c_0^2}+\rem 3,
     \eq
     \bq\label{expand:qzeta}
  \frac{1+q}{\zeta'}=  1-\eps\left(c_0^2+1\right)   \cos (x)+\eps ^2 \left(c_0^2+1+\left(-\frac{1}{2 c_0^4}+c_0^2-\frac{2}{c_0^2}-\frac{1}{2}\right) \cos (2 x)\right)+\rem3.
\eq
\end{lemm} 
\begin{proof}
 $\zeta$ is odd because $\zeta(x)=z_1(x, 0)$ and $z_1$ is odd in $x$. Since $\eta^*$ is even and $\psi^*$ and $\zeta$ are odd, we deduce from \eqref{def:B_and_V} that  $B^*$ is odd and $V^*$ is even, and hence $p$, $q$, and $\zeta'$ are even.  We now proceed to expand these functions in powers of $\eps$. By \cite{Huy21} (3.15), $\zeta$ satisfies 
\begin{equation}
\label{eq:zeta_bootstrap}
    \zeta(x)=x-\frac{i}{2\pi}\sum_{k\ne 0}e^{ikx}\sign(k)\widehat{\eta\circ\zeta}(k).
\end{equation}
But  $\eta(x)=O(\eps)$, so $\zeta$ assumes the expansion 
\begin{equation}
    \zeta(x)=x+\eps\zeta_1+\eps^2\zeta_2+\rem 3.
\end{equation}
Writing $\eta=\eps \eta_1+\eps \eta_2+O(\eps^3)$, we find 
\begin{equation*}
    \eta(\zeta(x))=\eps\eta_1+\eps^2\big((\p_x\eta_1)\zeta_1+\eta_2\big)+\rem 3.
\end{equation*}
Combining this with \eqref{expand:Stokes_without_g} and \eqref{eq:zeta_bootstrap} yields $ \zeta_1=\sin x$. 
Then 
\begin{equation*}
   (\p_x\eta_1)\zeta_1+\eta_2=-\frac{1-\cos(2x)}{2}+\frac{1+2c_0^2-c_0^4}{4c_0^4}\cos(2x),
\end{equation*}
and hence
\begin{equation*}
    \zeta_2(x)=-i\sign(D)\big((\p_x\eta_1)\zeta_1+\eta_2\big)=\frac{1+2c_0^2+c_0^4}{4c_0^4}\sin(2x).
\end{equation*}
We have proven \eqref{expand:zeta}.  

Next, we use the formula  \eqref{def:B_and_V} for $B^*$, $V^*$ and  the expansion  \eqref{eq:expansion:G_order_eps_3}  for $G(\eta^*)\psi^*$ to obtain 
\begin{align*}
    B^*&=\eps c_0\sin x+\eps^2 \frac{1+c_0^2-c_0^4}{2 c_0^3} \sin (2 x)+\rem 3,\\
    V^*&=\eps c_0 \cos x+\eps^2\left(\frac{c_0}{2}+\frac{1+c_0^2-c_0^4}{2 c_0^3} \cos (2 x)\right)+\rem 3,
\end{align*}
After that, recalling the definition of \eqref{def:pq} of $p$ and $q$, we arrive at the expansions \eqref{expand:p}, \eqref{expand:q}, and \eqref{expand:qzeta}.
\end{proof}
\subsection{A desingularizing transformation} 
In view of the $v_j$-to-$w_j$ relations in \eqref{eq:vtoW}, the time dependence of the original ansatz \eqref{eq:v12_vstar12} for $v_j$, and the fact that $w_j\in S(\T_L)$, we have 
\begin{align}
      \label{eq:q1_W1}w_j=e^{\lambda t}e^{i\mu x}q_j(x),
\end{align}
where $q_j:\T\to \Cc$ are time independent.  Then, using  \eqref{sys:W} and  \eqref{eq:Bloch_general}, we arrive at the eigenvalue problem
\begin{equation}
    \label{eq:eigen_Q} \lambda Q=\Ls^\sharp_{\mu,\eps}Q,\quad Q=
    \begin{bmatrix}
        q_1\\
        q_2
    \end{bmatrix},
\end{equation}
\begin{equation}
        \Ls^\sharp_{\mu, \eps}:=
\begin{bmatrix}
p(i\mu+\p_x) +\p_x p+\frac{\gamma}{2}\big(-i\sign(D+\mu)\big) & |D+\mu|\\
-\frac{1+q}{\zeta'}+\frac{\gamma^2}{4}(i\mu+\p_x)^{-1}\big(-i\sign(D+\mu)\big)
&p(i\mu+\p_x)+\frac{\gamma}{2}\big(-i\sign(D+\mu)\big)
\end{bmatrix},
\end{equation}
where $\mu$ is  a small nonzero rational number  as in \eqref{choice:mu}. We observe that the preceding formula for  $\Ls^\sharp_{\mu, \eps}$ is well-defined for all $\mu \in (-\mez, \mez)\setminus\{0\}$ as a bounded operator from $(H^1(\T))^2$ to $(L^2(\T))^2$. Therefore, we  extend $ \Ls^\sharp_{\mu, \eps}$ to all $\mu \in (-\mez, \mez)\setminus\{0\}$. The extended $\Ls^\sharp_{\mu, \eps}$ is Hamiltonian:
\begin{align*}
    \Ls^\sharp_{\mu,\eps}&=J\mathscr B^\sharp,\\
    \mathscr B^\sharp
    &=\begin{bmatrix}
        \frac{1+q}{\zeta'}-\frac{\gamma^2}{4}(i\mu+\p_x)^{-1}\big(-i\sign(D+\mu)\big)&-p(i\mu+\p_x)-\frac{\gamma}{2}\big(-i\sign(D+\mu)\big)\\
        p(i\mu+\p_x) +\p_x p+\frac{\gamma}{2}\big(-i\sign(D+\mu)\big)&|D+\mu|
    \end{bmatrix}.
\end{align*}
Moreover, since $\text{spec}_{(L^2(\T))^2}(\Ls^\sharp_{-\mu, \eps})$ is the conjugation of $\text{spec}_{(L^2(\T))^2}(\Ls^\sharp_{\mu, \eps})$, we will only consider $\mu\in(0,\frac{1}{2})$. 

Although the inverse $(i\mu +\p_x)^{-1}$ in the $(2, 1)$-entry of  $\Ls^\sharp_{\mu, \eps}$ is well-defined for $\mu\in (-\mez, \mez)\setminus\{0\}$, it is unbounded as $\mu\to 0$ since $(i\mu +\p_x)^{-1}1=\frac{1}{i\mu}$. This obstructs any analytic continuation from $\mu=0$. To circumvent this, we remove $(i\mu +\p_x)^{-1}$ using the singular symplectic transform 
\bq\label{def:A}
A=\begin{bmatrix}
    |D+\mu|^{1/2}&0\\
    0&|D+\mu|^{-1/2}
\end{bmatrix}
\eq
and the change of parameters
\bq
\nu =\sqrt{\mu}\in \left(0, \frac{1}{\sqrt{2}}\right).
\eq
Then, it follows from \eqref{eq:eigen_Q} that 
\bq\label{eq:QtoU}
U:=A^{-1}Q 
\eq
solves 
\begin{equation}
\label{eq:eigen_U}
    \lambda U=\L_{\nu,\eps}U,
\end{equation}
where 
\begin{align*}
    \L_{\nu,\eps}&:=A^{-1}\Ls^\sharp_{\mu,\eps}A\\
    &=\scalebox{0.8}{$\displaystyle\begin{bmatrix}
|D+\nu^2|^{-1/2} (i\nu^2+\p_x)\circ p |D+\nu^2|^{1/2}+\frac{\gamma}{2}\big(-i\sign(D+\nu^2)\big) & 1\\
-|D+\nu^2|^{1/2}\circ\frac{1+q}{\zeta'}\circ |D+\nu^2|^{1/2}-\frac{\gamma^2}{4}
&|D+\nu^2|^{1/2}\circ p(i\nu^2+\p_x)|D+\nu^2|^{-1/2}+\frac{\gamma}{2}\big(-i\sign(D+\nu^2)\big)
\end{bmatrix}$}\\
&=\begin{bmatrix}
    i\sign(D+\nu^2)\Big(|D+\nu^2|^{1/2} \circ p|D+\nu^2|^{1/2}-\frac{\gamma}{2}\Big)& 1\\
    -|D+\nu^2|^{1/2}\circ\frac{1+q}{\zeta'} |D+\nu^2|^{1/2}-\frac{\gamma^2}{4} & \Big(|D+\nu^2|^{1/2}\circ p |D+\nu^2|^{1/2}-\frac{\gamma}{2}\Big)i\sign(D+\nu^2)
\end{bmatrix}.
\end{align*}
After the desingularizing transformation $A$, $\L_{\nu,\eps}$ is well-defined for $\nu\in [0, \frac{1}{\sqrt{2}})$ and is bounded as $\nu \to 0^+$.  This resolves the aforementioned unboundedness issue of  $\Ls^\sharp_{\mu,\eps}$ as $\mu\to 0$. However, $\L_{\nu,\eps}$ is still discontinuous at $\nu=0$ since  $\sign(D+\nu^2)1=\sign(\nu^2)=1$ while $\sign(D)1=0$.   On the other hand, since our goal is to obtain the spectrum of $\L_{\nu,\eps}$ for $\nu>0$, we can modify the definition of $\sign(0)$ so as to make $\L_{\nu,\eps}$ right continuous at $\nu=0$. This amounts to replace the $\sign$ function in  $\L_{\nu,\eps}$ with 
\begin{equation}\label{def:signplus}
    \signp(X):=\begin{cases}
        1,&X\ge 0\\
        -1,&X<0
    \end{cases}.
\end{equation}
We have $\sign^+(k+\nu^2)=\sign^+(k)$ for all $k\in \mathbb{Z}$ and $\nu\in [0, \frac{1}{\sqrt{2}})$. We thus consider the modified operator
\begin{equation}\label{def:Lplus}
    \L_{\nu,\eps}^+:=\begin{bmatrix}
    i\sign^+(D)\Big(|D+\nu^2|^{1/2} \circ p|D+\nu^2|^{1/2}-\frac{\gamma}{2}\Big)& 1\\
    -|D+\nu^2|^{1/2}\circ\frac{1+q}{\zeta'} |D+\nu^2|^{1/2}-\frac{\gamma^2}{4} & \Big(|D+\nu^2|^{1/2}\circ p |D+\nu^2|^{1/2}-\frac{\gamma}{2}\Big)i\sign^+(D)
\end{bmatrix}
\end{equation}
defined for all $\nu\in [0, \frac{1}{\sqrt{2}})$, and we have by design that 
\bq\label{LLplus}
 \L_{\nu,\eps}^+= \L_{\nu,\eps}\quad\forall \nu\in (0, \frac{1}{\sqrt{2}}).
\eq
For the purpose of analyticity in a full neighborhood of $\nu=0$, we extend $|D+\nu^2|^\mez$ for $\nu\ge 0$ to 
\begin{equation}
    \label{def:S_nu}
    S_\nu:=\nu\Pi_0+|D+\nu^2|^{1/2}\Pi_0^\perp,\quad \nu\in(-\frac{1}{\sqrt2},\frac{1}{\sqrt2}).
\end{equation}
Indeed, when $\nu\ge 0$, we have $|D+\nu^2|^{1/2}=|\nu|\Pi_0+|D+\nu^2|^{1/2}\Pi_0^\perp=S_\nu$. Moreover,  
\bq
S_\nu ^2=\nu^2\Pi_0+|D+\nu^2|\Pi_0^\perp=|D+\nu^2|, \quad \nu\in(-\frac{1}{\sqrt2},\frac{1}{\sqrt2}). 
\eq
We then accordingly extend $ \L_{\nu,\eps}^+$ to the following operator, which is still denoted by $ \L_{\nu,\eps}^+$: 
\begin{equation}
\label{def:Lplus_new}
        \L_{\nu,\eps}^+:=\begin{bmatrix}
    i\sign^+(D)\Big(S_\nu  \circ pS_\nu -\frac{\gamma}{2}\Big)& 1\\
    -S_\nu \circ\frac{1+q}{\zeta'} S_\nu -\frac{\gamma^2}{4} & \Big(S_\nu \circ p S_\nu -\frac{\gamma}{2}\Big)i\sign^+(D)
\end{bmatrix},\quad \nu\in(-\frac{1}{\sqrt2},\frac{1}{\sqrt2}).
\end{equation}
$ \L^+_{\nu,\eps}: (H^1(\T))^2\to (L^2(\T))^2$ is again Hamiltonian:
\begin{equation}
    \label{def:cB}
    \begin{aligned}
        \L^+_{\nu,\eps}&=J\B_{\nu,\eps},\\
        \B_{\nu,\eps}&=\begin{bmatrix}
            S_\nu \circ\frac{1+q}{\zeta'} S_\nu +\frac{\gamma^2}{4} & -\Big(S_\nu \circ p S_\nu -\frac{\gamma}{2}\Big)i\sign^+(D)\\
            i\sign^+(D)\Big(S_\nu  \circ pS_\nu -\frac{\gamma}{2}\Big)& 1
        \end{bmatrix}.
    \end{aligned}
\end{equation}
In the next proposition, we verify  that the modified operator  $\L^+_{\nu,\eps}$ is indeed analytic in $\nu\in (-\frac{1}{\sqrt2},\frac{1}{\sqrt2})$.
\begin{prop}
\label{prop:L+}
(1) $\L^+_{\nu,\eps}$  is reversible in the sense that
    \begin{equation}
    \label{eq:def_rev}
        \newRev\L^+_{\nu,\eps}=-\L^+_{\nu,\eps}\newRev,
    \end{equation}
    \begin{equation}
    \label{eq:def_rev_R}
        \newRev\begin{bmatrix}
            u_1(x)\\
            u_2(x)
        \end{bmatrix}:=\begin{bmatrix}
            -\ol u_1(-x)\\
            \ol u_2(-x)
        \end{bmatrix}.
    \end{equation}

    (2) There exists  $\eps_0=\eps_0(\gamma)>0$ such that $\L^+_{\nu,\eps}: (H^1(\T))^2\to (L^2(\T))^2$ is analytic in $\nu,\eps\in (-\frac{1}{\sqrt2},\frac{1}{\sqrt2})\times (-\eps_0,\eps_0)$.
\end{prop}
\begin{proof}

(1) We define $\rev f(x)=\ol f(-x)$, so that    
\begin{equation}
    \label{eq:R_rev}
        \newRev\begin{bmatrix}
            u_1\\
            u_2
        \end{bmatrix}=\begin{bmatrix}
            -\rev u_1\\
            \rev u_2
        \end{bmatrix}.
    \end{equation}
    Since $\widehat{\rev f}(k)=\overline{\hat{f}(k)}$, we have $a(D)\rev f=\rev(a(D)f)$ for any real symbol $a$. In addition, $\rev(if)=-i\rev f$  and  $\rev(fg)=g\rev f$ if $g$ is real-valued. Using these properties of $\rev$, one can readily verify \eqref{eq:def_rev}. 

(2) Clearly $\L^+_{\nu,\eps}: (H^1(\T))^2\to (L^2(\T))^2$ is bounded. We observe that entries of $\L^+_{\nu,\eps}$ are compositions of Fourier multipliers and multiplications by  analytic-in-$\eps$ functions $p(x)$ and $\frac{1+q}{\zeta'}$. Thus, for the analyticity in $(\nu, \eps)$, it remains to verify that the $\nu$-dependent Fourier multiplier $ S_\nu $ is analytic in $\nu\in (-\frac{1}{\sqrt2},\frac{1}{\sqrt2})$.  To this end, we use the series expansion
\[
    \label{eq:power_series_sqrt_1+x}
    (1+X)^{1/2}=\sum_{j=0}^\infty a_jX^j, \quad |X|<1,
\]
to have
\begin{equation}
    \label{eq:expansion_sqrtOp_temp}
    \begin{aligned}
        S_\nu&=\nu\Pi_0+|D+\nu^2|^{1/2}\Pi_0^\perp=\nu\Pi_0+(1+\frac{\nu^2}{D})^{1/2}|D|^{1/2} \Pi_0^\perp\\
    &=\nu\Pi_0+\Big(\sum_{j=0}^\infty a_j\nu^{2j}D^{-j}\Big)|D|^{1/2} \Pi_0^\perp,\quad \nu\in (-\frac{1}{\sqrt2},\frac{1}{\sqrt2}).
    \end{aligned}
\end{equation}
Here, the series in \eqref{eq:expansion_sqrtOp_temp} converges in $\cL(H^{s+\mez}(\T),  H^s(\T))$ for all $s\in \Rr$ because $D^{-j} |D|^{1/2} \Pi_0^\perp\in \cL(H^{s+\mez}(\T),  H^s(\T))$ for all $j\ge 0$.  
    \end{proof}
%%%%%% % %%%%%%%
\section{Kato's perturbed basis}
Starting from this section, we always consider $\gamma \ne 0$. 
\subsection{Eigenvalues of $\L_{0,\eps}^+$}\label{sec:L+0}
In this subsection we describe eigenvalues near $0$ of $\L_{0,\eps}^+$. When $\eps=0$, we have the Fourier multiplier 
\begin{equation}
    \label{eq:L00+_op}
    \L_{0,0}^+=\begin{bmatrix}
    i\sign^+(D)\big(c_0|D|-\frac{\gamma}{2}\big) & 1\\
    -|D|-\frac{\gamma^2}{4}& i\sign^+(D)\big(c_0|D|-\frac{\gamma}{2}\big)
\end{bmatrix}.
\end{equation}
The characteristic equation is 
\begin{equation}
\label{eq:lambda}
    \begin{aligned}
    0&=\det(\ld I- \L_{0,0}^+)\\
    &= \lambda^2+i\lambda\bigg(-2c_0k+\sign^+(k)(c_0-\frac{1}{c_0})\bigg)-c_0^2(k^2-|k|),
\end{aligned}
\end{equation}
where we substituted $\gamma=c_0-\frac{1}{c_0}$. The spectrum of $\L_{0,0}^+$ consists of the purely imaginary eigenvalues 
\begin{equation}
\label{eq:omega_k_pm}
\begin{aligned}
    \ld=i\omega_k^\pm&=i\left\{\sign^+(k)\big(c_0|k|-\frac{\gamma}{2}\big)\pm\sqrt{|k|+\frac{\gamma^2}{4}}\right\}.
    \end{aligned}
\end{equation}
We have 
\begin{equation}
    \label{eq:values_of_omega_k}
    \begin{aligned}
      &  \omega_{0}^{\sign(\gamma)}=0,\quad \omega_{0}^{-\sign(\gamma)}=-\gamma\ne 0,\\
      & \omega_{1}^+=c_0+\frac{1}{c_0},\quad  \omega_{1}^-=0,\\ 
    & \omega_{-1}^+=0, \quad \omega_{-1}^-=-(c_0+\frac{1}{c_0}).
    \end{aligned}
\end{equation}
On the other hand, arguing as in the proof of Lemma \ref{lemm:0ev}, we find that any root of \eqref{eq:lambda} with $|k|\ge 2$  is bounded below by 
\[
|\lambda|\ge \min\left\{1,\frac{c_0^2}{1+3c_0+\frac{1}{c_0}}\right\}.
\]
In conjunction with \eqref{eq:values_of_omega_k}, this implies that $0= \omega_{0}^{\sign(\gamma)}=\omega_{1}^-=\omega_{-1}^+ $ is an isolated eigenvalue of $\L_{0,0}^+$ with algebraic multiplicity $3$.  By continuity, for small $\eps$, the spectrum of  $\L_{0,\eps}^+$  is separated into two parts, where the part near $0$ has a total dimension of  generalized eigenspaces equals  $3$. We will find three independent (generalized) eigenvectors of  $\L_{0,\eps}^+$ from the four (generalized) eigenvectors of $\Ls_\eps$, given in Theorem \ref{thm:eigenvectors_R}.  To this end, we recall that for {\it nonzero} $\mu$ as in \eqref{choice:mu}, $\L_{\nu, \eps}$ (and hence $\L_{\nu, \eps}^+$) is obtained from $\Ls_\eps$ through the transformations  \eqref{eq:v1_l1_temp}-\eqref{eq:v2_l2_temp}, \eqref{eq:ltoW}, and \eqref{eq:QtoU}. These transformations are not well-defined when $\mu=0$. Nevertheless, by restricting the domain, we can directly verify the relation between $\Ls_\eps $ and $\L_{0,\eps}^+$ in the next lemma. 
    \begin{lemm}\label{prop:convert_RU}
 We define       
  \begin{equation} \Pi_{01}\begin{bmatrix}
            r_1\\
            r_2
        \end{bmatrix}=\Pi_0r_1
    \end{equation} 
    and 
    \begin{equation}
\label{def:B0}
    B_0=\begin{bmatrix}
        |D|^{-\mez}\zeta_*&0\\
        -\frac{\gamma}{2}|D|^\mez \pinv \zeta_* & |D|^\mez \zeta_\sharp 
    \end{bmatrix}.
\end{equation}
For any $R\in (H^1(\T))^2$, if $\Pi_{01}R=0$, then $B_0R$ is well-defined  and 
\begin{equation}
    \label{RtoU}
    B_0\Ls_\eps R=\L_{0,\eps}B_0R=\L^+_{0,\eps}B_0R.
\end{equation}
\end{lemm}
\begin{proof}
Assume $\Pi_{01}R=0$. Then \eqref{Pi0zeta:1} implies $\Pi_0\zeta_*r_1=0$, and hence $B_0R$ is well-defined. Clearly  $ \Pi_{01}B_0R=\Pi_0   |D|^{-\mez}\zeta_*r_1=0$. 

Using the identity \eqref{identity:zeta}, we compute 
\begin{align*}
    B_0\Ls_\eps R&=\vet{|D|^{-\mez }\Big(\p_x(p\zeta_* r_1)+|D|\zeta_\sharp r_2\Big)}
    {-\frac{\gamma}{2}|D|^\mez \pinv\Big(\p_x(p\zeta_* r_1)+|D|\zeta_\sharp r_2\Big)+|D|^\mez \Big(-\frac{1+q}{\zeta'}\zeta_*r_1+p\p_x(\zeta_\sharp r_2)+\gamma\mathcal{H}(\zeta_\sharp r_2)\Big)}\\
    &=\vet{i\sign(D)|D|^\mez (p\zeta_* r_1)+|D|^\mez \zeta_\sharp r_2}{-\frac{\gamma}{2}|D|^\mez (p\zeta_* r_1)-|D|^\mez \circ \frac{1+q}{\zeta'}\zeta_*r_1+|D|^\mez \circ p\p_x(\zeta_\sharp r_2)-\frac{\gamma}{2}i|D|^\mez \sign(D)\zeta_\sharp r_2},
\end{align*}
where the terms $\Pi_0(V^* f_1)$ and $\Pi_0\zeta_\sharp^{-1}\mathcal{H}_L(\zeta_\sharp f_2)$ are constant functions and become $0$ after $|D|^\mez$. On the other hand, $ \L_{0,\eps}B_0 R$ equals
\begin{align}\label{LB0R}
   & \begin{bmatrix}
i \sign(D)\left( |D|^{\mez} \circ p |D|^{\mez} - \frac{\gamma}{2} \right) & 1 \\ \nonumber
- |D|^{\mez} \circ \frac{1+q}{\zeta'} |D|^{\mez} - \frac{\gamma^2}{4} & \left( |D|^{\mez} \circ p |D|^{\mez} - \frac{\gamma}{2} \right) i \sign(D)
\end{bmatrix}
\vet{|D|^{-\mez}\zeta_* r_1}{-\frac{\gamma}{2}|D|^\mez \pinv \zeta_* r_1+|D|^\mez \zeta_\sharp r_2}\\ \nonumber
&=\vet{i\sign(D)|D|^\mez (p\zeta_*r_1)-\frac{\gamma}{2}i\sign(D)|D|^{-\mez}\zeta_* r_1-\frac{\gamma}{2}|D|^\mez \pinv \zeta_* r_1+|D|^\mez \zeta_\sharp r_2}{-|D|^\mez \circ \frac{1+q}{\zeta'}\zeta_*r_1-\frac{\gamma^2}{4}|D|^{-\mez}\zeta_* r_1-\frac{\gamma}{2}|D|^\mez \circ p\zeta_* r_1+\frac{\gamma^2}{4}|D|^{-\mez}\zeta_* r_1+|D|^\mez \circ p\p_x\zeta_\sharp r_2-\frac{\gamma}{2}i\sign(D)|D|^\mez \zeta_\sharp r_2}\\
&=\vet{i\sign(D)|D|^\mez (p\zeta_*r_1)+|D|^\mez \zeta_\sharp r_2}{-|D|^\mez \circ \frac{1+q}{\zeta'}\zeta_*r_1-\frac{\gamma}{2}|D|^\mez \circ p\zeta_* r_1+|D|^\mez \circ p\p_x\zeta_\sharp r_2-\frac{\gamma}{2}i\sign(D)|D|^\mez \zeta_\sharp r_2}.
\end{align}
Therefore, $B_0\Ls_\eps R= \L_{0,\eps}B_0 R$.  $\L_{0,\eps}^+B_0R$ is given by \eqref{LB0R} with $\sign^+(D)$ in place of $\sign(D)$. Since both operators are composed with $|D|^\mez$ to the right, we deduce  $\L_{0,\eps}B_0R=\L_{0,\eps}^+B_0R$. 
    \end{proof} 
    From Lemma \ref{prop:convert_RU} and Theorem \ref{thm:eigenvectors_R}, we obtain
\begin{theo}
\label{thm:eigenvectors_U}
Let  $\gamma\ne 0$. There exists $\eps_0=\eps_0(\gamma)>0$ such that for all $0<|\eps|<\eps_0$,   $0$ is an eigenvalue of $\L_{0,\eps}^+$ with algebraic multiplicity 3 and geometric multiplicity 2. 
\end{theo}
\begin{proof}
    From the formulas of the (generalized) eigenvectors of $\Ls_\eps $ in Theorem \ref{thm:eigenvectors_R}, we have
    \begin{align*}
        \Pi_{01}R_1=\Pi_{01}R_2=0,\quad \Pi_{01}R_3=\Pi_0\p_a\eta\vert_{(a, P)=(\eps, 0)},\quad  \Pi_{01}R_4=1,
    \end{align*}
Consequently, $\Pi_{01}\big(R_3-(\Pi_{01}R_3)R_4\big)=0$, and hence we can define 
  \[
        \Ush_2=B_0R_2,\quad \Ush_3=B_0\big(R_3-(\Pi_{01}R_3)R_4\big).
        \]
        Since $\Ls_\eps R_2=0$,  \eqref{RtoU} implies that $\L^+_{0,\eps}\Ush_2=0$. 

 Next,  using   \eqref{RtoU}, \eqref{eq:thmR3}, \eqref{eq:thmR4}, and the fact that $B_0R_1=0$,  we get 
    \begin{equation}
        \label{eq:thmUsh3}
        \begin{aligned}
            \L_{0,\eps}^+U_3^\sharp&=
            B_0\Ls_\eps \Big(R_3-(\Pi_{01}R_3)R_4\Big)\\
            &=B_0\Big(-\p_ac\vert_{(a, P)=(\eps, 0)}R_2-\gamma(\Pi_{01}R_3)R_2\Big)\\
            &=\Big(-\p_ac\vert_{(a, P)=(\eps, 0)}-\gamma(\Pi_{01}R_3)\Big)\Ush_2.
        \end{aligned}
    \end{equation}
     From the expansion of $\eta$ in \eqref{expand:Stokes_without_g}, we find 
        \begin{equation}
            \begin{aligned}
                \Pi_{01}R_3
               =\Pi_0\p_a\eta\vert_{(a, P)=(\eps, 0)}
                =2\eps (\frac{c_0^2-1}{2})+\rem 2.
            \end{aligned}
        \end{equation}
        Hence, the coefficient of $\Ush_2$ in \eqref{eq:thmUsh3} is 
        \begin{align*}
            &-\p_ac\vert_{(a, P)=(\eps, 0)}-\gamma(\Pi_{01}R_3)\\
            &=-2\eps (\frac{-4 c_0^8 +11 c_0^6  -7 c_0^4 +3 c_0^2  +1}{8 c_0^5  })-(c_0-\frac{1}{c_0})2\eps (\frac{c_0^2-1}{2})+\rem 2\\
            &=(2\eps)(\frac{-3c_0^6+3c_0^4-3c_0^2-1}{8c_0^5})+\rem 2. 
        \end{align*}
   Since  $-3c_0^6+3c_0^4-3c_0^2-1=-3c_0^2\left[\left(c_0^2-\frac12\right)^2+\frac34\right]-1<0$, the preceding coefficient is nonzero. Thus  $\Ush_3$ is a generalized eigenvector. 
   Finally,  it is readily seen that
        \begin{equation}
            \label{eq:thmUsh1}
            \L_{0,\eps}^+\Ush_1=0,\quad   \Ush_1:=\begin{bmatrix}
            1\\
            \frac{\gamma}{2}i
        \end{bmatrix}
        \end{equation}
      and  $\Ush_1$ is independent of $\Ush_2$ because $\Pi_0\Ush_2=0$. 
\end{proof}
\subsection{Kato's similarity transformation}
From the discussion on $\cL^+_{0, 0}$ in Section \ref{sec:L+0}, the triple eigenvalue $\ld=0$ is spanned by the vectors 
\begin{equation}
\label{basis_k=1_2}
U_{1}=
    \begin{bmatrix}
        ie^{ix}\\
        (c_0-\frac{\gamma}{2})e^{ ix}
    \end{bmatrix},\quad U_2=\begin{bmatrix}
        - ie^{-ix}\\
        (c_0-\frac{\gamma}{2})e^{ - ix}
    \end{bmatrix},\quad   U_3=\begin{bmatrix}
        i\\
        -\frac{\gamma}{2}
    \end{bmatrix}.
\end{equation}
In order to track the eigenvalues of $\cL^+_{\nu, \eps}$ near $0$, we follow \cite{Berti_Full} and use the Kato similarity transformation,
\begin{equation}
\label{def:Kato}
    \cK_{\nu,\eps} := \{1-(P_{\nu,\eps}-P_{0, 0})^2\}^{-1/2} \{P_{\nu,\eps} P_{0, 0} + (1-P_{\nu,\eps})(1-P_{0, 0})\}.
\end{equation}
Here $P_{\nu,\eps}:  (L^2(\T))^2\to (H^1(\T))^2$ is the spectral projection
\begin{equation}
    \label{def:P_spec_proj}
P_{\nu,\eps} = -\frac1{2\pi i} \oint_\Gamma (\L^+_{\nu,\eps} - \lambda)^{-1} d\lambda,
\end{equation}
where $\Gamma$ is a closed circle of sufficiently small radius around $0$ such that $\Gamma\subset \rho(\L^+_{\nu,\eps})$ for  all $(\nu,\eps)$ near $(0,0)$. By Proposition \ref{prop:L+},  $\L_{\nu,\eps}^+$ is analytic in $(\nu,\eps)$ near $(0,0)$, and hence so are   $\cK_{\nu,\eps}$ and $P_{\nu,\eps}$. 

Set 
\begin{equation}
    \cV_{\nu,\eps}:=\operatorname{Ran}(P_{\nu,\eps}),
\end{equation}
From \cite[Lemma 3.1]{Berti_Full}, we have the  following properties:
\begin{align}
    \L_{\nu,\eps}^+:\cV_{\nu,\eps}&\to\cV_{\nu,\eps},\label{prop1_berti_kato}\\
    \sigma(\L^+_{\nu,\eps})\cap\{z\in\mathbb{C} \text{ inside }\Gamma\}&=\sigma(\L^+_{\nu,\eps}|_{\cV_{\nu,\eps}}),\label{prop2_berti_kato}\\
    \cV_{\nu,\eps}&=\cK_{\nu,\eps}\cV_{0,0}.\label{prop3_berti_kato}
\end{align}
These properties allow us to obtain the generalized eigenspace of $\cL^+_{\nu, \eps}$ by applying Kato's transformation to the unperturbed generalized eigenspace 
\[
\cV_{0, 0}=\operatorname{span}\{U_1, U_2,U_3\},
\]
that is,
\begin{equation}
    \cV_{\nu,\eps}=\Span \{U^{\nu,\eps}_{1},U^{\nu,\eps}_2,U^{\nu,\eps}_3\},\quad  U^{\nu,\eps}_m:=\cK_{\nu,\eps}U_m.
\end{equation}
Consequently, the spectrum of $\L^+_{\nu,\eps}$ near origin is that of the $3\times3$ matrix $\Lmat_{\nu,\eps}$ representing $\cL^+_{\nu, \eps}$ on the $\cV_{\nu,\eps}$ basis. We note that the $U^{\nu,\eps}_m$'s are analytic in $(\nu,\eps)$ near $(0,0)$ since $\cK_{\nu, \eps}$ is so. Since  $\cL^+_{\nu, \eps}$ is Hamiltonian and reversible,  $\cK_{\nu,\eps}$ is symplectic and reversible, i.e. $\cK_{\nu,\eps}^*J\cK_{\nu,\eps}=J$ and $\newRev\cK_{\nu,\eps}=\cK_{\nu,\eps}\newRev$, where $\newRev$ is the reversal operator \eqref{eq:def_rev_R}. See \cite[Lemma 3.2 (ii)]{Berti_Full}.  Consequently, arguing as in the proof of \cite[Lemma 3.1]{Huy_transverse}, we obtain the following structural properties of $\Lmat_{\nu,\eps}$: 
\begin{prop}
\label{prop:structure_Lmat}
    (1) The $3\times 3$ matrix that represents the linear operator $\L_{\nu,\eps}^+=J\B_{\nu,\eps}:\cV_{\nu,\eps}\to \cV_{\nu,\eps}$  with respect to the basis $\{U_1^{\nu,\eps},U_{2}^{\nu,\eps},U_3^{\nu,\eps}\}$ (in this order) is 
  \begin{equation}
  \label{def:Lmat}
\Lmat_{\nu,\eps}:=\begin{bmatrix}
\displaystyle \frac{i}{2\pi \beta_1} \left( \mathcal{B}_{\nu,\eps} U_1^{\nu,\eps}, U_1^{\nu,\eps} \right) 
& \displaystyle \frac{i}{2\pi \beta_2} \left( \mathcal{B}_{\nu,\eps} U_1^{\nu,\eps}, U_2^{\nu,\eps} \right) 
& \displaystyle \frac{i}{2\pi \beta_3} \left( \mathcal{B}_{\nu,\eps} U_1^{\nu,\eps}, U_3^{\nu,\eps} \right) \\
\displaystyle \frac{i}{2\pi \beta_1} \left( \mathcal{B}_{\nu,\eps} U_2^{\nu,\eps}, U_1^{\nu,\eps} \right) 
& \displaystyle \frac{i}{2\pi \beta_2} \left( \mathcal{B}_{\nu,\eps} U_2^{\nu,\eps}, U_2^{\nu,\eps} \right) 
& \displaystyle \frac{i}{2\pi \beta_3} \left( \mathcal{B}_{\nu,\eps} U_2^{\nu,\eps}, U_3^{\nu,\eps} \right) \\
\displaystyle \frac{i}{2\pi \beta_1} \left( \mathcal{B}_{\nu,\eps} U_3^{\nu,\eps}, U_1^{\nu,\eps} \right) 
& \displaystyle \frac{i}{2\pi \beta_2} \left( \mathcal{B}_{\nu,\eps} U_3^{\nu,\eps}, U_2^{\nu,\eps} \right) 
& \displaystyle \frac{i}{2\pi \beta_3} \left( \mathcal{B}_{\nu,\eps} U_3^{\nu,\eps}, U_3^{\nu,\eps} \right)
\end{bmatrix},
\end{equation}
where the constants \begin{equation}
    \label{def:beta}
    \beta_1=-(c_0+\frac{1}{c_0}),\quad \beta_2=c_0+\frac{1}{c_0},\quad \beta_3=\gamma
\end{equation}
satisfy 
\begin{equation}
    \label{eq:JUmn_general}
   (JU_m^{\nu,\eps},U_n^{\nu,\eps})= (JU_m,U_n)=2\pi i\delta_{mn}\beta_m.
\end{equation}

(2) The entries of $\Lmat_{\nu,\eps}$ are pure imaginary, and
\begin{equation}\label{eq:symm_Lmat}
    (\Lmat_{\nu,\eps})_{jk}\beta_k=(\Lmat_{\nu,\eps})_{kj}\beta_j, \quad j,k\in\{1,2,3\}.
\end{equation}
\end{prop}
We will use the following consequence of \eqref{eq:JUmn_general}:
\begin{equation}\label{eq:expr_cV_vector_in_basis_by_inner_product|}
      U=-i \sum_{k=1}^3\frac{1}{2\pi\beta_k}(JU,U_k^{\nu,\eps})U_k^{\nu,\eps}\quad\forall U\in \cV_{\nu, \eps}.
  \end{equation}
\subsection{Expansion of the matrix $\Lmat_{\nu,\eps}$}
Since  $\mathcal{B}_{\nu,\eps}$ and  $U_m^{\nu,\eps}$ are analytic in $(\nu, \eps)$ near $(0, 0)$, so is the $3\times 3$ matrix $\Lmat_{\nu,\eps}$ given by \eqref{def:Lmat}. We will expand  $\Lmat_{\nu,\eps}$ in powers of $\eps$ and $\nu$ and tolerate error terms $\errorTol$ (see \eqref{sec:notation}):
\begin{equation}
\label{expansion:Lmat_formal}
    \Lmat_{\nu,\eps}=\sum_{\substack{
       k=0,\ j\le 5 \\
        \text{or } k=1,\ j\le 3 \\
        \text{or } k=2,\ j\le 1
    }}\Lmat^{j,k}\nu^j\eps^k+\errorTol.
\end{equation}
We begin with the following observation.
\begin{lemm}\label{lemm:Lnu0}
    \label{prop:Lmat_entry_0_for_eps_0_nu_odd} If $j$ is odd, then $\Lmat^{j,0}=0$. 
\end{lemm}
\begin{proof}
We recall that  $S_\nu $ in $\cL^+_{\nu, \eps}$ is not analytic in $\mu=\nu^2$. However,  when $\eps=0$,  we have $p\equiv c_0$ by \eqref{expand:p}, so that    
 \[
 S_\nu \circ (p|_{\eps=0})S_\nu =c_0|D+\nu^2|
 \]
is analytic in small $\mu=\nu^2$. The same argument applies to $S_\nu \circ \frac{1+q}{\zeta'}S_\nu $ since $\frac{1+q}{\zeta'}\vert_{\eps=0}\equiv 1$ by \eqref{expand:q}. Consequently, $\cL^+_{\nu, 0}$ is analytic in small $\mu=\nu^2$, and hence so is $ \Lmat_{\nu, 0}$. Therefore, in the series expansion of $ \Lmat_{\nu, 0}$, only even powers of $\nu$ appear. 
\end{proof}
By virtue of Lemma \ref{lemm:Lnu0}, in order to obtain \eqref{expansion:Lmat_formal}, we still need to carry out explicit calculations to obtain the coefficients $\Lmat^{j,k}$, where 
\[
k=0,~j\in\{ 2, 4\},~~\text{or}~~k=1,~j\in\{0, 1, 2, 3\},~~\text{or}~~ k=2,~j\in\{0, 1\}.
\]
To this end, it is convenient to introduce the notation
\begin{equation}
    \label{def:Delta} \Delta\L_{\nu,\eps}^+=\L_{\nu,\eps}^+-\L_{0,0}^+,\quad \Delta P_{\nu,\eps}=P_{\nu,\eps}-P_{0,0}.
\end{equation}
Recalling from \eqref{expand:p} that  $p(\eps)=c_0+O(\eps)$, we have 
\begin{equation}
\label{eq:sqrtOp_p_sqrtOp}
    \begin{aligned}
    S_\nu \circ pS_\nu =c_0|D+\nu^2|+O(\eps)=c_0|D|+O(\eps,\nu^2)\quad\text{in}~\cL(H^1, L^2). 
\end{aligned}
\end{equation}
Similarly, \eqref{expand:q} implies that 
\[
S_\nu \circ \frac{1+q}{\zeta'}S_\nu =|D|+O(\eps,\nu^2)\quad\text{in}~\cL(H^1, L^2),
\]
 and hence 
 \begin{equation}
    \label{error:deltaL}
    \Delta\L_{\nu,\eps}^+=\L^+_{\nu,\eps}-\L_{0,0}^+=O(\eps,\nu^2)\quad\text{in}~\cL((H^1)^2, (L^2)^2). 
\end{equation}
To expand the inverse operator in the spectral projection $P_{\nu, \eps}$ \eqref{def:P_spec_proj}, we use the Neumann series formula 
\begin{equation}
    \label{eq:Neumann_series}
    (\L^+_{\nu,\eps}-\lambda)^{-1}-(\L_{0,0}^+-\lambda)^{-1}=
    \sum_{k=1}^\infty (-1)^k(\L_{0,0}^+-\lambda)^{-1}\big(\Delta\L_{\nu,\eps}^+(\L_{0,0}^+-\lambda)^{-1}\big)^k
\end{equation}
and notice that 
\begin{equation}\label{error:Neumann}
     (-1)^k(\L_{0,0}^+-\lambda)^{-1}\big(\Delta\L_{\nu,\eps}^+(\L_{0,0}^+-\lambda)^{-1}\big)^k=\errorTol\quad \text{ for }k\ge 3.
\end{equation}
Thus we obtain 
\[
    \label{cutoff:Neumann}
    (\L^+_{\nu,\eps}-\lambda)^{-1}-(\L_{0,0}^+-\lambda)^{-1}=\cT_{\nu, \eps}(\ld)+\errorTol,
\]
where
\bq\label{def:cT}
\cT_{\nu, \eps}(\ld):=-(\L_{0,0}^+-\lambda)^{-1}\Delta\L_{\nu,\eps}^+ (\L_{0,0}^+-\lambda)^{-1}\\+(\L_{0,0}^+-\lambda)^{-1}\Delta\L_{\nu,\eps}^+ (\L_{0,0}^+-\lambda)^{-1}\Delta\L_{\nu,\eps}^+ (\L_{0,0}^+-\lambda)^{-1}.
\eq
Applying the residue theorem to \eqref{def:P_spec_proj}, we find  
\begin{equation}\label{cutoff:P}
    \begin{aligned}
        (\Delta P_{\nu,\eps})U&=-\res\bigg( (\L^+_{\nu,\eps}-\lambda)^{-1}U-(\L_{0,0}^+-\lambda)^{-1}U\bigg)\\
        &=-[\lambda^{-1}]\cT_{\nu, \eps}(\ld)+\errorTol,
    \end{aligned}
\end{equation}
where $[\lambda^{-1}]\cT_{\nu, \eps}(\ld)$ denotes the coefficient of $\ld^{-1}$ in the Laurent expansion of $\cT_{\nu, \eps}(\ld)$. We recall from  \eqref{eq:L00+_op} that
\begin{equation}
    \label{eq:L00+_op_ab}
    \L_{0,0}^+=
    \begin{bmatrix}
        a(D) &1\\
        b(D)& a(D)
    \end{bmatrix},\quad a(k):= i\sign^+(k)\big(c_0|k|-\frac{\gamma}{2}\big),\quad b(k):=-|k|-\frac{\gamma^2}{4}. 
\end{equation}
For each $k\in \Zz$, since $b(k)\ne 0$, the determinant $(a(k)-\lambda)^2-b(k)$ cannot have double root. Consequently, the smallest power of $\ld$ in the Laurent  expansion of 
\begin{equation}\label{eq:Linv_op_ab}
   (\L_{0,0}^+-\lambda)^{-1}=
    \begin{bmatrix}
        a(D)-\lambda &1\\
        b(D)& a(D)-\lambda
    \end{bmatrix}^{-1}
    =\frac{1}{(a(D)-\lambda)^2-b(D)}\begin{bmatrix}
        a(D)-\lambda &-1\\
       - b(D)& a(D)-\lambda
    \end{bmatrix}
\end{equation}
is  $\lambda^{-1}$.  Since there are at most three  copies of $(\L_{0,0}^+-\lambda)^{-1}$  in $\cT_{\nu, \eps}(\ld)$, to obtain the coefficient of $\ld^{-1}$ in $\cT_{\nu, \eps}(\ld)$, it suffices to calculate the $\ld^j$ coefficients of $(\L_{0,0}^+-\lambda)^{-1}$ for  $j\in\{-1, 0, 1\}$. 

For $j\in \{1, 2, 3\}$, we have $P_{0,0}U_j=U_j$, so that 
\bq
  \cK_{\nu,\eps}U_j=(1-(\Delta P_{\nu,\eps})^2)^{-1/2}P_{\nu,\eps}U_j=(1-(\Delta P_{\nu,\eps})^2)^{-1/2}(1+\Delta P_{\nu,\eps})U_j.
  \eq
It follows from \eqref{error:deltaL}, \eqref{def:cT}, and \eqref{cutoff:P} that $\Delta P_{\nu,\eps}=O(\eps,\nu^2)$. Consequently,
\begin{equation}
\label{expansion:cK}
    \begin{aligned}
    U_j^{\nu,\eps}&=\cK_{\nu,\eps}U_j=(1-(\Delta P_{\nu,\eps})^2)^{-1/2}(1+\Delta P_{\nu,\eps})U_j\\
    &=(1+\Delta P_{\nu,\eps}+\frac{1}{2}(\Delta P_{\nu,\eps})^2)U_j+\errorTol.
\end{aligned}
\end{equation}
Using \eqref{expansion:cK} and the aforementioned  observation about calculating $\Delta P_{\nu,\eps}$, we use Mathematica's symbolic computation to obtain the coefficients $\Lmat^{j, k}$ in \eqref{expansion:Lmat_formal}. We refer to the companion  Mathematica file \verb|Kato_basis_and_matrix_L.nb| for details. The result  is: 
\bq
\label{eq:Lmat_expand_in_nu_eps}
\begin{aligned}
\Lmat_{\nu,\eps}&=\nu^2\left[
\begin{array}{ccc}
 \frac{i c_0^3}{c_0^2+1} & 0 & 0 \\
 0 & \frac{i c_0^3}{c_0^2+1} & 0 \\
 0 & 0 & \frac{i c_0^3}{c_0^2-1} \\
\end{array}
\right]+\nu^4\left[
\begin{array}{ccc}
 \frac{i c_0^3}{\left(c_0^2+1\right)^3} & 0 & 0 \\
 0 & -\frac{i c_0^3}{\left(c_0^2+1\right)^3} & 0 \\
 0 & 0 & -\frac{i c_0^3}{\left(c_0^2-1\right)^3} \\
\end{array}
\right]\\
&\qquad+\nu\eps\left[
\begin{array}{ccc}
 0 & 0 & -\frac{i \left(c_0^2+1\right)}{2 c_0} \\
 0 & 0 & \frac{i \left(c_0^2+1\right)}{2 c_0} \\
 \frac{i \left(c_0^2-1\right)}{2 c_0} & \frac{i \left(c_0^2-1\right)}{2 c_0} & 0 \\
\end{array}
\right]\\
&\qquad+\nu^3\eps\left[
\begin{array}{ccc}
 0 & 0 & -\frac{i \left(c_0^6-3 c_0^4-c_0^2-1\right)}{4 c_0 \left(c_0^4-1\right)} \\
 0 & 0 & -\frac{i \left(c_0^6-3 c_0^4-c_0^2-1\right)}{4 c_0 \left(c_0^4-1\right)} \\
 \frac{i \left(c_0^6-3 c_0^4-c_0^2-1\right)}{4 c_0 \left(c_0^2+1\right)^2} & -\frac{i \left(c_0^6-3 c_0^4-c_0^2-1\right)}{4 c_0 \left(c_0^2+1\right)^2} & 0 \\
\end{array}
\right]\\
&\qquad+\eps^2\left[
\begin{array}{ccc}
 -\frac{i \left(3 \left(c_0^6-c_0^4+c_0^2\right)+1\right)}{8 c_0^5} & -\frac{i \left(3 \left(c_0^6-c_0^4+c_0^2\right)+1\right)}{8 c_0^5} & 0 \\
 \frac{i \left(3 \left(c_0^6-c_0^4+c_0^2\right)+1\right)}{8 c_0^5} & \frac{i \left(3 \left(c_0^6-c_0^4+c_0^2\right)+1\right)}{8 c_0^5} & 0 \\
 0 & 0 & 0 \\
\end{array}
\right]
+\errorTol.
\end{aligned}
\eq
We note in particular that 
\bq\label{eq:Lmat_expand_in_nu_eps:2}
\Lmat^{0, 1}=\Lmat^{2, 1}=\Lmat^{1, 2}=0,\quad \Lmat^{2j+1, 0}=0\quad\forall j\ge 0.
\eq 
\section{Analysis of $\Lmat_{\nu,\eps}$}
 The proof of Theorem \ref{thm:main} requires much finer analysis of $\Lmat_{\nu,\eps}$ than the finite-order expansion \eqref{eq:Lmat_expand_in_nu_eps}. In this section, we shall exploit the structure of $\cL^+_{\nu, \eps}$ and parity of variable coefficients of $\cL^+_{\nu, \eps}$ and eigenvectors to obtain {\it exact} relations and vanishing of certain entries of $\Lmat_{\nu,\eps}$. 
\subsection{Analyticity in $\mu$}
The operator $\cL^+_{\nu, \eps}$ \eqref{def:Lplus_new} fails to be analytic in $\mu=\nu^2$ precisely because of the zero frequency, i.e. $(|k|+\mu)^\mez$ when $k=0$. The goal of this subsection is to identify all entries of $\Lmat_{\nu,\eps}$ that are analytic in $\mu$. To separate the zero frequency, we decompose 
\bq
L^2(\T)=\C\oplus \rL^2(\T),\quad \rL^2(\T):=\{f\in L^2(\T):\Pi_0f=0\}.
\eq
We also denote 
\bq
\mathring{H}^m(\T):=H^m(\T)\cap \rL^2(\T).
\eq
In order to investigate the analyticity in $\mu$, we introduce 
\begin{defi}
    (1) If $X$ is a Banach space, we denote the space of $X$-valued analytic functions near $(\mu,\eps)=(0,0)$ by $\sC_{X}$. When $X=\Cc$, we simply write $\sC_{\Cc}=\sC$.  On the other hand, if $X$ and $Y$ are Banach spaces, we write  $ \sC_{\cL(X,Y)}=  \sC_{X\to Y}$. 
    
(2) Let  $X$, $Y$, $H$ be Hilbert spaces such that $X, Y\subset H$. Suppose that $\Pi : H\to H$ is a projection such that $X$ and $Y$ are invariant under $\Pi $. We denote $\Pi ^\perp=\I-\Pi $ and define
\[
\label{def:sE}
\begin{aligned}
    \sE_{X\to Y}&:=\Pi \sC_{X\to Y}\Pi +\Pi ^\perp\sC_{X\to Y}\Pi ^\perp+\nu\big(\Pi ^\perp\sC_{X\to Y}\Pi +\Pi \sC_{X\to Y}\Pi ^\perp\big)\\
&:=
    \Big\{\Pi F_1\Pi +\Pi ^\perp F_2\Pi ^\perp+\nu\big(\Pi ^\perp F_3\Pi +\Pi  F_4\Pi ^\perp\big):F_j\in\sC_{X\to Y},1\le j\le 4\Big\}.
     \end{aligned}
\]
\end{defi}
We gather properties of the space $\sE_{X\to Y}$ in the next lemma. 
\begin{lemm}
\label{lemma:sE}
    (1) An operator $F(\nu,\eps)\in \cL(X, Y)$ is in $\sE_{X\to Y}$ if and only if
    \begin{equation}
    \label{def:alter_def_of_sE}
        \Pi  F \Pi ,
        \Pi ^\perp F \Pi ^\perp \in \sC_{X\to Y}\quad \text{and }\quad
        \Pi ^\perp F \Pi ,
        \Pi  F \Pi ^\perp \in \nu \sC_{X\to Y}.
    \end{equation}

    (2) $\sE_{X\to Y}$ is a linear subspace of $L(X,Y)\{\nu,\eps\}$. If   $\{F^{(n)}\}_{n=1}^\infty\subset \sE_{X\to Y}$ converges to $F\in L(X, Y)\{\nu, \eps\}$ locally uniformly in $(\nu, \eps)$, then  $F\in \sE_{X\to Y}$. 

    (3) If $F\in\sE_{X\to Y}$ and $G\in\sE_{Y\to Z}$, then the composition $G F$ is in $\sE_{X\to Z}$.

    (4) If $F_{jk}\in\sE_{X\to Y}$ for $1\le j,k\le 2$, then 
    \[
    F:=\begin{bmatrix}
    F_{11} & F_{12} \\
    F_{21} & F_{22}
\end{bmatrix}\in\sE_{X^2\to Y^2}.
    \]

    (5) If $F(\nu,\eps)\in\sE_{X\to Y}$, then $F(0,\eps),F(0,0)\in\sE_{X\to Y}$.

    (6) If $F(\nu,\eps)\in\sE_{X\to Y}$ and $F(0,0)\in \cL(X, Y)$ is invertible, then $F(\nu,\eps)$ is invertible and $F(\nu,\eps)^{-1}\in\sE_{Y\to X}$ for $(\nu, \eps)$ near $(0, 0)$.  
\end{lemm}
\begin{proof}
    (1) Suppose that $F\in\sE_{X\to Y}$, where 
    \[
    F=\Pi F_1\Pi +\Pi ^\perp F_2\Pi ^\perp+\nu\big(\Pi ^\perp F_3\Pi +\Pi  F_4\Pi ^\perp\big),\quad F_j\in\sC_{X\to Y}. 
    \]
Since 
   \begin{align*}
\Pi  F \Pi  = \Pi  F_1 \Pi  \in \sC_{X\to Y}, \quad \Pi ^\perp F \Pi ^\perp = \Pi ^\perp F_2 \Pi ^\perp \in \sC_{X\to Y}, \\
\Pi ^\perp F \Pi  = \nu\Pi ^\perp F_3 \Pi  \in \nu\sC_{X\to Y}, \quad \Pi  F \Pi ^\perp = \nu \Pi  F_4 \Pi ^\perp \in \nu\sC_{X\to Y},
\end{align*}
it is readily seen that $F$  satisfies \eqref{def:alter_def_of_sE}.

Conversely, if $F$ satisfies \eqref{def:alter_def_of_sE}, then 
\begin{align*}
    F &= \Pi  F \Pi  + \Pi  F \Pi ^\perp + \Pi ^\perp F \Pi  + \Pi ^\perp F \Pi ^\perp \\
      &= \Pi  \big( \Pi  F \Pi  \big) \Pi 
         + \Pi  \big( \Pi  F \Pi ^\perp \big) \Pi ^\perp
         + \Pi ^\perp \big( \Pi ^\perp F \Pi  \big) \Pi 
         + \Pi ^\perp \big( \Pi ^\perp F \Pi ^\perp \big) \Pi ^\perp \\
      &\in \Pi  \sC_{X\to Y} \Pi 
         + \Pi  \big( \nu\sC_{X\to Y} \big) \Pi ^\perp
         + \Pi ^\perp \big( \nu\sC_{X\to Y} \big) \Pi 
         + \Pi ^\perp \sC_{X\to Y} \Pi ^\perp = \sE_{X\to Y}.
\end{align*}
(2) $\sE_{X\to Y}$ is a linear subspace of $L(X,Y)\{\nu,\eps\}$ because $\sC_{X\to Y} $ is so.   Suppose that $\{F^{(n)}\}_{n=1}^\infty\subset \sE_{X\to Y}$ converges to $F\in L(X, Y)\{\nu, \eps\}$ locally uniformly in $(\nu, \eps)$. From part (1),  we have
\[
\Pi  F^{(n)} \Pi ,\Pi ^\perp F^{(n)} \Pi ^\perp \in \sC_{X\to Y}, \quad
\Pi ^\perp F^{(n)} \Pi ,\Pi  F^{(n)} \Pi ^\perp \in \nu\sC_{X\to Y}.
\]
This is equivalent to 
\[
[\nu^{2j+1}](\Pi  F^{(n)} \Pi ) = 
[\nu^{2j+1}](\Pi ^\perp F^{(n)} \Pi ^\perp) =
[\nu^{2j}](\Pi ^\perp F^{(n)} \Pi ) = 
[\nu^{2j}](\Pi  F^{(n)} \Pi ^\perp) = 0~\forall j\in\Z_{\ge 0}.
\]
Since the preceding equations are preserved under the  limit $n\to \infty$, part (1) again implies  $F\in\sE_{X\to Y}$. 

(3) Let  $F\in\sE_{X\to Y}$ and $G\in\sE_{Y\to Z}$. Appealing to part (1), we deduce 
\begin{align*}
    \Pi  G F \Pi  
    &= (\Pi  G \Pi  )(\Pi  F \Pi )
       + ( \Pi  G \Pi ^\perp )(\Pi ^\perp F \Pi ) \\
    &\in (\sC_{Y\to Z}\circ\sC_{X\to Y}) + (\nu\sC_{Y\to Z}\circ\nu\sC_{X\to Y}) \subset \sC_{X\to Z},
\end{align*}
\begin{align*}
    \Pi ^\perp G F \Pi 
    &= (\Pi ^\perp G \Pi ^\perp)(\Pi ^\perp F \Pi )
       + (\Pi ^\perp G \Pi )(\Pi  F \Pi ) \\
    &\in (\sC_{Y\to Z} \circ \nu\sC_{X\to Y}) + (\nu\sC_{Y\to Z} \circ \sC_{X\to Y}) \subset \nu\sC_{X\to Z}
\end{align*}
and $\Pi  G F \Pi ^\perp\in \nu\sC_{X\to Z}$, $\Pi ^\perp G F \Pi ^\perp\in\sC_{X\to Z}$ by an analogous argument. Thus   $GF\in \sE_{X\to Z}$ by (1). 

(4) Since 
\[
\Pi \equiv \begin{bmatrix}
    \Pi  &0\\
    0& \Pi 
\end{bmatrix}\quad \text{on}~ X^2,
\]
(1) implies  
\begin{align*}
    \Pi  F\Pi &=\begin{bmatrix}
    \Pi  &0\\
    0& \Pi 
\end{bmatrix}\begin{bmatrix}
    F_{11} & F_{12} \\
    F_{21} & F_{22}
\end{bmatrix}\begin{bmatrix}
    \Pi  &0\\
    0& \Pi 
\end{bmatrix}\\
&=
\begin{bmatrix}
    \Pi  F_{11} \Pi  & \Pi  F_{12} \Pi  \\
    \Pi  F_{21} \Pi  & \Pi  F_{22} \Pi 
\end{bmatrix}\\
&\in
\begin{bmatrix}
    \sC_{X\to Y} & \sC_{X\to Y} \\
    \sC_{X\to Y} & \sC_{X\to Y}
\end{bmatrix}
\subset \sC_{X^2 \to Y^2}.
\end{align*}
Repeating this argument, we obtain that $F$ satisfies \eqref{def:alter_def_of_sE}, and hence $F\in\sE_{X^2\to Y^2}$.

(5) This is obvious from the definition \ref{def:sE}.

(6) Setting $\nu=0$ in  \eqref{def:alter_def_of_sE} yields 
\[
F(0,0)=\Pi  F(0,0)\Pi +\Pi ^\perp F(0,0)\Pi ^\perp.
\]

This implies $F(0, 0)\Pi =\Pi  F(0,0)$, whence $F(0, 0)^{-1}\Pi =\Pi  F(0, 0)^{-1}$. It follows that
\bq\label{inv:Eclass}
F(0,0)^{-1}=\Pi  F(0,0)^{-1}\Pi +\Pi ^\perp F(0,0)^{-1}\Pi ^\perp\in\sE_{Y\to X}.
\eq
For sufficiently small $\nu$ and $\eps$, we have 
\begin{equation}
    \label{eq:Neumann_F_nu_eps}
   F(\nu,\eps)^{-1}=  \sum_{k=0}^\infty (-1)^kF(0,0)^{-1}\Big[\Big(F(\nu,\eps)-F(0,0)\Big)F(0,0)^{-1}\Big]^k.
\end{equation}
Combining \eqref{inv:Eclass} with part (3), we deduce 
\[
\Big(F(\nu,\eps)-F(0,0)\Big)F(0,0)^{-1}\in \sE_{Y\to Y},\quad \Big[\Big(F(\nu,\eps)-F(0,0)\Big)F(0,0)^{-1}\Big]^k\in \sE_{Y\to Y},
\]
and hence each term in the series \eqref{eq:Neumann_F_nu_eps} is in $\sE_{Y\to X}$. Therefore, part  (2) implies $F(\nu,\eps)^{-1}\in \sE_{Y\to X}$.
\end{proof}
Next, we apply the preceding lemma to determine entries of $ \Lmat_{\nu,\eps}$ that are analytic in $\mu$. 
\begin{prop} We have
\begin{equation}
\label{eq:Lmat_in_sC}
    \Lmat_{\nu,\eps}\in\begin{bmatrix}
        \sC&\sC&\nu\sC\\
        \sC&\sC&\nu\sC\\
        \nu\sC&\nu\sC&\sC
    \end{bmatrix}.
\end{equation}
\end{prop}
\begin{proof}
It follows from the expansion  \eqref{eq:expansion_sqrtOp_temp} of $ S_\nu $ that 
\begin{equation}
\label{eq:sqrtOp_in_sC}
    S_\nu \in \nu\sC_{H^1\to H^{1/2}}\Pi_0+\sC_{H^1\to H^{1/2}}\Pi_0^\perp~
    \text{ and }~ S_\nu \in\nu\Pi_0\sC_{H^{1/2}\to L^2}+\Pi_0^\perp \sC_{H^{1/2}\to L^2}.
\end{equation}
Since $p\in C^\infty(\T)$ is independent of $\nu$,  we deduce that 
\begin{equation}
    \begin{aligned}
         &S_\nu \circ pS_\nu \\
         &\in(\nu\Pi_0\sC_{H^{1/2}\to L^2}+\Pi_0^\perp \sC_{H^{1/2}\to L^2})\circ(\nu\sC_{H^1\to H^{1/2}}\Pi_0+\sC_{H^1\to H^{1/2}}\Pi_0^\perp)\subset \sE_{H^1\to L^2}.
    \end{aligned}
\end{equation}
Since 
\begin{align*}
    i\signp(D)=\Pi_0i\Pi_0+\Pi_0^\perp \Big(i\sign(D)\Big)\Pi_0^\perp &\in \sE_{L^2\to L^2},
\end{align*}
 Lemma \ref{lemma:sE} (3) implies
\[
i\sign^+(D)\Big(S_\nu  \circ pS_\nu -\frac{\gamma}{2}\Big)\in\sE_{H^1\to L^2}.
\]
By the same argument, all  entries of  $\L_{\nu,\eps}^+$ are in $\sE_{H^1\to L^2}$. Consequently, it follows from  Lemma \ref{lemma:sE} (4) and (5) that 
\begin{equation}
    \L_{\nu,\eps}^+,  \B_{\nu,\eps},  \L_{0,\eps}^+,  \L_{0,0}^+\in\sE_{(H^1)^2\to (L^2)^2}.
\end{equation}
In particular, for any $\ld \in \Gamma$, $ \L_{0,0}^+-\ld$ is invertible and belongs to $\sE_{(H^1)^2\to (L^2)^2}$. Consequently, $(\L_{\nu,\eps}^+-\lambda)^{-1}\in\sE_{(L^2)^2\to (H^1)^2}$ by Lemma \ref{lemma:sE} (6). We conclude that $P_{\nu,\eps}$ is in $\sE_{(H^1)^2\to (H^1)^2}$, and hence so is $\cK_{\nu,\eps}$ and 
\[
\B_{\nu,\eps}\cK_{\nu,\eps}\in  \sE_{(H^1)^2\to (L^2)^2}.
\]
From the definition \eqref{basis_k=1_2} of the $U_j$, we have 
\[
\Pi_0U_1=\Pi_0U_2=\Pi_0^\perp U_3=0.
\]
Hence, recalling the definition \ref{def:sE} of $\sE$, we deduce
\begin{align*}
    \cK_{\nu,\eps}U_j&\in \Pi_0^\perp \sC_{(H^1)^2}+\nu \Pi_0\sC_{(H^1)^2}\subset\sC_{(\mathring{H}^1)^2}+\nu\sC_{\C^2},~j=1,2,\\
  \B_{\nu,\eps}\cK_{\nu,\eps}U_j&\in \Pi_0^\perp \sC_{(L^2)^2}+\nu \Pi_0\sC_{(L^2)^2}\subset\sC_{(\rL^2)^2}+\nu\sC_{\C^2},~j=1,2,\\
    \cK_{\nu,\eps}U_3&\in \Pi_0 \sC_{(H^1)^2}+\nu \Pi_0^\perp\sC_{(H^1)^2}
       \subset\sC_{\C^2}+\nu\sC_{(\mathring{H}^1)^2},\\
       \B_{\nu,\eps}\cK_{\nu,\eps}U_3&\in \Pi_0 \sC_{(L^2)^2}+\nu \Pi_0^\perp\sC_{(L^2)^2}\subset\sC_{\C^2}+\nu\sC_{(\mathring{L}^2)^2}. 
\end{align*}
Using these properties for the  inner products $(\B_{\nu,\eps}\cK_{\nu,\eps}U_j,\cK_{\nu,\eps}U_k)$, we obtain \eqref{eq:Lmat_in_sC}. 
\end{proof}
%%%%%%%%%%%%%%%
\subsection{Expansion in $\nu$}
In this subsection, we refine \eqref{eq:Lmat_in_sC} by proving exact relations and vanishing of certain entries in the  $\nu^j$-coefficients of $\Lmat_{\nu,\eps}$, for $j=0, 1, 2$. By ``exact'' we mean that these relations and vanishing hold at all orders $\eps^k$, $k\ge 0$. To this end, we shall exploit parity of not only the variable coefficients of $\L^+_{\nu, \eps}$  but also the eigenvectors. Since the unperturbed eigenvectors $U_j$'s are exponentials, we obtain parity by considering 
\begin{equation}
    \label{def:basis_Vpm}
    V_\pm :=U_1\pm U_2\in\cV_{0,0},
\end{equation}
that is,
\begin{equation}
\label{basis_Vpm}
    V_+=\begin{bmatrix}
        -2\sin(x)\\
        (2c_0-\gamma)\cos(x)
    \end{bmatrix}, \quad V_-=\begin{bmatrix}
        2i\cos(x)\\
        (2c_0-\gamma)i\sin(x)
    \end{bmatrix}.
\end{equation}
Each entry of $V_\pm$ is either an odd or even function on $\T$.  The corresponding perturbed eigenvectors are 
\begin{equation}\label{form:Vpm:cK}
    V_\pm^{\nu,\eps}:=U_1^{\nu,\eps}\pm U_2^{\nu,\eps}=\cK_{\nu,\eps}V_\pm\in\cV_{\nu,\eps}.
\end{equation}
If $X(\nu, \eps)$ is an analytic function, we shall use the notation 
\bq
X^{[j]}:=[\nu^j]X,\quad  \L\bracketNu{j}:=[\nu^j]\L_{\nu,\eps}^+. 
\eq
\subsubsection{The $\nu^0$-coefficient of $\Lmat$}
We recall from Theorem \ref{thm:eigenvectors_U} that for $0<|\eps|<\eps_0$, $0$ is an eigenvalue of $\L_{0,\eps}^+$ with algebraic multiplicity 3 
and geometric multiplicity 2. Hence, we have $(\L_{0,\eps}^+)^2|_{\cV_{0,\eps}}=0$, so that
\begin{equation}\label{squareL=0}
    \Lmat_{0,\eps}^2=0.
\end{equation}
Since $  \Lmat_{0, 0}=0$, \eqref{squareL=0} is also valid when $\eps=0$.  From \eqref{eq:symm_Lmat} and the fact that $\beta_2=-\beta_1$, we have
\begin{equation}
    (\Lmat_{0,\eps})_{12}=-(\Lmat_{0,\eps})_{21}.
\end{equation}
On the other hand,  \eqref{eq:Lmat_in_sC} implies the form 
\[
\Lmat\bracketNu{0}\equiv \Lmat_{0,\eps}= \begin{bmatrix}
    A& B&0\\
    -B&C&0\\
    0&0&D
\end{bmatrix}.
\]
Using this and \eqref{squareL=0} gives 
 \begin{equation}
        0=\begin{bmatrix}
            A^2-B^2&AB+BC&0\\
            -AB-BC&-B^2+C^2&0\\
            0&0&D^2
        \end{bmatrix},
    \end{equation}
whence 
\[
        A^2-B^2=(A+C)B=-B^2+C^2=D^2=0. 
\]
It follows that 
    \begin{equation}\label{ABCD}
        A=\pm B, C=\mp B, D=0.
    \end{equation}
Setting $\nu=0$ in the expansion \eqref{eq:Lmat_expand_in_nu_eps} of $\Lmat_{0,\eps}$, we deduce in view of \eqref{ABCD} that $A=B$, so that 
\bq
(\Lmat_{0,\eps})_{11}= (\Lmat_{0,\eps})_{12},\quad (\Lmat_{0,\eps})_{22}= -(\Lmat_{0,\eps})_{12}. 
\eq
Thus  $\Lmat\bracketNu{0}$ has the form 
\begin{equation}
    \label{eq:Lmat_nu_0_final}
    \Lmat\bracketNu{0}=\begin{bmatrix}
        -\frac{i \left(3 \left(c_0^6-c_0^4+c_0^2\right)+1\right)}{8 c_0^5}\eps^2-ib_1(\eps)\eps^3&-\frac{i \left(3 \left(c_0^6-c_0^4+c_0^2\right)+1\right)}{8 c_0^5}\eps^2-ib_1(\eps)\eps^3&0\\
        \frac{i \left(3 \left(c_0^6-c_0^4+c_0^2\right)+1\right)}{8 c_0^5}\eps^2+ib_1(\eps)\eps^3&\frac{i \left(3 \left(c_0^6-c_0^4+c_0^2\right)+1\right)}{8 c_0^5}\eps^2+ib_1(\eps)\eps^3&0\\
        0&0&0
    \end{bmatrix},\quad b_1\in\anaR{\eps}.
\end{equation}
In particular, we have 
\begin{equation}
    (\Lmat_{0,\eps})_{1j}+(\Lmat_{0,\eps})_{2j}=0,\quad j=1,2,3,
\end{equation}
which in conjunction with the  expression  \eqref{def:Lmat} of $\Lmat_{\nu,\eps}$ implies 
\bq\label{BV+Uj}
\begin{aligned}
    0=(\B_{0,\eps}(U_1^{0,\eps}+U_2^{0,\eps}),U_j^{0,\eps})=(\B_{0,\eps}V_+^{0,\eps},U_j^{0,\eps}),\quad j=1,2,3.
\end{aligned}
\eq
Next, since  $\L_{0,\eps}^+V_+^{0,\eps}\in \cV_{0,\eps}$, using  \eqref{eq:expr_cV_vector_in_basis_by_inner_product|} gives 
\begin{equation}
    \begin{aligned}
             \L\bracketNu{0}V_+\bracketNu{0}=\L_{0,\eps}^+V_+^{0,\eps}&=-i \sum_{k=1}^3\frac{1}{2\pi\beta_k}(J\L_{0,\eps}^+V_+^{0,\eps},U_k^{0,\eps})U_k^{0,\eps}=-i \sum_{k=1}^3\frac{1}{2\pi\beta_k}(-\B_{0,\eps}V_+^{0,\eps},U_k^{0,\eps})U_k^{0,\eps},
    \end{aligned}
\end{equation}
Combining this with \eqref{BV+Uj} yields
\bq\label{eq:LVp_is_0}
\L\bracketNu{0}V_+\bracketNu{0}=0.
\eq
Since $(\Lmat_{0,\eps})_{3j}=0$ (from \eqref{eq:Lmat_nu_0_final}), an analogous argument gives
\begin{equation}
    \label{eq:LU3_is_0}
   \L\bracketNu{0}U_3\bracketNu{0}= \L^+_{0,\eps}U_3^{0,\eps}=0.
\end{equation}
\subsubsection{Parity}
\def\oddVec{\begin{bmatrix}
even_0\\
odd
\end{bmatrix}}
\def\evenVec{\begin{bmatrix}
odd\\
even_0
\end{bmatrix}}
\def\Rt{\wt{\newRev}}
\def\rt{\wt{\rev}}
\def\vecU{\mathtt{U}}
\def\vecV{\mathtt{V}}
\def\A{\mathcal{A}}
To leverage parity, we single out two types of vectors, $\oddVec$ and $\evenVec$. Clearly the inner product between a vector of one type and a vector of the other type is zero. Moreover, the inner product between a vector of either type and a constant vector is  zero.  The above types of vectors can be characterized using an involution operator similar to the one in Proposition \ref{prop:L+} but without conjugation:
\begin{equation}
\label{def:Rt_rt}
    \begin{aligned}
        \rt f(x)&=f(-x),\quad 
        \Rt\begin{bmatrix}
            f\\
            g
        \end{bmatrix}&=\begin{bmatrix}
            -\rt f\\
            \rt g
        \end{bmatrix}.
    \end{aligned}
\end{equation}
Then we have
\begin{align}\label{iff:evenVec_Rt}
\forall \vecV,~\text{if~}\Pi_0\vecV=0,~\text{then}~
    \Rt\vecV=\vecV&\iff \vecV=\evenVec~\text{and}~  \Rt\vecV=-\vecV\iff \vecV=\oddVec.
\end{align}
We record simple properties of $\rt$ in the following lemma, whose proof is omitted. 
\begin{lemm}
\label{lemma:commute_rt}
    (1) If $f(x)$ is even, then $\rt(fg)=f \rt g$.
    
    (2 For any symbol $a: \Z\to \C$, we have 
    \[
    \rt \circ a(D)=a(D)\circ \rt~\text{if}~a~\text{is~even,~and}~  \rt \circ a(D)=-a(D)\circ \rt~\text{if}~a~\text{is~odd}.
    \]
\end{lemm}
The next lemma concerns $\L\bracketNu{0}$ and $\L\bracketNu{2}$. 
\begin{lemm}
\label{lemma:parity_main}    
    (1) The following mapping properties hold:
   \begin{align}
   &\L\bracketNu{0}: (\mathring{H}^1)^2\to (\rL^2)^2,\quad  \L\bracketNu{0}: \C^2\to \C^2,\\
   &P\bracketNu{0}:  (\rL^2)^2\to (\mathring{H}^1)^2,\\ \label{mean:cK0}
    &\cK\bracketNu{0}: (\rL^2)^2\to (\rL^2)^2,\quad \cK\bracketNu{0}: (\mathring{H}^1)^2\to (\mathring{H}^1)^2,\\
   & \L\bracketNu{2}: (\mathring{H}^1)^2\to (\rL^2)^2.
   \end{align}
    (2) If $\Pi_0\vecV=0$, then 
    \begin{align}
    \label{eq:parity_L0_switch}
        \L\bracketNu{0}\Rt\vecV&=-\Rt\L\bracketNu{0}\vecV,\\
        \label{eq:parity_P0_preserve}
        P\bracketNu{0}\Rt \vecV&=\Rt P\bracketNu{0}\vecV,\\
        \label{eq:parity_cK0_preserve}
        \cK\bracketNu{0}\Rt \vecV&=\Rt \cK\bracketNu{0}\vecV,\\
        \label{eq:parity_L2_preserve}
        \L\bracketNu{2}\Rt\vecV&=\Rt\L\bracketNu{2}\vecV.
    \end{align}
\end{lemm}
\begin{proof}
    (1) We recall that 
    \bq\label{form:LNu0}
    \L\bracketNu{0}=\begin{bmatrix}
        i\signp(D)\Big(|D|^{1/2}\circ p|D|^{1/2}-\frac{\gamma}{2}\Big)&1\\
        -|D|^{1/2}\circ \frac{1+q}{\zeta'}|D|^{1/2}-\frac{\gamma^2}{4}&\Big(|D|^{1/2}\circ p|D|^{1/2}-\frac{\gamma}{2}\Big) i\signp(D)
    \end{bmatrix}.
    \eq
    It is readily seen that $  \L\bracketNu{0}: \C^2\to \C^2$ and $\L\bracketNu{0}: (\mathring{H}^1)^2\to (\rL^2)^2$. Moreover,  $\L\bracketNu{0}-\lambda: (H^1)^2\to (L^2)^2$ is invertible for all $\ld \in \Gamma$. Thus  $\L\bracketNu{0}-\lambda: (\mathring{H}^1)^2\to (\rL^2)^2$ is invertible for all $\ld \in \Gamma$. It follows that 
\[
P\bracketNu{0}=-\frac{1}{2\pi i}\oint_\Gamma (\L\bracketNu{0}-\lambda)^{-1}d\lambda: (\rL^2)^2\to (\mathring{H}^1)^2. 
\]
Then, using the definition of $\cK\bracketNu{0}=\cK_{0, \eps}$, we deduce that $\cK\bracketNu{0}: (\rL^2)^2\to (\rL^2)^2$ and $\cK\bracketNu{0}: (\mathring{H}^1)^2\to (\mathring{H}^1)^2$. 

To obtain the formula for $\L\bracketNu{2}$, we first expand the operator
\[
   \fD_\nu(f):=S_\nu \circ fS_\nu 
   \]
  using \eqref{eq:expansion_sqrtOp_temp}, i.e.,
  \[
   S_\nu =|D|^{1/2}\Pi_0^\perp+\Pi_0\nu+\frac{1}{2}D^{-1}|D|^{1/2}\Pi_0^\perp\nu^2+O(\nu^3).
   \]
  The result is 
  \begin{equation}\label{eq:template_fD}
    \begin{aligned}
        \fD_\nu(f)
        &=\fD\bracketNu 0(f)+\nu\fD\bracketNu 1(f)+\nu^2\fD\bracketNu 2(f)+O(\nu^3),
    \end{aligned}
\end{equation}
where
\begin{align}
    \fD\bracketNu 0(f)&=|D|^{1/2}\Pi_0^\perp\circ f|D|^{1/2}\Pi_0^\perp,\\
    \label{eq:fD1}
    \fD\bracketNu 1(f)&=\Pi_0\circ f|D|^{1/2}\Pi_0^\perp+|D|^{1/2}\Pi_0^\perp\circ f\Pi_0,\\
    \label{eq:fD2}
    \fD\bracketNu 2(f)&=\frac{1}{2}|D|^{1/2}\Pi_0^\perp\circ fD^{-1}|D|^{1/2}\Pi_0^\perp +\frac{1}{2}D^{-1}|D|^{1/2}\Pi_0^\perp\circ f|D|^{1/2}\Pi_0^\perp+\Pi_0\circ f\Pi_0.
\end{align}
The above formula of $ \fD\bracketNu 2(f)$ implies that $\Pi_0\fD\bracketNu 2(f)g=0$ if $\Pi_0g=0$. Consequently, we obtain
\begin{equation}
    \label{eq:local_L2_vecV}
    \begin{aligned}
        \L\bracketNu{2}
        =\begin{bmatrix}
            i\signp(D)\fD\bracketNu{2}(p)&0\\
            -\fD\bracketNu{2}(\frac{1+q}{\zeta'})&\fD\bracketNu{2}(p) i\signp(D)
        \end{bmatrix}: (\mathring{H}^1)^2\to (\rL^2)^2.
    \end{aligned}
\end{equation}
(2) Suppose that $\vecV=\begin{bmatrix} v_1 \\ v_2\end{bmatrix}$ satisfies $\Pi_0\vecV=0$. Noting that $\signp(D)g=\sign(D)g$ if $\Pi_0g=0$,   \eqref{form:LNu0} gives
 \bq\label{form:LNu0:2}
    \L\bracketNu{0}\vecV=\begin{bmatrix}
        i\sign(D)\Big(|D|^{1/2}\circ p|D|^{1/2}-\frac{\gamma}{2}\Big)v_1+v_2\\
        \Big(-|D|^{1/2}\circ \frac{1+q}{\zeta'}|D|^{1/2}-\frac{\gamma^2}{4}\Big)v_1+\Big(|D|^{1/2}\circ p|D|^{1/2}-\frac{\gamma}{2}\Big) i\sign(D)v_2
    \end{bmatrix}.
    \eq
We recall from Lemma \ref{lemm:pq} that $p$ and $\frac{1+q}{\zeta'}$ are even functions. On the other hand, by Lemma \ref{lemma:commute_rt}, for any even function $f(x)$, $|D|^{1/2}\circ f|D|^{1/2}$ commutes with $\rt$, and $i\sign(D)$ anti-commutes with $\rt$. Therefore, 
\begin{align*}
    \L\bracketNu{0}\Rt\vecV
    &=\begin{bmatrix}
        i\sign(D)\Big(|D|^{1/2}\circ p|D|^{1/2}-\frac{\gamma}{2}\Big)(-\rt v_1)+\rt v_2\\
        \Big( -|D|^{1/2}\circ \frac{1+q}{\zeta'}|D|^{1/2}-\frac{\gamma^2}{4}\Big)(-\rt v_1)+\Big(|D|^{1/2}\circ p|D|^{1/2}-\frac{\gamma}{2}\Big) i\sign(D)\rt v_2
    \end{bmatrix}\\
    &=\begin{bmatrix}
       \rt i\sign(D)\Big(|D|^{1/2}\circ p|D|^{1/2}-\frac{\gamma}{2}\Big)v_1+\rt v_2\\
        -\rt \Big( -|D|^{1/2}\circ \frac{1+q}{\zeta'}|D|^{1/2}-\frac{\gamma^2}{4}\Big)v_1-\rt \Big(|D|^{1/2}\circ p|D|^{1/2}-\frac{\gamma}{2}\Big) i\sign(D) v_2
    \end{bmatrix}\\
    &=-\Rt\L\bracketNu{0}\vecV,
\end{align*} 
proving \eqref{eq:parity_L0_switch}. 

Regarding  $ P\bracketNu{0}$, we define for each $\ld \in \Gamma$ the vector $\vecU$ given by $\Rt\vecU=(\L\bracketNu{0}-\lambda)^{-1}\Rt \vecV$. From the above proof of (1), we know that $(\L\bracketNu{0}-\lambda)^{-1}: (\rL^2)^2\to (\mathring{H}^1)^2$. Hence, $\Pi_0 \vecU=0$ since $\Pi_0 \Rt \vecV=0$. Then applying  \eqref{eq:parity_L0_switch} gives 
\[
\Rt \vecV=(\L\bracketNu{0}-\lambda)\Rt \vecU=\Rt(-\L\bracketNu{0}-\lambda)\vecU. 
\]
It follows that
\[
\vecU=-(\L\bracketNu{0}+\lambda)^{-1}\vecV,
\]
whence
\[
(\L\bracketNu{0}-\lambda)^{-1}\Rt \vecV=-\Rt(\L\bracketNu{0}+\lambda)^{-1}\vecV.
\]
Since the contour $\Gamma$ is a circle centered at origin, we deduce  
\begin{align*}
    P\bracketNu{0}\Rt\vecV
    &=-\frac{1}{2\pi i}\oint_\Gamma(\L\bracketNu{0}-\lambda)^{-1}\Rt \vecV d\lambda\\
    &=-\frac{1}{2\pi i}\oint_\Gamma-\Rt(\L\bracketNu{0}+\lambda)^{-1}\vecV d\lambda\\
    &\xlongequal{\lambda\to -\lambda}-\frac{1}{2\pi i}\oint_\Gamma-\Rt(\L\bracketNu{0}-\lambda)^{-1}\vecV (-d\lambda)\\
    &=\Rt P\bracketNu{0}\vecV.
\end{align*}
This proves \eqref{eq:parity_P0_preserve}, from which  \eqref{eq:parity_cK0_preserve} follows. 

Applying  Lemma \ref{lemma:commute_rt}, we find that  $\fD\bracketNu{2}$ (given by \eqref{eq:fD2}) satisfies $\fD\bracketNu{2}\rt v_j=-\rt \fD\bracketNu{2}v_j$ and $\sign(D)$   anti-commute with $\rt$.  Moreover, since $\Pi_0 \Rt\vecV=0$, we have $\Pi_0\fD\bracketNu{2}(f)\Rt\vecV=0$ for any even function $f: \T\to \C$. Hence,  when calculating  $  \L\bracketNu{2}\Rt\vecV$ using \eqref{eq:local_L2_vecV}, $\signp(D)$ can be replaced by $\sign(D)$. Thus we obtain 
\begin{align*}
    \L\bracketNu{2}\Rt\vecV
    &=\begin{bmatrix}
            i\sign(D)\fD\bracketNu{2}(p)(-\rt v_1)\\
             -\fD\bracketNu{2}(\frac{1+q}{\zeta'})(-\rt v_1)+\fD\bracketNu{2}(p) i\sign(D)\rt v_2
        \end{bmatrix}\\
        &=\begin{bmatrix}
            -\rt i\sign(D)\fD\bracketNu{2}(p)v_1\\
             -\rt \fD\bracketNu{2}(\frac{1+q}{\zeta'})v_1+\rt \fD\bracketNu{2}(p) i\sign(D)v_2
        \end{bmatrix}\\
        &=\Rt\L\bracketNu{2}\vecV,
\end{align*}
which completes the proof of \eqref{eq:parity_L2_preserve}.
\end{proof}
As for $\L\bracketNu{1}$, we have
\begin{lemm}
    \label{lemma:parity_L1}
     (1)  $\L\bracketNu{1}: (\mathring{H}^1)^2\to \C^2$.

    (2) If  $\vecV=\evenVec$, then $\L\bracketNu{1}\vecV=0$.

    (3) If $\vecU_0\in \C^2$, then $\L\bracketNu{1}\vecU_0=\evenVec$.
\end{lemm}
\begin{proof}
   (1)  From  \eqref{eq:fD1}  we have
\bq\label{LNu1}
        \L\bracketNu{1}
        =\begin{bmatrix}
            i\signp(D)\Big(\Pi_0\circ p|D|^{1/2}\Pi_0^\perp+|D|^{1/2}\Pi_0^\perp\circ p\Pi_0\Big)&0\\
            -\Big(\Pi_0\circ \frac{1+q}{\zeta'}|D|^{1/2}\Pi_0^\perp+|D|^{1/2}\Pi_0^\perp\circ \frac{1+q}{\zeta'}\Pi_0\Big)&\Big(\Pi_0\circ p|D|^{1/2}\Pi_0^\perp+|D|^{1/2}\Pi_0^\perp\circ p\Pi_0\Big)i\signp(D)
        \end{bmatrix}.
     \eq
For  $\vecV=\begin{bmatrix} v_1(x)\\ v_2(x)\end{bmatrix}\in (\mathring{H}^1)^2$, we have $\Pi_0 v_j=0$, and hence 
      \bq\label{L1V}
         \L\bracketNu{1}\vecV=\begin{bmatrix}
            i\signp(D)\Pi_0\circ p|D|^{1/2}v_1\\
            -\Pi_0\circ \frac{1+q}{\zeta'}|D|^{1/2}v_1+\Pi_0\circ p|D|^{1/2}i\signp(D)v_2
        \end{bmatrix}\in \C^2
\eq
    thanks to the presence of $\Pi_0$. 
    
(2)  Suppose that $\vecV=\evenVec$, i.e. $v_1(x)$ is odd and $v_2(x)$ is even and mean-free. Then $ \L\bracketNu{1}\vecV$ is given by \eqref{L1V} since $\Pi_0\vecV=0$. By Lemma \ref{lemma:commute_rt}, $|D|^{1/2}v_1$ and $|D|^{1/2}\signp(D)v_2=|D|^{1/2}\sign(D)v_2$ are odd functions. Recalling that $p$ and $\frac{1+q}{\zeta'}$ are even functions, we conclude from the formula \eqref{L1V} that $ \L\bracketNu{1}\vecV=0$. 

(3)  For $\vecU_0=\begin{bmatrix} a_1\\ a_2\end{bmatrix}\in \C^2$, \eqref{LNu1} gives 
\begin{align*}
    \L\bracketNu{1}\vecU_0
        &=\begin{bmatrix}
            a_1 i\signp(D)|D|^{1/2}p\\
            -a_1|D|^{1/2}\frac{1+q}{\zeta'}+a_2i|D|^{1/2}p
        \end{bmatrix}=\begin{bmatrix}
            a_1 i\sign(D)|D|^{1/2}p\\
            -a_1|D|^{1/2}\frac{1+q}{\zeta'}+a_2i|D|^{1/2}p
        \end{bmatrix}.
    \end{align*}
   Appealing to  Lemma \ref{lemma:commute_rt} again yields  $\L\bracketNu{1}\vecU_0=\evenVec$.
\end{proof}
\subsubsection{The  $\nu^1$-coefficient of $\Lmat$}
From  \eqref{eq:Lmat_in_sC}, we know that $\Lmat\bracketNu{1}$ has the form 
\[
\Lmat\bracketNu{1}=\begin{bmatrix}
    0&0&(\Lmat\bracketNu{1})_{13}\\
    0&0&(\Lmat\bracketNu{1})_{23}\\
    (\Lmat\bracketNu{1})_{31}& (\Lmat\bracketNu{1})_{32}&0
\end{bmatrix},
\]
Invoking the symmetry  \eqref{eq:symm_Lmat},  we only need to determine $(\Lmat\bracketNu{1})_{13}$ and $(\Lmat\bracketNu{1})_{23}$. To this end, we first recall the form \eqref{def:Lmat} of $\Lmat$ to have
\begin{align*}
    (\Lmat\bracketNu{1})_{13}+(\Lmat\bracketNu{1})_{23}&=[\nu^1]\Big(\frac{i}{2\pi \beta_3} \left( \mathcal{B}_{\nu,\eps} U_1^{\nu,\eps}, U_3^{\nu,\eps} \right)+\frac{i}{2\pi \beta_3} \left( \mathcal{B}_{\nu,\eps} U_2^{\nu,\eps}, U_3^{\nu,\eps} \right)\Big)\\
    &=\frac{i}{2\pi \beta_3} \Big([\nu^1]\left( \mathcal{B}_{\nu,\eps} V_+^{\nu,\eps}, U_3^{\nu,\eps} \right)\Big)\\
    &=\frac{i}{2\pi \beta_3} 
    \Big((\B\bracketNu{1}V_+\bracketNu{0},U_3\bracketNu{0})
    +(\B\bracketNu{0}V_+\bracketNu{1},U_3\bracketNu{0})
    +(\B\bracketNu{0}V_+\bracketNu{0},U_3\bracketNu{1})\Big)\\
    &=\frac{i}{2\pi \beta_3} 
    \Big((-J\L\bracketNu{1}V_+\bracketNu{0},U_3\bracketNu{0})
    +(V_+\bracketNu{1},-J\L\bracketNu{0}U_3\bracketNu{0})
    +(-J\L\bracketNu{0}V_+\bracketNu{0},U_3\bracketNu{1})\Big).
\end{align*}
But  $\L\bracketNu{0}V_+\bracketNu{0}= \L\bracketNu{0}U_3\bracketNu{0}=0$ by \eqref{eq:LVp_is_0} and  \eqref{eq:LU3_is_0}, so that 
\[
  (\Lmat\bracketNu{1})_{13}+(\Lmat\bracketNu{1})_{23}=\frac{i}{2\pi \beta_3}(-J\L\bracketNu{1}V_+\bracketNu{0},U_3\bracketNu{0}).
  \]
Now we recall from \eqref{basis_Vpm} and \eqref{form:Vpm:cK}  that $V_+\bracketNu{0}=\cK\bracketNu{0}V_+$ and $V_+=\evenVec$. Hence, combining the properties \eqref{mean:cK0} and \eqref{eq:parity_cK0_preserve} of $\cK\bracketNu{0}$ yields $V_+\bracketNu{0}=\evenVec$. Then Lemma \ref{lemma:parity_L1} (2) implies   
\bq
\label{eq:L1Vp_is_0}
\L\bracketNu{1}V_+\bracketNu{0}=0,
\eq
 and thus 
\begin{equation}
    (\Lmat\bracketNu{1})_{13}+(\Lmat\bracketNu{1})_{23}=0.
\end{equation}
Combining this with  \eqref{eq:Lmat_expand_in_nu_eps}, \eqref{eq:Lmat_expand_in_nu_eps:2} and \eqref{eq:symm_Lmat}, we conclude that $ \Lmat\bracketNu{1}$ has the form
\begin{equation}
\label{eq:Lmat_nu_1_final}
    \Lmat\bracketNu{1}=\begin{bmatrix}
        0&0&-\frac{i \left(c_0^2+1\right)}{2 c_0}\eps-i\eps^3b_2(\eps)\\
        0&0&\frac{i \left(c_0^2+1\right)}{2 c_0}\eps+i\eps^3b_2(\eps)\\
        (\Lmat\bracketNu{1})_{31}& (\Lmat\bracketNu{1})_{32}&0
    \end{bmatrix},\quad (\Lmat\bracketNu{1})_{31}=\frac{\beta_3}{\beta_1}(\Lmat\bracketNu{1})_{13},\quad (\Lmat\bracketNu{1})_{32}=\frac{\beta_3}{\beta_2}(\Lmat\bracketNu{1})_{23}
\end{equation}
for some $b_2(\eps)\in\anaR{\eps}$.
\subsubsection{The $\nu^2$-coefficient of $\Lmat$}
From  \eqref{eq:Lmat_in_sC}, we have 
\[
\Lmat\bracketNu{2}=\begin{bmatrix}
    *&*&0\\
    *&*&0\\
    0&0&*
\end{bmatrix}.
\]
Since $\B_{\nu,\eps}$ is self-adjoint and $\L\bracketNu{0}V_+\bracketNu{0}=\L\bracketNu{1}V_+\bracketNu{0}=0$ by \eqref{eq:LVp_is_0} and \eqref{eq:L1Vp_is_0}, we have
\bq\label{Nu2BVp}
\begin{aligned}
    [\nu^2](\B_{\nu,\eps}V_+^{\nu,\eps},V_+^{\nu,\eps})&=
   2\RE (\B\bracketNu{0}V_+\bracketNu{0},V_+\bracketNu{2})+(\B\bracketNu{0}V_+\bracketNu{1},V_+\bracketNu{1})+2\RE(\B\bracketNu{1}V_+\bracketNu{0},V_+\bracketNu{1})+(\B\bracketNu{2}V_+\bracketNu{0},V_+\bracketNu{0})\\
   &=2\RE (-J\L\bracketNu{0}V_+\bracketNu{0},V_+\bracketNu{2})+(-J\L\bracketNu{0}V_+\bracketNu{1},V_+\bracketNu{1})+2\RE(-J\L\bracketNu{1}V_+\bracketNu{0},V_+\bracketNu{1})+(\B\bracketNu{2}V_+\bracketNu{0},V_+\bracketNu{0})\\
   &=(-J\L\bracketNu{0}V_+\bracketNu{1},V_+\bracketNu{1})+(\B\bracketNu{2}V_+\bracketNu{0},V_+\bracketNu{0}). 
\end{aligned}
\eq
Because $\Lmat_{\nu,\eps}$ is the matrix representation  of $ \L_{\nu,\eps}^+$  on the basis $\cV_{\nu,\eps}$, we have 
\begin{align*}
    \L_{\nu,\eps}^+V_+^{\nu,\eps}&=\L_{\nu,\eps}^+U_1^{\nu,\eps}+\L_{\nu,\eps}^+U_2^{\nu,\eps}=\sum_{k=1}^3 \Big((\Lmat_{\nu,\eps})_{1k}+(\Lmat_{\nu,\eps})_{2k}\Big)U_k^{\nu,\eps},
\end{align*}
so that 
\begin{align*}
    [\nu^1](\L_{\nu,\eps}^+V_+^{\nu,\eps})&=\sum_{k=1}^3\Big((\Lmat\bracketNu{0})_{1k}+(\Lmat\bracketNu{0})_{2k}\Big)U_k\bracketNu{1}
    +\sum_{k=1}^3\Big((\Lmat\bracketNu{1})_{1k}+(\Lmat\bracketNu{1})_{2k}\Big)U_k\bracketNu{0}.
\end{align*}
From the form  \eqref{eq:Lmat_nu_0_final} of $\Lmat\bracketNu{0}$ and the form  \eqref{eq:Lmat_nu_1_final} of $\Lmat\bracketNu{1}$, we see that 
\[
(\Lmat\bracketNu{0})_{1k}+(\Lmat\bracketNu{0})_{2k}=(\Lmat\bracketNu{1})_{1k}+(\Lmat\bracketNu{1})_{2k}=0,\quad \forall k\in\{1, 2, 3\}. 
\]
It follows that $ [\nu^1](\L_{\nu,\eps}^+V_+^{\nu,\eps})=0$. Combing this with \eqref{eq:L1Vp_is_0} yields
\begin{equation}
    \label{eq:L0_V+1_is_0}
    \L\bracketNu{0}V_+\bracketNu{1}=[\nu^1](\L_{\nu,\eps}^+V_+^{\nu,\eps})-\L\bracketNu{1}V_+\bracketNu{0}=0.
\end{equation}
Recalling that $V_+\bracketNu{0}=\evenVec$, we can use \eqref{eq:parity_L2_preserve} and the fact that $J\Rt\vecV=-\Rt J\vecV$ to have $J\L\bracketNu{2}V_+\bracketNu{0}=\oddVec$. This implies 
\begin{equation}
\label{eq:B2_V+0_is_0}
    (\B\bracketNu{2}V_+\bracketNu{0},V_+\bracketNu{0})=(-J\L\bracketNu{2}V_+\bracketNu{0},V_+\bracketNu{0})=0.
\end{equation}
It follows from \eqref{Nu2BVp}, \eqref{eq:L0_V+1_is_0}, and \eqref{eq:B2_V+0_is_0}  that 
\begin{equation}\label{Nu2:BVp}
    [\nu^2](\B_{\nu,\eps}V_+^{\nu,\eps},V_+^{\nu,\eps})=0.
\end{equation}
 We recall from Proposition \ref{prop:structure_Lmat}  that  $(\B_{\nu,\eps}U_1^{\nu,\eps},U_2^{\nu,\eps})\in\R$ and $\beta_1=-\beta_2$.  Consequently, we have
\begin{align*}
    0&=[\nu^2](\B_{\nu,\eps}V_+^{\nu,\eps},V_+^{\nu,\eps})\\
    &=[\nu^2]\Big[(\B_{\nu,\eps}U_1^{\nu,\eps},U_1^{\nu,\eps})+(\B_{\nu,\eps}U_2^{\nu,\eps},U_2^{\nu,\eps})+2(\B_{\nu,\eps}U_1^{\nu,\eps},U_2^{\nu,\eps})\Big]\\
    &=\frac{2\pi \beta_2}{i}[\nu^2]\Big(-(\Lmat_{\nu,\eps})_{11}+(\Lmat_{\nu,\eps})_{22}+ 2(\Lmat_{\nu,\eps})_{12}\Big),
\end{align*}
whence
\bq
(\Lmat\bracketNu{2})_{11}=(\Lmat\bracketNu{2})_{22}+2(\Lmat\bracketNu{2})_{12}.
\eq
Combining this with \eqref{eq:Lmat_expand_in_nu_eps}, \eqref{eq:Lmat_expand_in_nu_eps:2} and  \eqref{eq:symm_Lmat}, we conclude
\begin{equation}
    \label{eq:Lmat_nu_2_final}
    \Lmat\bracketNu{2}=\begin{bmatrix}
        \frac{ic_0^3}{c_0^2+1}+\eps^2i\big[b_3(\eps)+b_4(\eps)\big] & \eps^2ib_4(\eps) &0\\
        -\eps^2ib_4(\eps) &  \frac{ic_0^3}{c_0^2+1}+\eps^2i\big[b_3(\eps)-b_4(\eps)\big] &0\\
        0&0&\frac{ic_0^3}{c_0^2-1}+O(\eps^2)
    \end{bmatrix}
\end{equation}
for some $b_3(\eps),b_4(\eps)\in\anaR{\eps}$.
\subsection{Final form of $\Lmat_{\nu,\eps}$}
\def\highPhantom{\vphantom{\begin{bmatrix} 0 \\ 0 \\ 0 \\ 0 \\ 0 \end{bmatrix}}}
\def\highDelim#1{\left#1\highPhantom\right.}
We substitute \eqref{eq:Lmat_nu_0_final}, \eqref{eq:Lmat_nu_1_final}, and \eqref{eq:Lmat_nu_2_final} in  \eqref{eq:Lmat_expand_in_nu_eps} to arrive at 
\begin{prop}\label{prop:finalmatrix}    Let $\gamma \ne 0$. There exist $\eps_0=\eps_0(\gamma)>0$ and $\nu_0=\nu_0(\gamma)>0$ such that for all $|\eps|<\eps_0$ and $|\nu|<\nu_0$, we have
\begin{equation}
 \label{eq:Lmat_final}
    \begin{aligned}
       \Lmat_{\nu,\eps}&= \begin{bmatrix}
        -\frac{i \left(3 \left(c_0^6-c_0^4+c_0^2\right)+1\right)}{8 c_0^5}\eps^2-ib_1(\eps)\eps^3&-\frac{i \left(3 \left(c_0^6-c_0^4+c_0^2\right)+1\right)}{8 c_0^5}\eps^2-ib_1(\eps)\eps^3&0\\
        \frac{i \left(3 \left(c_0^6-c_0^4+c_0^2\right)+1\right)}{8 c_0^5}\eps^2+ib_1(\eps)\eps^3&\frac{i \left(3 \left(c_0^6-c_0^4+c_0^2\right)+1\right)}{8 c_0^5}\eps^2+ib_1(\eps)\eps^3&0\\
        0&0&0
    \end{bmatrix}\\
   & +\mu \begin{bmatrix}
        \frac{ic_0^3}{c_0^2+1}+\eps^2i\big[b_3(\eps)+b_4(\eps)\big] & \eps^2ib_4(\eps) &0\\
        -\eps^2ib_4(\eps) &  \frac{ic_0^3}{c_0^2+1}+\eps^2i\big[b_3(\eps)-b_4(\eps)\big] &0\\
        0&0&\frac{ic_0^3}{c_0^2-1}+\eps^2ib_5(\eps)
    \end{bmatrix}\\
   & +\mu^2\begin{bmatrix}
        \frac{i c_0^3}{\left(c_0^2+1\right)^3}+\eps ib_6(\eps) & \eps ib_7(\eps) & 0 \\
 * & -\frac{i c_0^3}{\left(c_0^2+1\right)^3}+\eps ib_8(\eps) & 0 \\
 0 & 0 & -\frac{i c_0^3}{\left(c_0^2-1\right)^3}+\eps ib_9(\eps) \\
    \end{bmatrix}\\
    &+\mu^3\begin{bmatrix}
        ia_1(\mu,\eps)&ia_2(\mu,\eps)&0\\
        *&ia_3(\mu,\eps)&0\\
        0&0&ia_6(\mu,\eps)
    \end{bmatrix}\\
    &\hspace{-2em}+\nu\highDelim{(}
\begin{bmatrix}
        0&0&-\frac{i \left(c_0^2+1\right)}{2 c_0}\eps-i\eps^3b_2(\eps)\\
        0&0&\frac{i \left(c_0^2+1\right)}{2 c_0}\eps+i\eps^3b_2(\eps)\\
        *&*&0
    \end{bmatrix}+\mu\begin{bmatrix}
         0 & 0 & -\eps\frac{i \left(c_0^6-3 c_0^4-c_0^2-1\right)}{4 c_0 \left(c_0^4-1\right)}+\eps^2ib_{10}(\eps) \\
 0 & 0 & -\eps\frac{i \left(c_0^6-3 c_0^4-c_0^2-1\right)}{4 c_0 \left(c_0^4-1\right)}+\eps^2ib_{11}(\eps) \\
 *&* & 0
    \end{bmatrix}\\
    &+\mu^2\begin{bmatrix}
        0&0&\eps i b_{12}(\eps)\\
        0&0&\eps i b_{13}(\eps)\\
        *&*&0
    \end{bmatrix}+\mu^3\begin{bmatrix}
        0&0&i a_4(\mu,\eps)\\
        0&0& i a_5(\mu,\eps)\\
        *&*&0
    \end{bmatrix}
    \highDelim{)},
\end{aligned}
\end{equation}
where $\mu=\nu^2$ and the  $*$ terms are determined using the symmetry \eqref{eq:symm_Lmat}, and 
\[
a_1(\mu,\eps),\dots,a_6(\mu,\eps)\in\anaR{\mu,\eps},\quad b_1(\eps),\dots,b_{13}(\eps)\in\anaR{\eps}.
\]
\end{prop}
\begin{rema} We recall from \eqref{Nu2:BVp} that $ [\nu^2](\B_{\nu,\eps}V_+^{\nu,\eps},V_+^{\nu,\eps})=0$. It can be proven (but more difficult!) that  $[\nu^2](\B_{\nu,\eps}V_-^{\nu,\eps},V_-^{\nu,\eps})=0$ which then implies $b_4(\eps)=0$. It turns out that the vanishing of $b_4(\eps)$ will not be needed for our spectral analysis of $\Lmat$ in the next section. 
  \end{rema}
\section{Spectral analysis}
\def\poly{\mathtt{p}}
\def\polyNew{\mathtt{p}_\mathtt{n}}
\def\phparen#1{\phi^{(#1)}}
\def\deri{\mathtt{d}}
Having the expansion \eqref{eq:Lmat_final} of the $3\times 3$ matrix $\Lmat_{\nu,\eps}$, we proceed to analyze its spectrum. For the Benjamin--Feir instability of irrotational Stokes waves,  \cite{Berti_Full} obtained a  $4\times 4$ matrix and analyzed it by a  block diagonalization procedure inspired by KAM theory. In this section,  we  develop a different, direct method to  find all roots of the cubic characteristic polynomial $\poly(\lambda;\mu,\eps)$ of $\Lmat_{\nu,\eps}$. After noticing that the roots $\ld$ are proportional to  $\mu$, we consider the rescaling polynomial $ \polyNew(f;\mu,\eps)$ \eqref{def:polyNew} whose coefficients are analytic in $(\mu, \eps)$.  By explicit calculations, $\polyNew(f; 0, 0)$ has a simple root $\phi^{(0)}\in i\Rr$ and a double root $\tilde \phi$. We then perform an analytic implicit function argument to show that the simple root $\phi^{(0)}$ continues to a simple root $\phi_*(\mu, \eps)\in i\Rr$ as $(\mu, \eps)$ is perturbed from $(0, 0)$. Moreover, by a Newton-like method, we obtain the second-order expansion in $(\mu, \eps)$ of $\phi_*(\mu, \eps)$. After that, the two remaining roots $\phi_j(\mu, \eps)$, $j=1,2$, of $\polyNew$ are obtained as roots of a quadratic polynomial whose coefficients depend on $\phi_*$, $[f^1]\polyNew$, and $[f^2]\polyNew$. Using symbolic computation in Mathematica, we can expand the discriminant $\Delta$ up to second order and deduce that $\Delta<0$ when $\gamma>\frac{2}{\sqrt{3}}$. Finally, when $\gamma<\frac{2}{\sqrt{3}}$, another application of the analytic implicit function allows us to determine a critical value $\underline{\mu}(\eps)$ across which  $\Delta$ changes sign.
\subsection{Technical note: Wolfram Mathematica}
The Mathematica codes in this section are selected from the companion file \verb|brute_force_eigenvalues.nb|. We implement a function that displays the coefficients of a polynomial \texttt{poly} in two variables \texttt{var1} and \texttt{var2}, in tabular form. If an entry contains more than $500$ characters, we replace it with $``\dots"$ in the output. 
\begin{mmacode}
  \texttt{\lstinline|showCoefficients[poly_, {var1_, var2_}, cutOff_:500]|}
\end{mmacode}
 For example,
\begin{mmain}
    $\text{showCoefficients}(b \mu  \eps ^2+c \eps ^3+d \mu ^4+e+2^{10000} \mu ^6 \eps ^6+\mu ^2 \eps ,\{\eps ,\mu \})$
\end{mmain}
\begin{mmaout}
    \[
    \left(
\begin{array}{cccccccc}
 \text{} & 1 & \mu  & \mu ^2 & \mu ^3 & \mu ^4 & \mu ^5 & \mu ^6 \\
 1 & e & 0 & 0 & 0 & d & 0 & 0 \\
 \eps  & 0 & 0 & 1 & 0 & 0 & 0 & 0 \\
 \eps ^2 & 0 & b & 0 & 0 & 0 & 0 & 0 \\
 \eps ^3 & c & 0 & 0 & 0 & 0 & 0 & 0 \\
 \eps ^4 & 0 & 0 & 0 & 0 & 0 & 0 & 0 \\
 \eps ^5 & 0 & 0 & 0 & 0 & 0 & 0 & 0 \\
 \eps ^6 & 0 & 0 & 0 & 0 & 0 & 0 & \text{...} \\
\end{array}
\right)
    \]
\end{mmaout}
All the variables $b,c,d,e$ are put into the coefficients in corresponding positions, and the long coefficient $2^{10000}$ of $\eps^6\mu^6$ is replaced by $``\dots"$.

The analytic functions $a_{k}(\mu,\eps)$ and $b_k(\eps)$ in \eqref{eq:Lmat_final} are power series in $\eps$ and possibly $\mu$. However, the dependence on $\eps$ and $\mu$ is not explicitly coded in Mathematica, and the software treats them as independent variables. This causes no issue for our purpose, as our use of Mathematica is restricted to algebraic manipulations, every step of which is a polynomial in $a_k(\mu,\eps)$ and  $b_k(\eps)$. There is only one point to keep in mind when interpreting the output table: lower order terms with coefficients $a_{k}(\mu,\eps)$ and $b_k(\eps)$ may implicitly contribute to higher order terms. For example, consider the analytic function
\[
f_{example}(\mu,\eps)=\mu ^3 a_{1}(\mu,\eps)+\eps ^3 a_{2}(\mu,\eps)+5 \mu ^3 \eps ^3+3 \mu ^2 \eps ^2+1,
\]
which is displayed as 
\begin{mmaout}
    \[
    \left(
\begin{array}{ccccc}
 \text{} & 1 & \mu  & \mu ^2 & \mu ^3 \\
 1 & 1 & 0 & 0 & a_{1} \\
 \eps  & 0 & 0 & 0 & 0 \\
 \eps ^2 & 0 & 0 & 3 & 0 \\
 \eps ^3 & a_{2} & 0 & 0 & 5 \\
\end{array}
\right)
    \]
\end{mmaout}
From this table we may read that 
\[
[\mu^0\eps^0]f_{example}=1,[\mu^2\eps^2]f_{example}=3.
\]
On the other hand, the coefficient $[\mu^3\eps^3]f_{example}$ is not necessarily equal to $5$, since it may also be affected by $[\eps^3]a_1$ and $[\mu^3]a_2$. 
\subsection{Main equation}
We first construct the matrix $\Lmat_{\nu,\eps}$ in Mathematica from \eqref{eq:Lmat_final}, using the symmetry \eqref{eq:symm_Lmat}. The characteristic polynomial of $\Lmat_{\nu,\eps}$ is $\det(\lambda I-\Lmat_{\nu,\eps})$,  which is a cubic polynomial  with coefficients in $\anaC{\nu,\eps}$, where $|\nu|<\nu_0$.  It can be readily checked that for any matrix of the form \eqref{eq:Lmat_in_sC}, the coefficients of its characteristic polynomial depend on $\nu$ only through  $\mu=\nu^2$ since $\nu$ only appears as $\nu^2$. Therefore, we denote 
\begin{equation}
    \label{def:poly}
\poly(\lambda;\mu,\eps)=\det(\lambda I-\Lmat_{\nu,\eps})
\end{equation}
and extend its coefficients analytically to $\mu\in (-\mu_0, \mu_0)$, $\mu_0:=\nu_0^2$. Moreover, from Mathematica we find the reality of these coefficients: 
\begin{equation}
\label{eq:real_valued_poly}
    [\lambda^3]\poly=1,\quad [\lambda^2]\poly,[\lambda^0]\poly\in i\anaR{\mu,\eps},\quad [\lambda^1]\poly \in\anaR{\mu,\eps}.
\end{equation}
Next, we make the change of variables
\begin{equation}
    \lambda=\mu f,
\end{equation}
and define  new polynomial of $f$,
\begin{equation}
    \label{def:polyNew}
    \polyNew(f;\mu,\eps):=\frac{\poly(\mu f;\mu,\eps)}{\mu^3}.
\end{equation}
We can check using Mathematica that $\polyNew$ contains no negative powers of $\mu$, so that 
\begin{equation}
     \polyNew(f;\mu,\eps)\in\anaC{\mu,\eps}[f].
\end{equation}
Clearly \eqref{eq:real_valued_poly} also holds for $\polyNew$:
\begin{equation}
\label{eq:real_valued_polyNew}
    [f^3]\polyNew=1,\quad [f^2]\polyNew,[f^0]\polyNew\in i\anaR{\mu,\eps},\quad [f^1]\polyNew \in\anaR{\mu,\eps}.
\end{equation}
\subsection{First root of $\polyNew$}
To begin with, we let $\mu=\eps=0$ in $\polyNew$ and factorize the obtained polynomial:
\begin{equation}
    \label{def:poly0}
    \poly_0(f):=\polyNew(f;0,0)=\frac{\left(-c_0^2 f+i c_0^3+f\right) \left(i c_0^2 f+c_0^3+i f\right)^2}{\left(c_0-1\right) \left(c_0-i\right)^2 \left(c_0+i\right)^2 \left(c_0+1\right)}=\left(f-\frac{i c_0^3}{c_0^2-1}\right)\left(f-\frac{i c_0^3}{c_0^2+1}\right)^2.
\end{equation}
Thus  $\poly_0(f)$ has a simple root $\phparen{0}=\frac{i c_0^3}{c_0^2-1}$ and a double root $\wt\phi=\frac{i c_0^3}{c_0^2+1}$, and 
\begin{equation}
    \label{def:deri}
    \deri:=\p_f\poly_0(\phparen{0})=(\phparen{0}-\wt\phi)^2=-\frac{4 c_0^6}{(c_0^2-1) (c_0^2+1)^2}\ne0.
\end{equation}
Therefore, by applying the analytic implicit function theorem, we obtain
\begin{prop}
 \label{prop:IFT_near_i}   The polynomial $\polyNew(f;\mu,\eps)$ has a unique complex-valued real analytic root  $f=\phi_*(\mu,\eps)$ in a neighborhood of $(\mu,\eps)=(0,0)\in\R^2$ satisfying $\phi_*(0,0)=\phparen{0}$.
\end{prop}
By Proposition \ref{prop:IFT_near_i}, we have $\phi^*(\mu, \eps)-\phi^{(0)}=O(\mu, \eps)$.   We describe an algorithm to lift $\phparen{0}$ to a higher-order approximation of $\phi_*$. This is a less precise variant of Newton's method.
Suppose $\phparen{j}$ is an approximation to $\phi_*$. Linearizing $\polyNew(f;\mu,\eps)$ about  $f=\phparen{j}$ gives
\[
\polyNew(f;\mu,\eps)\approx \polyNew(\phparen{j};\mu,\eps)+\deri(f-\phparen{j}),\quad \deri=\p_f\poly_0(\phparen{0}).
\]
Setting the right-hand side to zero, we get 
\[
f=\phparen{j}-\frac{\polyNew(\phparen{j};\mu,\eps)}{\deri}
\]
as a higher-order approximation to $\phi_*$. In practice, it is sufficient to retain the lowest order terms of $\polyNew(\phparen{j};\mu,\eps)$:
\begin{equation}
    \phparen{j+1}=\phparen{j}-\frac{\text{lowest order terms of  }\polyNew(\phparen{j};\mu,\eps)}{\deri}.
\end{equation}

Beginning with $\phparen{0}$, the procedure is carried out recursively as follows:
\bq
    \polyNew(\phparen{0};\mu,\eps)=-\frac{4 i c_0^9}{\left(c_0^2-1\right)^5 \left(c_0^2+1\right)^2}\mu+0\eps+O(\mu^2,\mu\eps,\eps^2),
    \eq
    \bq
    \phparen{1}:=\phparen{0}-\frac{-\frac{4 i c_0^9}{\left(c_0^2-1\right)^5 \left(c_0^2+1\right)^2}\mu+0\eps}{\deri},
\eq
\begin{align*}
        \polyNew(\phparen{1};\mu,\eps)&=\frac{\mu ^2 \left(4 i a_6 c_0^6\right)}{\left(c_0^2-1\right)^2 \left(c_0^2+1\right)^2}+\frac{\mu  \eps  \left(4 i b_9 c_0^6\right)}{\left(c_0^2-1\right)^2 \left(c_0^2+1\right)^2}\\
        &-\frac{\eps ^2 \left(i \left(-8 b_5 c_0^6+2 c_0^9-7 c_0^7+c_0^5+3 c_0^3+c_0\right)\right)}{2 \left(c_0^2-1\right)^2 \left(c_0^2+1\right)^2}+O(\mu^3,\mu^2\eps,\mu\eps^2,\eps^3),
        \end{align*}
 \bq\label{form:phiparen2}   \phparen{2}:=\phparen{1}-\frac{\frac{\mu ^2 \left(4 i a_6 c_0^6\right)}{\left(c_0^2-1\right)^2 \left(c_0^2+1\right)^2}+\frac{\mu  \eps  \left(4 i b_9 c_0^6\right)}{\left(c_0^2-1\right)^2 \left(c_0^2+1\right)^2}-\frac{\eps ^2 \left(i \left(-8 b_5 c_0^6+2 c_0^9-7 c_0^7+c_0^5+3 c_0^3+c_0\right)\right)}{2 \left(c_0^2-1\right)^2 \left(c_0^2+1\right)^2}}{\deri}.
\eq
In order to rigorously bound the error $|\phparen{2}-\phi_*|$, we shall appeal to the following lemma.
\begin{lemm}
    \label{lemma:error_control}
    If an analytic function $f_1(\mu,\eps)=\phparen{0}+O(\mu,\eps)$  satisfies
    \begin{equation}\label{polyNewf1}
        \polyNew\Big(f_1(\mu,\eps);\mu,\eps\Big)=O(\eps^{m_1}\mu^{n_1},\dots,\eps^{m_p}\mu^{n_p}),    \end{equation}
        then 
        \begin{equation}\label{est:phistar}
            \phi_*(\mu,\eps)=f_1(\mu,\eps)+O(\eps^{m_1}\mu^{n_1},\dots,\eps^{m_p}\mu^{n_p})
        \end{equation}
\end{lemm}
\begin{proof}
 Set $ h(\mu,\eps)=  f_1(\mu,\eps)-\phi_*(\mu,\eps)$, an analytic function. Since   $h(0,0)=\phparen{0}-\phparen{0}=0$, we have $h=O(\mu,\eps)$. Using this and the fact that $\phi_*(\mu, \eps)-\phi^{(0)}=O(\mu, \eps)$, we can expand
    \begin{align*}
        \polyNew\Big(f_1(\mu,\eps);\mu,\eps\Big)&=\polyNew\Big(\phi_*(\mu,\eps)+h(\mu,\eps);\mu,\eps\Big)\\
        &=\polyNew\Big(\phi_*(\mu,\eps);\mu,\eps\Big)+h(\mu,\eps)\p_f\polyNew\Big(\phi_*(\mu,\eps);\mu,\eps\Big)+O(h^2)\\
          &=\polyNew\Big(\phi_*(\mu,\eps);\mu,\eps\Big)+h(\mu,\eps)\Big[\p_f\polyNew\Big(\phparen{0};\mu,\eps\Big)+O(\mu,\eps)\Big]+O(h^2)\\
        &=\deri h(\mu,\eps)+h(\mu,\eps)O(\mu, \eps).
    \end{align*}
Combining this with \eqref{polyNewf1} yields
\begin{equation*}
    \deri h(\mu,\eps)+h(\mu,\eps)O(\mu, \eps)=O(\eps^{m_1}\mu^{n_1},\dots,\eps^{m_p}\mu^{n_p})
\end{equation*}
which then implies \eqref{est:phistar}. 
\end{proof}
Inserting the formula \eqref{form:phiparen2} of $\phi^{(2)}$ into $\polyNew$, we find using  Mathematica that 
\[
\polyNew(\phparen{2};\mu,\eps)=O(\mu^3,\mu^2\eps,\mu\eps^2,\eps^3).
\]
 Hence,  applying Lemma \ref{lemma:error_control} to $f_1=\phi^{(2)}$ gives
\begin{equation}
\label{eq:phi_star_phi_3_temp}
    \phi_*=\phparen{2}+O(\mu^3,\mu^2\eps,\mu\eps^2,\eps^3).
\end{equation}
Next, it follows from the properties in \eqref{eq:real_valued_polyNew} that $\ol{\polyNew(f;\mu,\eps)}=-\polyNew(-\ol f;\mu,\eps)$. Consequently, both $\phi_*$ and $-\ol{\phi_*}$ are roots of $\polyNew$. These roots are equal when $\mu=\eps=0$ since $\phi^{(0)}\in i\Rr$.  By the uniqueness in Proposition \ref{prop:IFT_near_i}, we conclude that $\phi_*= -\ol{\phi_*}$, and thus $\phi_*\in i\Rr$. Moreover, by applying  Lemma \ref{lemma:error_control} to $f_1=\phi^{(1)}$, we obtain 
\begin{equation}
\label{eq:phi_star_final_form}
    \begin{aligned}
        \phi_*&=\phi^{(1)}+ir_0(\mu^2,\mu\eps,\eps^2)=\frac{i c_0^3}{c_0^2-1}-\frac{ic_0^3}{(c_0^2-1)^4}\mu+ir_0(\mu^2,\mu\eps,\eps^2),\quad r_0\in \Rr.
    \end{aligned}
\end{equation}
\subsection{Two roots by quadratic formula}
We denote the remaining two roots of the cubic polynomial $\polyNew(f;\mu,\eps)$ by  $\phi_1(\mu,\eps),\phi_2(\mu,\eps)$. Then $\phi_1$ and $\phi_2$ are roots of the quadratic polynomial 
\begin{equation}
    \label{eq:quad_phi1_phi2}
    f^2+pf+q=0,
\end{equation}
where $p:=-(\phi_1+\phi_2)$ and   $q:=\phi_1\phi_2$. Invoking the relations
\begin{align*}
    -[f^2]\polyNew&=\phi_*+\phi_1+\phi_2,\\
   [f^1]\polyNew&=\phi_*\phi_1+\phi_*\phi_2+\phi_1\phi_2= \phi_*(\phi_1+\phi_2)+\phi_1\phi_2,
\end{align*}
we find
\[
    p=\phi_*+[f^2]\polyNew,\quad  q=\phi_*p+[f^1]\polyNew.
\]
Since $\phi_*,[f^2]\polyNew\in i\anaR{\mu,\eps}$ and $[f^1]\polyNew\in\anaR{\mu,\eps}$, we have $p\in i\anaR{\mu,\eps}$ and $q\in \anaR{\mu,\eps}$. Consequently, the discriminant of \eqref{eq:quad_phi1_phi2} is
\begin{equation}
    \label{def:Delta_quadratic}
    \Delta=\Delta(\mu,\eps):=p^2-4q\in\anaR{\mu,\eps}.
\end{equation}
We substitute the second-order  expansions of $\phi_*$ (given by \eqref{form:phiparen2}-\eqref{eq:phi_star_phi_3_temp}) and $ [f^1]\polyNew$, $[f^2]\polyNew$ (directly obtained from  \eqref{eq:Lmat_final} by symbolic computation in Mathematica) 
 to find second-order expansions  $p$ and $\Delta$:
\begin{mmain}
 (*In the output below for $p$ and $\Delta$, we  only keep the coefficients of  $\mu^m\eps^e$ with $m+e\le 2$ and set the higher-order coefficients to $0$.*)
 
    $\text{showCoefficients}(p,\{\eps ,\mu \}) $
    
    $\text{showCoefficients}(\text{quadraticDelta},\{\eps ,\mu \},100)$
\end{mmain}
\begin{mmaout}
    \[
  \left(
\begin{array}{cccc}
 \text{} & 1 & \mu  & \mu ^2 \\
 1 & -\frac{2 i c_0^3}{c_0^2+1} & 0 & -i \left(a_1+a_3\right) \\
 \epsilon  & 0 & -i \left(b_6+b_8\right) & 0 \\
 \epsilon ^2 & -\frac{i \left(16 b_3 c_0^5+2 c_0^8-7 c_0^6+c_0^4+3 c_0^2+1\right)}{8 c_0^5} & 0 & 0 \\
\end{array}
\right)
    \]
    \[\left(
\begin{array}{cccc}
 \text{} & 1 & \mu  & \mu ^2 \\
 1 & 0 & 0 & -\frac{4 c_0^6}{\left(c_0^2+1\right)^6} \\
 \epsilon  & 0 & 0 & 0 \\
 \epsilon ^2 & \frac{-c_0^6+3 c_0^4-c_0^2+3}{\left(c_0^2+1\right)^3} & 0 & 0 \\
\end{array}
\right)
    \]
\end{mmaout}
Since $p$ is purely imaginary and $\Delta $ is real, we can read off the preceding output that
\begin{align}
    p&=-\frac{2 i c_0^3}{c_0^2+1}+ir_1(\mu^2,\mu\eps,\eps^2),\\
    \label{eq:Delta_quadratic}
    \Delta&=-\frac{4 c_0^6}{\left(c_0^2+1\right)^6}\mu^2+\frac{-c_0^6+3 c_0^4-c_0^2+3}{\left(c_0^2+1\right)^3}\eps^2+r_2(\mu^3,\mu^2\eps,\mu\eps^2,\eps^3),
\end{align}
where $r_1$, $r_2\in \Rr\{\mu, \eps\}$ depend only on $\gamma\ne 0$.  By the quadratic formula, the roots  $\phi_1$ and $\phi_2$ are given by 
\begin{equation}\label{phi12-final}
    \phi_{1,2}=\mez \left(\frac{2 i c_0^3}{c_0^2+1}-ir_1(\mu^2,\mu\eps,\eps^2)\pm\sqrt{-\frac{4 c_0^6}{\left(c_0^2+1\right)^6}\mu^2+\frac{-c_0^6+3 c_0^4-c_0^2+3}{\left(c_0^2+1\right)^3}\eps^2+r_2(\mu^3,\mu^2\eps,\mu\eps^2,\eps^3)} \right).
\end{equation}
Recalling the scaling  $\lambda=\mu f$,  we obtain from the formulas \eqref{eq:phi_star_final_form} and \eqref{phi12-final} that  the eigenvalues  of $\Lmat_{\nu,\eps}$ are 
\begin{align}
\label{eq:lambda_star_last_section}
    \lambda_*&=\frac{i c_0^3}{c_0^2-1}\mu-\frac{ic_0^3}{(c_0^2-1)^4}\mu^2+i\mu r_0(\mu^2,\mu\eps,\eps^2),\\
    \label{eq:lambda_12_last_section}
    \lambda_{1,2}&=\mez \left(\frac{2 i c_0^3}{c_0^2+1}\mu-i\mu r_1(\mu^2,\mu\eps,\eps^2)\pm\mu\sqrt{-\frac{4 c_0^6}{\left(c_0^2+1\right)^6}\mu^2+\frac{-c_0^6+3 c_0^4-c_0^2+3}{\left(c_0^2+1\right)^3}\eps^2+r_2(\mu^3,\mu^2\eps,\mu\eps^2,\eps^3)} \right).
\end{align}
Since 
\[
-c_0^6+3 c_0^4-c_0^2+3=(3-c_0^2)(c_0^4+1),
\]
we have 
\begin{equation}\label{signDelta:gamma}
    -c_0^6+3 c_0^4-c_0^2+3\le 0 \iff c_0\ge \sqrt{3}\iff \gamma\ge \frac{2}{\sqrt3}.
\end{equation}

\subsection{Proof of Theorem \ref{thm:main}}
We have proven that there are  $\mu_1,\eps_1>0$ and a neighborhood  $U$ of  $0$ in $\C$, all depending on $\gamma\ne 0$, such that for any $\condMuEps{1}{1}$, the spectrum of $\L_{\nu,\eps}=\L_{\nu,\eps}^+$ in $U$ consists of the eigenvalues  $\{\ld_*,\ld_1,\ld_2\}$ given by \eqref{eq:lambda_star_last_section}-\eqref{eq:lambda_12_last_section}. To prove the two parts of Theorem \ref{thm:main}, we consider the following two cases. 

 \underline{Case 1: $\gamma>\frac{2}{\sqrt3}$}.  Then $\frac{-c_0^6+3 c_0^4-c_0^2+3}{\left(c_0^2+1\right)^3}<0$ in view of \eqref{signDelta:gamma}. Since 
 \[
 |r_2(\mu^3,\mu^2\eps,\mu\eps^2,\eps^3)|=o(|\mu|^2+|\eps|^2),
 \]
  there exist $\mu_2,\eps_2$ depending on $\gamma$ such that, for any $|\eps|<\eps_2$ and $|\mu|<\mu_2$, we have
\begin{equation}
   \Delta= -\frac{4 c_0^6}{\left(c_0^2+1\right)^6}\mu^2+\frac{-c_0^6+3 c_0^4-c_0^2+3}{\left(c_0^2+1\right)^3}\eps^2+r_2(\mu^3,\mu^2\eps,\mu\eps^2,\eps^3)<0.
\end{equation}
This proves part (i) of Theorem \ref{thm:main}.

\underline{Case 2: $\gamma<\frac{2}{\sqrt3}$}. Then $\frac{-c_0^6+3 c_0^4-c_0^2+3}{\left(c_0^2+1\right)^3}>0$. There exist $r_3$, $r_4\in \Rr \{\mu,\eps\}$ depending only on $\gamma$,  and $r_j=O(\mu, \eps)$, such that 
\begin{equation}
\label{def:r3_r4}
\begin{aligned}
       \Delta
       &=\frac{-c_0^6+3 c_0^4-c_0^2+3}{\left(c_0^2+1\right)^3}\eps^2\Big(1+r_3(\mu,\eps)\Big)-\frac{4 c_0^6}{\left(c_0^2+1\right)^6}\mu^2\Big(1+r_4(\mu,\eps)\Big).
\end{aligned}
\end{equation}
Then we can choose $\mu_3<\mu_2$ and $\eps_3<\eps_2$ such that $|r_3|$, $|r_4|<\mez$ for all $\condMuEps{3}{3}$. Then we can find $r_5$, $r_6\in \Rr\{\mu, \eps\}$, $r_j=O(\mu, \eps)$, such that 
\begin{equation}
    1+r_3(\mu,\eps)=\Big(1+r_5(\mu,\eps)\Big)^2,\quad 1+r_4(\mu,\eps)=\Big(1+r_6(\mu,\eps)\Big)^2.
\end{equation}
Consequently, we can factorize 
\begin{equation}
    \begin{aligned}
        \Delta 
        &=\Bigg\{\sqrt{\frac{-c_0^6+3 c_0^4-c_0^2+3}{\left(c_0^2+1\right)^3}}\eps\Big(1+r_5(\mu,\eps)\Big)+\sqrt{\frac{4 c_0^6}{\left(c_0^2+1\right)^6}}\mu\Big(1+r_6(\mu,\eps)\Big)\Bigg\}\\
        &\qquad \cdot \Bigg\{\sqrt{\frac{-c_0^6+3 c_0^4-c_0^2+3}{\left(c_0^2+1\right)^3}}\eps\Big(1+r_5(\mu,\eps)\Big)-\sqrt{\frac{4 c_0^6}{\left(c_0^2+1\right)^6}}\mu\Big(1+r_6(\mu,\eps)\Big)\Bigg\}\\
        &=:\Delta_+(\mu,\eps)\Delta_-(\mu,\eps).
    \end{aligned}
\end{equation}
Since $\Delta_+(0,0)=0$ and 
\begin{equation}
    \partial_\mu\Delta_+(0,0)=\sqrt{\frac{4 c_0^6}{\left(c_0^2+1\right)^6}}\ne 0,
\end{equation}
the analytic implicit function theorem implies that  there exists a unique real-valued real analytic function $\critmu_1(\eps)$ for $|\eps|$ small enough, such that $\Delta_+(\critmu_1(\eps),\eps)=0$. In addition, by a proof similar to that of  Lemma \ref{lemma:error_control}, since
\begin{equation}
    \Delta_+\left(-\sqrt{\frac{(c_0^2+1)^3(-c_0^6+3 c_0^4-c_0^2+3)}{4c_0^6}}\eps,\eps\right)=O(\eps^2),
\end{equation}
we have 
\begin{equation}
    \critmu_1(\eps)=-\sqrt{\frac{(c_0^2+1)^3(-c_0^6+3 c_0^4-c_0^2+3)}{4c_0^6}}\eps+O(\eps^2).
\end{equation}
Similarly, we obtain a real-valued real analytic function 
\begin{equation}
    \critmu_2(\eps)=\sqrt{\frac{(c_0^2+1)^3(-c_0^6+3 c_0^4-c_0^2+3)}{4c_0^6}}\eps+O(\eps^2)
\end{equation}
such that $\Delta_-(\critmu_2(\eps),\eps)=0$. 

We define the piecewise analytic function $\critmu:(-\eps_4,\eps_4)\to [0,\mu_4)$, 
\begin{equation}
    \critmu(\eps)=\begin{cases}
        \critmu_1(\eps)&\eps<0\\
        \critmu_2(\eps)&\eps\ge 0
    \end{cases},
\end{equation}
The above analysis shows that there exist  $\mu_4<\mu_3,\eps_4<\eps_3$ depending on $\gamma$ such that for any $\condMuEps{4}{4}$, we have 
\begin{equation}
 \Delta(\mu,\eps)\lesseqqgtr0 \iff \mu\gtreqqless\critmu(\eps).
\end{equation}
This proves part (ii) of  Theorem \ref{thm:main}.

{\bf Maximum growth rate}. For each fixed $\eps\in (-\eps_4, \eps_4)$, we  find the maximum growth rate when $\mu$ varies in $(0, \mu_4)$.We note that $\RE(\lambda_{1,2})\ne 0$  only when $\mu<\critmu(\eps)$. In that case, we have
\[
    \RE(\lambda_{1})=-\mez \mu\sqrt{\Delta}<\RE(\ld_2)=\mez \mu\sqrt{\Delta}.
\]
For  the maximum of $\mu^2\Delta(\mu,\eps)$, we compute  
\begin{align*}
    \p_\mu(\mu^2\Delta(\mu,\eps))
    =\mu\Big(-\frac{16 c_0^6}{\left(c_0^2+1\right)^6}\mu^2+\frac{2(-c_0^6+3 c_0^4-c_0^2+3)}{\left(c_0^2+1\right)^3}\eps^2+r_2(\mu^3,\mu^2\eps,\mu\eps^2,\eps^3)\Big).
\end{align*}
Similarly to the analysis of $\Delta$ above, there exist  $\mu_5<\mu_4,\eps_5<\eps_4$ depending on $\gamma$, and a piecewise analytic function $\mu_m(\eps)$ in $\eps<0$ and $\eps\ge0 $ of the form 
\begin{equation}
    \mu_m(\eps)=\sqrt{\frac{(c_0^2+1)^3(-c_0^6+3 c_0^4-c_0^2+3)}{8c_0^6}}|\eps|+O(\eps^2)
\end{equation}
such  that for any $\condMuEps{5}{5}$ we have 
\[
 \p_\mu(\mu^2\Delta(\mu,\eps))\lesseqqgtr0\iff \mu\gtreqqless\mu_m(\eps).
\]
We note that $\mu_m(\eps)<\critmu(\eps)$ for $\eps\ne 0$. Consequently, for any fixed $\gamma\in \Rr\setminus\{0, \frac{2}{\sqrt3}\}$ and $\eps\in (-\eps_5, \eps_5)$, we have 
\bq\label{maxReld}
    \max_{\mu\in (0, \mu_5)}\RE(\lambda_2(\mu, \eps))=\mez \mu_m\sqrt{\Delta(\mu_m,\eps)}=\frac{(-c_0^6+3 c_0^4-c_0^2+3)\eps^2}{8c_0^3}+O(\eps^3).
\eq
{\bf Approximate figure-$8$ curve}.  We recall that for  $0<\mu<\critmu(\eps)$, $\ld_{1,2}$ have nonzero real parts and are given by \eqref{eq:lambda_12_last_section}, which is also valid for $\gamma=0$ by virtue of \cite[Theorem 1.1]{Berti_Full}. The real and imaginary parts of these eigenvalues are approximated by 
\begin{align*}
    \RE(\lambda)=\pm\mez \mu\sqrt{-\frac{4 c_0^6}{\left(c_0^2+1\right)^6}\mu^2+\frac{-c_0^6+3 c_0^4-c_0^2+3}{\left(c_0^2+1\right)^3}\eps^2} ,\quad  \IM(\lambda)=\frac{c_0^3}{c_0^2+1}\mu.
\end{align*}
Then, $(\RE(\lambda)/\eps^2,\IM(\lambda)/\eps)$ form the  figure-$8$ curve  parametrized by $s:=\mu/\eps$:
\bq\label{form:figure8}
    \frac{\RE(\lambda)}{\eps^2}=\pm\mez s\sqrt{-\frac{4 c_0^6}{\left(c_0^2+1\right)^6}s^2+\frac{-c_0^6+3 c_0^4-c_0^2+3}{\left(c_0^2+1\right)^3}},\quad   \frac{\IM(\lambda)}{\eps}=\frac{c_0^3}{c_0^2+1}s.
\eq
 %%%%%%%%%%%%%%%%  APPENDIX  %%%%%%%%%%%%%%%%%%%%%
\appendix\label{appendix}
\section{Stokes waves with constant vorticity}\label{appendix:local_bif}
This appendix is devoted to the proof of Theorem \ref{thm:Stokes_without_g}. 

\subsection{Local bifurcation}
The existence of small-amplitude Stokes waves with constant vorticity was proven in \cite[Theorem 1.1]{2025}. For the sake of completeness and to prepare for the second-order expansion of Stokes waves, we will include the existence proof, which relies on the Crandall-Rabinowitz bifurcation theorem.
\begin{theo}[Crandall-Rabinowitz bifurcation theorem]\label{theo:CR} \cite[Theorem 4.1, Chap.5]{Ambrosetti}) 
     Let $X, Y$ be Banach spaces and $U \subset X$ be an open neighborhood of 0 . Let $F: U \times \mathbb{R} \rightarrow Y, F(u, c)$, be an analytic map satisfying $F(0, c)=0$ for any $c \in \mathbb{R}$. Let $c_0$ be such that $L:=\mathrm{d}_u F\left(0, c_0\right) \in L(X, Y)$ is not invertible and
     \begin{itemize}
         \item $\operatorname{Ker}(L)=\operatorname{span}\{u^*\}$, $u^* \in X$;
         \item the range $\operatorname{Ran}(L)$ is closed and $\operatorname{codim} \operatorname{Ran}(L)=1$;
         \item (transversality) $\partial_c \mathrm{~d}_u F\left(0, c_0\right)\left[u^*\right] \notin \operatorname{Ran}(L)$.
     \end{itemize}
Then there exist $\eps_*>0$ and an analytic function
$$
\left(-\eps_*, \eps_*\right) \rightarrow U \times \mathbb{R}, \quad \eps \mapsto\left(u_\eps, c_\eps\right), \quad u_\eps=\eps u^*+O\left(\eps^2\right), \quad c_\eps=c_0+O(\eps)
$$
such that $F\left(u_\eps, c_\eps\right)=0$ for any $|\eps|<\eps_*$.
\end{theo}
For the Stokes wave system \eqref{eq:main}, we set 
\begin{equation}\label{def:F:Stokes}
    F\left(\vet{\eta}{\psi},c\right)=\Bigg[c\eta_x+G(\eta) \psi+\gamma \eta \eta_x,~c\psi_x-g \eta-\frac{{\psi_x^2}}{2}+\frac{\big(\eta_x \psi_x+G(\eta) \psi\big)^2}{2\left(1+\eta_x^2\right)}+\gamma \eta \psi_x+\gamma \partial_x^{-1} G(\eta) \psi\Bigg]^T.
\end{equation}
Following \cite{2025}, we set 
\begin{equation}
    X=H^{\sigma,s}_{\text{ev}}\times H^{\sigma,s}_{\text{odd}}, \quad Y=H_{\text{odd}}^{\sigma, s-1}\times H^{\sigma, s-1}_{\text{ev}}
\end{equation}
for any $\sigma,s$ satisfying
\begin{equation}\label{sigmas}
    \sigma\ge 0,\quad s\ge\frac72, \quad s+\mez \in\Z.
\end{equation}
Here, the analyticity norm is defined by 
\[
  \| u\|_{H^{\sigma, s}}^2:=\sum_{k\in \Zz} e^{2\sigma |k|}(1+k^2)^s|\hat{u}(k)|^2,\quad H^{\sigma, s}=\left\{u\in L^2(\T): \| u\|_{H^{\sigma, s}}<\infty\right\}. 
\] 
By \cite[Theorem 1.2]{2025}, there exists $\eps_0=\eps_0(s)>0$ such that the map
\bq
\eta\mapsto G(\eta),\quad  H^{\sigma, s}\cap B^{\sigma, s}(\eps_0)\to \cL(H^{\sigma, s}, H^{\sigma, s-1})
\eq
 is analytic, and the Neumann series $(1+\eta_x^2)^{-1}$ converge for $\eta\in B^{\sigma, s}(\eps_0)$. It follows that 
 \[
 F: U\times \Rr\to Y,\quad U:=  \left(H^{\sigma, s}_{\text{ev}}\cap B^{\sigma, s}(\eps_0)\right)\times H^{\sigma,s}_{\text{odd}},
 \]
is analytic. 

Let $c\in \Rr$. Clearly $F\left(\vet{0}{0}, c\right)=0$ and the linearization of $F$ about $\vet{0}{0}$ is 
\begin{equation}\label{bifur:opL}
\begin{aligned}
        L=        \begin{bmatrix}
            c\partial_x& G(0)\\
            -g& c\partial_x+\gamma\partial_x^{-1}G(0)
        \end{bmatrix},
\end{aligned}
\end{equation}
where $G(0)=|D|\tanh(h|D|)$ with the convention that $\tanh(hk)=1$ if $h=\infty$. We have 
\begin{equation}
\label{eq:F_linearized}
    \begin{aligned}
        \F( L(\eta, \psi))(k)&=
        \begin{bmatrix}
            ick & k\tanh(hk)\\
            -g& ick-i\gamma \tanh(hk)
        \end{bmatrix}
        \begin{bmatrix}
            \wh{\eta}(k)\\
            \wh{\psi}(k)
        \end{bmatrix}=: M_c(k)
        \begin{bmatrix}
            \wh{\eta}(k)\\
            \wh{\psi}(k)
        \end{bmatrix}
    \end{aligned}
\end{equation}
and 
\begin{equation}\label{dispersion}
    \det(M_c(k))=-c^2k^2+c\gamma k\tanh(hk)+gk\tanh(hk)=:kE_c(k). 
\end{equation} 
Due to the gravity term, the only constant vector in $\operatorname{Ker}(L)$ is $0$. Hence, $\operatorname{Ker}(L)$ is nontrivial if and only if $   \det(M_c(k))=0$ for some $k\in \Zz\setminus\{0\}$. 

We note that, for any $k\in \Zz\setminus\{0\}$,
\bq \label{def:c0pm}
\begin{aligned}
\det(M_c(k))=0 &\iff E_c(k)=0\\
&\iff 
   c= c^\pm_0(k)=\frac{\gamma\tanh(hk)\pm\sqrt{\gamma^2\tanh^2(hk)+4gk\tanh(hk)}}{2k},
    \end{aligned}
\eq
where $c^+_0(k)>0>c^-_0(k)$. 

We fix any $k_0\in \Nn$ and set $c_0=c^+_0(k_0)>0$. Let us show that the unique solution $k\in \Nn=\{1, 2, \dots\}$ of  $E_{c_0}(k)=0$ is $k=k_0$. When  $h=\infty$, this is obvious since $E_{c_0}(k)=-c_0^2k+c_0\gamma+g$ is an affine function. As for the case $h<\infty$, we have 
    \[
    \p_\xi^2E_{c_0}( \xi)=(c_0\gamma+g)h^2(\tanh'')(h\xi)=-2(c_0\gamma+g)h^2\operatorname{sech}^2(h\xi)\tanh(h\xi).
    \]
  If $c_0\gamma+g=0$, then $E_{c_0}(k)=-c_0^2k\ne 0$ for all $k\in \Nn$, contradicting the fact that $E_{c_0}(k_0)=0$. Thus $c_0\gamma+g\ne 0$, and hence  $\p_\xi^2E_{c_0}(\xi)\ne 0$ for all $\xi>0$. By Rolle's theorem,  $E_{c_0}(\cdot)$ has at most two roots on  $[0,\infty)$. But  $E_{c_0}(0)=0$, so $k_0$ is the unique  root in $\Nn$. We have shown that when $c=c_0$, the kernel of the linearized operator  $L$ is one dimensional. Moreover, by a simple calculation, we have $\operatorname{Ker}(L)=\operatorname{span}\{u^*\}$ and $\operatorname{Ker}(L^*)=\operatorname{span}\{u^\dag\}$, where  
\bq\label{kernelcokernel}
u^*=\begin{bmatrix}
    \tanh(hk_0)\cos(k_0 x)\\
    c_0\sin(k_0 x)
\end{bmatrix}, \quad 
u^\dag=\begin{bmatrix}
   g\sin(k_0x)\\
   -c_0k_0\cos(k_0 x)
\end{bmatrix}.
\eq
Moreover, it is easy to check that $f=(f_1, f_2)\in \operatorname{Ran}(L)$ if and only if $f\perp u^\dag$ in $Y$. In particular, we have $\operatorname{codim} \operatorname{Ran}(L)=1$. Since 
\[
    \partial_c \mathrm{~d}_u F\left(0, c_0\right)\vet{\eta}{\psi}=
   \vet{\p_x\eta}{\p_x\psi},
   \]
we have
\begin{align*}
    &\big(u^\dag,  \partial_c \mathrm{~d}_u F\left(0, c_0\right)\left[u^*\right]\big)_Y\\
&=2\pi e^{2\sigma k_0}(1+k_0^2)^{s-1}\big(u^\dag,  \partial_c \mathrm{~d}_u F\left(0, c_0\right)\left[u^*\right]\big)_{(L^2)^2}\\
&=2\pi e^{2\sigma k_0}(1+k_0^2)^{s-1}\int_\T\Big( -gk_0\tanh(hk_0)\sin^2(k_0 x)-c_0^2k_0^2\cos^2(k_0x)\Big)dx<0,
\end{align*}
so that the transversality condition is satisfied. Therefore, by virtue of Theorem \ref{theo:CR}, there is a local curve of Stokes waves with the properties (i) and (ii) in Theorem \ref{thm:Stokes_without_g}. 

\subsection{Second-order expansion}\label{appendix:2nd_expansions}
Let $(\eta, \psi)=O(\eps)$ be a Stoke wave obtained above. From \cite[formula (2.14)]{Craig_Sulem}, we have the expansion 
\begin{equation}
\label{eq:expansion:G_order_eps_3}
\begin{aligned}
G(\eta) \psi & =\underbrace{G_0 \psi}_{O(\eps)}
+ \underbrace{\left( \eta D^2 - \eta_x \partial_x  \right) \psi -G_0(\eta G_0(\psi))}_{O(\eps^2)}\\
&+\underbrace{\left((\eta_x)^2G_0\psi+G_0\big(\eta G_0(\eta G_0\psi)\big)-\frac{1}{2}G_0(\eta^2D^2\psi)-\eta D^2(\eta G_0\psi)+\frac{1}{2}\eta^2 D^2G_0\psi\right)}_{O(\eps^3)}
+ \rem{4}.
\end{aligned}
\end{equation}
where $G_0:=G(0)=|D|\tanh (h|D|)$. We note that each $O(\eps^j)$-term in \eqref{eq:expansion:G_order_eps_3} contains terms of order $\eps^j$ and higher. We proceed to find the coefficients in the following expansions:   
\begin{equation}
\begin{aligned}
       \eta&=\eps\eta_1+\eps^2\eta_2+\rem 3,\\
       \psi&=\eps\psi_1+\eps^2\psi_2+\rem 3,\\
       c&=c_0+\eps c_1+\eps^2c_2+\rem 3.
   \end{aligned}
\end{equation}
In view of \eqref{kernelcokernel}, we have 
\begin{equation}\label{eta1psi1}
\eta_1  = \tanh \left(hk_0\right) \cos \left(k_0 x\right),\quad \psi_1 =c_0\sin \left(k_0 x\right).
\end{equation}
Using \eqref{eq:expansion:G_order_eps_3}, we can expand \eqref{eq:main} and obtain the  $\eps^2$-order system 
\begin{equation}
    \left\{
\begin{aligned}
 0&=\big(c_1 \partial_x \eta_1 + c_0 \partial_x \eta_2\big)+ G_0 \psi_2 + \big(\eta_1 D^2 - \partial_x \eta_1 \partial_x - G_0 \eta_1 G_0\big) \psi_1 
+ \gamma\eta_1 \partial_x \eta_1, \\
0&=\big(c_1 \partial_x \psi_1 + c_0 \partial_x \psi_2\big)  -g \eta_2 - \frac{(\partial_x \psi_1)^2}{2} + \frac{1}{2} (G_0 \psi_1)^2 \\
&\quad + \gamma \eta_1 \partial_x \psi_1 + \gamma \partial_x^{-1} \bigg(G_0 \psi_2 + \big(\eta_1 D^2 - \partial_x \eta_1 \partial_x - G_0 \eta_1 G_0\big) \psi_1 \bigg).
\end{aligned}
\right.
\end{equation}
Substituting \eqref{eta1psi1}, the above equations become 
\begin{equation}
\label{eq1_order2}
    \begin{aligned}
    -c_0 \partial _x \eta_2 -G_0 \psi_2
    &=\Big(\eta_1 D^2-\partial_x \eta_1 \partial_x-G_0 \eta_1 G_0\Big) \psi_1+\gamma \eta_1 \partial_x \eta_1+c_1\partial_x \eta_1\\
    &=\left(
c_0 k_0^2 \tanh(h k_0) 
-  c_0 k_0^2 \tanh(2 h k_0) \tanh^2(h k_0) 
- \frac{1}{2}\gamma{k_0} \tanh^2(h k_0)
\right) \sin(2 k_0 x)\\
&\qquad-c_1 k_0 \tanh(h k_0)\sin(k_0 x)\\
&=:a\sin(2k_0 x)-c_1 k_0 \tanh(h k_0)\sin(k_0 x)
\end{aligned}
\end{equation}
and 
\begin{equation}
\label{eq2_order2}
    \begin{aligned}
    &\quad(-c_0\partial_x-\gamma\partial_x^{-1} G_0)\psi_2+g\eta_2\\
    &=\frac{1}{4}\Big(-c_0^2 k_0^2 + c_0^2 k_0^2 \tanh^2(h k_0) + 2 \gamma k_0 c_0 \tanh(h k_0)\Big)\\
    &\qquad+\frac{1}{4}\Big(-c_0^2k_0^2-c_0^2k_0^2 \tanh^2(h k_0)+2\gamma c_0 k_0\tanh(2hk_0)\tanh^2(hk_0)\Big)\cos(2k_0 x)+c_0 c_1 k_0\cos(k_0 x)\\
    &=:b_0+b\cos(2k_0 x)+c_0 c_1 k_0\cos(k_0 x),
\end{aligned}
\end{equation}
where the coefficients  $a$, $b$, and $b_0$ depend only  on $(c_0, h, k_0)$. Comparing these equations to \eqref{bifur:opL}, we find 
\[
-L\begin{bmatrix}
\eta_2\\ \psi_2
\end{bmatrix}=\begin{bmatrix}
a\sin(2k_0 x)-c_1 k_0 \tanh(h k_0)\sin(k_0 x)\\
b_0+b\cos(2k_0 x)+c_0 c_1 k_0\cos(k_0 x)
\end{bmatrix}=:F,
\]
where, at frequency $k$, $L$ is given by the matrix $M_{c_0}(k)$ in \eqref{eq:F_linearized}. Since  $\eta_2$ is even and $\psi_2$ is odd, we have 
\begin{equation}
    \begin{aligned}
    \eta_2=\frac{1}{2\pi}\hat{\eta}_2(0)+\frac{1}{\pi}\sum_{k\ge 1} \hat{\eta}_2(k) \cos(kx),\quad 
    \psi_2=\frac{i}{\pi}\sum_{k\ge 1} \hat{\psi}_2(k) \sin(kx).
\end{aligned}
\end{equation}
Integrating \eqref{eq2_order2} over $\T$ yields $ \hat{\eta}_2(0)=\frac{2\pi b_0}{g}$. 
For $k\ne 0$, we have  $\F(\cos(kx))(k)=\pi$ and $\F(\sin(kx))(k)=\frac{\pi}{i}$. In particular, at $k=k_0$, we have 
\begin{equation}
    -M_{c_0}(k_0)
    \begin{bmatrix}
        \hat{\eta}_2(k_0)\\
        \hat{\psi}_2(k_0)
    \end{bmatrix}
    =
    \begin{bmatrix}
        -\frac{\pi}{i}c_1 k_0 \tanh(h k_0)\\
        \pi c_0 c_1 k_0
    \end{bmatrix}.
\end{equation}
Recalling that $M_{c_0}(k_0)$ is singular, we choose  the trivial solution $\hat{\eta}_2(k_0)=\hat{\psi}_2(k_0)=c_1=0$. It remains to consider $k\ge 1$ and $k\ne k_0$. Then $M_{c_0}(k)$ is invertible and we obtain a unique solution $\begin{bmatrix} \hat{\eta}_2(k) \\ \hat{\psi}_2(k)\end{bmatrix}=M_{c_0}(k)^{-1}\hat F(k)$. Since $F$ only has frequencies $k_0$ and $2k_0$, we have $\hat{\eta}_2(k)=\hat{\psi}_2(k)=0$ for  $k\ne 2k_0$. Finally,  at $k=2k_0$,  the system reads
\begin{equation}\label{sys:2k0}
-M_{c_0}(2k_0)\begin{bmatrix}
    \hat{\eta}_2(2k_0)\\
    \hat{\psi}_2(2k_0)
\end{bmatrix}=
    \left[
\begin{array}{cc}
- ic_0  2 k_0 & - 2k_0 \tanh(h2k_0) \\
g & i \gamma \tanh(h2k_0) - ic_02k_0
\end{array}
\right]
\begin{bmatrix}
    \hat{\eta}_2(2k_0)\\
    \hat{\psi}_2(2k_0)
\end{bmatrix}
 =\begin{bmatrix}
\frac{\pi}{i}a \\
\pi b
\end{bmatrix}.
\end{equation}
In the remainder of this appendix, we take the wave number $k_0=1$. For brevity, we denote 
\begin{equation}
\label{def:th}
    \th:=\tanh(h).
\end{equation}
Then, the resonance condition $\det M_{c_0}(1)=0$ reads $c_0^2-\gamma \th c_0-g\th =0$, whence  
\begin{equation}
\label{eq:gamma_by_other_quantities}
    \gamma=\frac{c_0}{\th}-\frac{g}{c_0}.
\end{equation}
 $\tanh(nh)$ can be expressed in terms of  $\th$:
\begin{lemm}
For  $h<\infty$ and $n\in\Nn$, we have
    \begin{equation}\label{tanhnh}
        \tanh(nh)=\frac{(1+\th)^n-(1-\th)^n}{(1+\th)^n+(1-\th)^n}.
    \end{equation}
\end{lemm}
\begin{proof}
  Using that 
   \[
    \th=\frac{e^{h}-e^{-h}}{e^{h}+e^{-h}},\quad 1+\th=\frac{2e^h}{e^h+e^{-h}},\quad 1-\th=\frac{2e^{-h}}{e^h+e^{-h}},
    \]
    we find
    \begin{align*}
        \frac{(1+\th)^n-(1-\th)^n}{(1+\th)^n+(1-\th)^n}
        =\frac{(\frac{2e^h}{e^h+e^{-h}})^n-(\frac{2e^{-h}}{e^h+e^{-h}})^n}{(\frac{2e^h}{e^h+e^{-h}})^n+(\frac{2e^{-h}}{e^h+e^{-h}})^n}
        =\frac{e^{nh}-e^{-nh}}{e^{nh}+e^{-nh}}
        =\tanh(nh).
    \end{align*}
\end{proof}
For $n=2$, \eqref{tanhnh} implies
\begin{align}
    \tanh(2h)=\frac{2\th}{1+\th^2}.
\end{align}
We substitute this into the coefficients $a$ and $b$ in \eqref{eq1_order2} and \eqref{eq2_order2}. Then, after straightforward calculations, we obtain the solution of  \eqref{sys:2k0} in terms of $\th$ as given in the $\eps^2$-coefficients below:
\begin{equation}
    \eta(x)=\eps\Big[\th \cos(x)\Big]+\frac{\eps^2}4\Big[(\frac{c_0^2}{g}-2\th+\frac{c_0^2\th ^2}{g})+\Big(\frac{1}{\th}+\frac{c_0^2 g+\th(g^2-2c_0^4)+c_0^2g\th ^2}{c_0^4}\Big)\cos(2x)\Big]+\rem3,
\end{equation}
\begin{equation}
    \psi(x)=\eps\Big[ c_0\sin(x)\Big]+\eps^2\Big[\frac{c_0^4+c_0^2g\th +(g^2-2c_0^4)\th^2+c_0^2 g\th^3+(g^2+c_0^4)\th^4}{8c_0^3\th^2}\sin(2x)\Big]+\rem3.
\end{equation}
 It remains to find the coefficient $c_2$. Since $c$ only appears in $c\eta_x$, $c\psi_x$, and since $\eta$, $\psi=O(\eps)$, finding $c_2$ requires the third-order expansions of the equations $F_j(\eta, \psi, c)=0$, $j=1, 2$. By inspecting the terms in $F_j$, we see that the $\eps^3$-coefficients $\eta_3$ and $\psi_3$ only appear  linearly in the terms $G[0]\psi_3$, $-g\eta_3$, and $\gamma \p_x^{-1}G[0]\psi_3$. Moreover, the $\eps^3$-terms $c_2\eta_1$ and $c_2\psi_1$ only have frequency $1$. Therefore, it suffices to  calculate the $\eps^3$-coefficient of $\F((F_j(\eta,\psi,c))(1)$  using the above second-order expansions of $\eta$, $\psi$, and the Fourier coefficients $z_1=\frac{\widehat{\eta_3}(1) }{2\pi}$  and $z_2=\frac{\widehat{\psi_3}(1) }{2\pi}$. Using Mathematica to carry out the calculation, we obtain 
\begin{align*}
   [\eps^3] \F(F_1(\eta,\psi,c))(1)&=\frac{\pi  \epsilon ^3 }{8  c_0 ^5 g \th}\Bigg(2 i  c_0 ^8 \th^5+2 i  c_0 ^8 \th^3+16 i  c_0 ^6 g z_1 \th-9 i  c_0 ^6 g \th^4+i  c_0 ^6 g \th^2-i  c_0 ^6 g\\
   &\quad +8 i  c_0 ^5 c_2 g \th^2+16  c_0 ^5 g z_2 \th^2-i  c_0 ^4 g^2 \th^5+6 i  c_0 ^4 g^2 \th^3\\
   &\quad -2 i  c_0 ^4 g^2 \th-i  c_0 ^2 g^3 \th^4-2 i  c_0 ^2 g^3 \th^2-i g^4 \th^3\Bigg),\\
   [\eps^3]\F(F_2(\eta,\psi,c))(1)&=\frac{\pi  \epsilon ^3 }{8  c_0 ^4 g \th^2}\Bigg(2  c_0 ^8 \th^5+2  c_0 ^8 \th^3-2  c_0 ^6 g \th^6-13  c_0 ^6 g \th^4+5  c_0 ^6 g \th^2\\
   &\quad - c_0 ^6 g+8  c_0 ^5 c_2 g \th^2-16  c_0 ^4 g^2 z_1 \th^2+5  c_0 ^4 g^2 \th^5+ c_0 ^4 g^2 \th^3- c_0 ^4 g^2 \th\\
   &\quad +16 i  c_0 ^3 g^2 z_2 \th^3+ c_0 ^2 g^3 \th^6+ c_0 ^2 g^3 \th^4- c_0 ^2 g^3 \th^2-g^4 \th^5\Bigg) .
\end{align*}
We observe that the coefficients of $z_1$ and $z_2$ in the above equations are proportional. Therefore, a linear combination of the two equations eliminates $z_1$ and $z_2$:
\begin{align*}
    0&=-ig[\eps^3]\F(F_1(\eta,\psi,c))(1)+c_0[\eps^3]\F(F_2(\eta,\psi,c))(1)\\
    &=-\frac{\pi  c_0^3}{8 \th^2}+\frac{\pi  \th \left(c_0^8+c_0^3 g^2 (c_0+4 c_2)-g^4\right)}{4 c_0^3 g}+\frac{\pi  \left(-11 c_0^8 g+7 c_0^4 g^3-g^5\right) \th^2}{8 c_0^5 g}\\
     &\quad +\frac{\pi  \left(c_0^8-2 c_0^4 g^2-g^4\right) \th^3}{4 c_0^3 g}+\frac{\pi  \left(c_0^7 g (5 c_0+8 c_2)-3 c_0^4 g^3\right)}{8 c_0^5 g}-\frac{\pi  c_0 g}{4 \th}.
\end{align*}
Solving this equation, we obtain 
\begin{equation}
\label{c_2}
    \begin{aligned}
            c_2&=\frac{1}{8 c_0^5 \th^2 g \left(c_0^2+g \th\right)}\Bigg(c_0 ^8 g+2 c_0 ^6 g^2 \th+\left(3 c_0 ^4 g^3-5 c_0 ^8 g\right) \th^2+\left(-2 c_0 ^{10}-2 c_0 ^6 g^2+2 c_0 ^2 g^4\right) \th^3\\
    &\quad +\left(11 c_0 ^8 g-7 c_0 ^4 g^3+g^5\right) \th^4+\left(-2 c_0 ^{10}+4 c_0 ^6 g^2+2 c_0 ^2 g^4\right) \th^5\Bigg),
    \end{aligned}
\end{equation}
thereby completing the proof of Theorem \ref{thm:Stokes_without_g} (upon setting $g=1$).
%%%%%%%

\subsection*{Acknowledgements}  HQN was partially supported by NSF CAREER Grant DMS-2541807.

%\ifoverleaf
%    \printbibliography
%\else 
%    \bibliographystyle{amsplain} 
%    \bibliography{references}
%\fi

\bibliographystyle{amsplain}
\bibliography{references}

\end{document}